\documentclass[11pt]{article}
\usepackage{graphicx}
\graphicspath{{figure/}}
\usepackage{geometry}
\usepackage{amsthm, amsmath, amssymb, bbm, bm, mathrsfs}
\usepackage[round]{natbib}
\usepackage[colorlinks, 
			linkcolor=blue,
			anchorcolor=blue,
			urlcolor=blue,
			citecolor=blue]{hyperref}

\usepackage{threeparttable}
\usepackage[title]{appendix}
\usepackage{url}
\usepackage{svg}
\usepackage{booktabs}
\usepackage{caption}
\usepackage{subfigure}
\usepackage{float}
\usepackage{listings}
\usepackage{multirow}
\usepackage{rotating}
\usepackage{longtable}
\usepackage{enumitem}
\usepackage{cases}
\usepackage{algorithm}
\usepackage{algorithmic}

\usepackage{xr}

\newtheorem{theorem}{Theorem}[section]

\newtheorem{proposition}{Proposition}[section]

\newtheorem{lemma}{Lemma}[section]
\newtheorem{definition}{Definition}[section]

\newtheoremstyle{mycase}{5pt}{5pt}{\upshape}{}{\bfseries}{.}{ }{} \theoremstyle{mycase}

\newtheoremstyle{myexample}{5pt}{5pt}{\upshape}{}{\bfseries}{.}{ }{} \theoremstyle{myexample}

\renewcommand{\theequation}{\thesection.\arabic{equation}}
\numberwithin{equation}{section}

\renewcommand{\hat}{\widehat}
\renewcommand{\tilde}{\widetilde}

\newcommand{\Var}{\mathrm{Var}}
\newcommand{\bX}{\mathbf{X}}
\newcommand{\bA}{\mathbf{A}}
\newcommand{\bB}{\mathbf{B}}

\newcommand{\bk}{\mathbf{k}}
\newcommand{\bb}{\mathbf{b}}
\newcommand{\bi}{\mathbf{i}}
\newcommand{\bj}{\mathbf{j}}
\newcommand{\DP}{{\rm{dp}}}
\newcommand{\dHSIC}{{\rm{dHSIC}}}
\newcommand{\dpdHSIC}{{\rm{dpdHSIC}}}
\newcommand{\Lap}{{\rm{Laplace}}}
\newcommand{\ham}{{\rm{ham}}}
\newcommand{\TV}{{\rm{TV}}}
\newcommand{\KL}{{\rm{KL}}}
\newcommand{\dd}{{\rm{d}}}

\newcommand{\bPi}{\bm{\Pi}}
\newcommand{\bpi}{\bm{\pi}}
\newcommand{\bphi}{\bm{\varphi}}
\newcommand{\bve}{\bm{\varepsilon}}
\newcommand{\blambda}{\bm{\lambda}}
\newcommand{\mbR}{\mathbb{R}}

\newcommand{\mbE}{\mathbb{E}}
\newcommand{\mbP}{\mathbb{P}}
\newcommand{\mbX}{\mathbb{X}}
\newcommand{\mcF}{\mathcal{F}}
\newcommand{\mcH}{\mathcal{H}}
\newcommand{\mcX}{\mathcal{X}}
\newcommand{\mcP}{\mathcal{P}}
\newcommand{\mcS}{\mathcal{S}}
\newcommand{\mcA}{\mathcal{A}}

\begin{document}
\allowdisplaybreaks[3] 
	
\title{\bf Differentially Private Joint Independence Test}
\author{{Xingwei Liu$^1$, 
Yuexin Chen$^1$, Jin-Ting Zhang$^2$,
Wangli Xu$^1$\thanks{Corresponding author: wlxu@ruc.edu.cn}}\\~\\
	{\small \it $^{1}$ Center for Applied Statistics and School of Statistics, Renmin University of China, Beijing, China}\\
    {\small \it $^{2}$ Department of Statistics and Data Science, National University of Singapore, Singapore}}

\date{}
\maketitle
	
\begin{abstract}
Identification of joint dependence among several random vectors plays an important role in many statistical applications, where the data may contain sensitive or confidential information. In this paper, we consider the $d$-variable Hilbert-Schmidt independence criterion (dHSIC) in the context of differential privacy. 
Given that the limiting distribution of the empirical estimate of dHSIC is a complicated Gaussian chaos, constructing tests in the non-private regime is typically based on permutation and bootstrap methods. 
To detect joint dependence under privacy constraints, we propose a dHSIC-based testing procedure employing a differentially private permutation methodology. 
We show that our method enjoys privacy guarantees, a valid level, and pointwise consistency, whereas the bootstrap counterpart suffers from inconsistent power. 
We further investigate the uniform power of the proposed test under the dHSIC and $L_2$ metrics, showing that the proposed test attains the minimax optimal power across different privacy regimes.  As a byproduct, we show that the non-private permutation dHSIC test proposed in \cite{Pfister2018KernelbasedTF} is a special case of our differentially private permutation test, and our results also establish its pointwise and uniform power--thus resolving an open problem from that work. Both numerical simulations and real data analysis in causal inference suggest that our proposed test performs well empirically.

\medskip
		
	
\end{abstract}

\section{Introduction}
Testing independence among variables is a fundamental problem in statistics, with applications spanning causal inference, feature screening, and gene detection \citep{Chatterjee2021ANC}. 
Two of the most classical nonparametric tests are based on Spearman's $\rho$ and Kendall's $\tau$, but these are not effective for detecting non-monotonic associations. 
Numerous proposals have emerged to address this limitation. 
Among these, kernel-based methods \citep{Gretton2005MeasuringSD} and distance correlation \citep{Szekely2007MeasuringAT,Szekely2009BrownianDC} have received considerable attention for both their theoretical appeal and simplicity of form. Although these approaches quantify the difference between the joint distributions of two random vectors and their marginals using different metrics, the equivalence between the two frameworks was established by \cite{Sejdinovic2013EquivalenceOD}. 
Other representative approaches include mutual information \citep{cover2005elements}, projection-based measures \citep{Zhu2017ProjectionCB,Chen2025AKI}, and rank-based correlations \citep{Bergsma2014ACT}. 
We refer the reader to \cite{Tjstheim2022StatisticalDB} for a comprehensive review.

Modern statistical methodologies increasingly require the exploration and identification of joint independence structures among $d > 2$ random variables or vectors. This need extends far beyond the conventional scope of pairwise independence testing, which is limited to the case $d = 2$.
In causal discovery, for example, inferring pairwise independence is generally insufficient for model diagnostic checking, particularly for models with an additive noise structure and their associated directed acyclic graphs (DAGs) \citep{Pfister2018KernelbasedTF}. 
Another example is independent component analysis, which seeks to represent multivariate data using jointly independent components \citep{Chen2006EfficientIC,Matteson2017IndependentCA}. 
However, joint independence testing is acknowledged as a more ambitious task than pairwise independence testing, and has therefore attracted significant attention and contributions from many scholars. 
As a generalization of kernel-based methods \citep{Gretton2005MeasuringSD}, \cite{Sejdinovic2013AKT} and \cite{Pfister2018KernelbasedTF} embed the joint distribution and the product of the marginals into a reproducing kernel Hilbert space (RKHS). They define the $d$-variable Hilbert-Schmidt independence criterion (dHSIC) as the squared distance among these embeddings. 
\cite{Chakraborty2019DistanceMF} generalized the notion of Brownian covariance \citep{Szekely2009BrownianDC} and defined joint dCov as a linear combination of pairwise dCov and their higher-order counterparts. 
Other recent developments include \cite{Bottcher2019DistanceMN,Drton2020HighdimensionalCI,Niu2022DistributionfreeJI}.

Joint independence testing is particularly relevant in areas such as economics, finance, and biology, where data are often sensitive or confidential. 
For example, model diagnostic checking for DAGs may involve datasets containing sensitive personal financial information and detailed body indices associated with rare diseases \citep{Chakraborty2019DistanceMF}. 
Analyzing such data therefore raises concerns about potential privacy leakage. 
Differential privacy, introduced by \cite{Dwork2006CalibratingNT}, has become the most widely adopted formal framework for privacy protection in both academia \citep{Dwork2014TheAF} and industry. 
Its broad adoption has motivated extensive work in statistics and related fields on differentially private (DP) mechanisms and variants \citep{Bun2016,Mironov2017,dong2022gaussian,Su2024ASV}. 
Across these developments, a central challenge is to balance privacy protection with statistical efficiency \citep{Cai2021TheCO}.

A line of work has focused on black-box approaches for privatizing existing hypothesis tests, most notably through the subsample-and-aggregate framework \citep{Nissim2007}. 
This strategy partitions the data into smaller subsets, applies a non-private test to each subset, and privately aggregates the resulting decisions or statistics. 
Recent work by \cite{Kazan2023TOT} and \cite{Pena2025DifferentiallyPH} has demonstrated the flexibility of this framework. However, such flexibility comes at the cost of statistical power \citep{Canonne2020}, especially in low-privacy regimes, a phenomenon also reflected in our numerical studies in Section \ref{sec_simulation}. 
To avoid this loss, another line of work has developed private tests for specific problems, including ANOVA \citep{Campbell2018DifferentiallyPA}, nonparametric tests \citep{Couch2019}, identity tests \citep{Cai2017}, conditional independence tests \citep{kalemaj2024}, and high-dimensional tests \citep{narayanan2022}. 
To the best of our knowledge, however, DP joint independence testing remains unexplored. The closest related direction is private independence testing, which has largely focused on $\chi^2$-type tests \citep{Gaboardi2016,Kakizaki2017}. More recently, \cite{Kim2026DPP} proposed a framework for private permutation tests and applied it to two-sample and independence testing, showing that their DP independence test is theoretically optimal and empirically effective. 
A natural approach is to reduce DP joint independence testing to a sequence of pairwise tests, as detailed in Algorithm~\ref{al_mul_dphsic}. This strategy, however, is asymmetric in the variables, loses power due to multiple testing, and requires the privacy budget to be split across several tests; see \cite{Pfister2018KernelbasedTF} and the discussion for Algorithm~\ref{al_mul_dphsic}.

In this paper, we propose a DP joint independence test by combining dHSIC \citep{Pfister2018KernelbasedTF} with private permutation calibration, and call the resulting method DP permutation dHSIC (dpdHSIC). 
It inherits the finite-sample validity of the permutation test in the non-private regime and further ensures differential privacy and uniform power guarantees. 
On the theoretical side, we establish minimax results for dpdHSIC across varying privacy regimes in terms of dHSIC and $L_2$ metrics. 
On the empirical side, the proposed test outperforms several competitors in comprehensive comparisons. 
Our contributions are summarized as follows. 
\begin{itemize} 
    \item \textbf{Permutation Theory for dHSIC.}
    \cite{Pfister2018KernelbasedTF} established the validity of the permutation dHSIC test but left its power consistency unresolved. We address this open problem by deriving concentration inequalities for permuted U- and V-statistics for dHSIC, where the main technical challenge lies in handling their high-order structures. These results establish pointwise consistency and uniform power for DP permutation dHSIC, including the non-private permutation dHSIC test as a special case. 
    
    
    \item \textbf{Failure of Bootstrap Under Privacy.}
    We illustrate that private bootstrap dHSIC differs fundamentally from private permutation dHSIC, even though both are implemented via the same private resampling procedure. The private permutation statistic has vanishing sensitivity and uniform power, whereas its bootstrap counterpart faces a non-vanishing sensitivity barrier and can be inconsistent even against fixed alternatives. This gap reveals that the permutation construction is essential for obtaining meaningful privacy-power guarantees in this problem. 

    \item \textbf{Uniform Power and Minimax Optimality.} 
    We characterize the minimum separation required by dpdHSIC under both the dHSIC metric and the $L_2$ metric. We further demonstrate that the proposed test attains minimax optimal power across different privacy regimes under the dHSIC metric. Our uniform power improves the polynomial factors in the error rates of the general result in \citet[Theorem~4]{Kim2026DPP} to logarithmic factors. We also analyze the U-statistic counterpart in Appendix~\ref{sec_u_dhsic}, which further illustrates the bias-privacy trade-off between the U- and V-statistics.  
\end{itemize}

The rest of the paper is organized as follows. Section~\ref{sec_dhsic_resampling} reviews the dHSIC statistic for joint independence testing and non-private resampling tests. 
Section~\ref{sec_dpdhsic} introduces the DP resampling test, applies it to permutation dHSIC, and shows that the bootstrap counterpart suffers from inconsistent power. 
Section~\ref{sec_power_minimax} investigates the uniform power and minimax optimality of the proposed permutation test under the dHSIC and $L_2$ metrics. 
Empirical performance based on numerical simulations and real data is reported in Sections~\ref{sec_simulation} and~\ref{sec_realdata}, respectively.
We conclude the paper and discuss future directions in Section~\ref{sec_discussion}. Preliminaries, additional results on the U-statistic, and all proofs and technical lemmas are included in the Appendix.

\emph{Notations.} 
Given two datasets $\mcX_n := (\bX_1,\dots,\bX_n)$ and $\tilde{\mcX}_n := (\tilde \bX_1,\dots,\tilde \bX_n)$, 
we denote the Hamming distance between $\mcX_n$ and $\tilde{\mcX}_n$ by $d_{\ham}({\mcX}_n, \tilde{\mcX}_n) := \sum_{i=1}^n \mathbbm{1}(\bX_i \neq \tilde \bX_i)$.  
$\|v\|_p$ denotes the $p$-norm of the vector $v \in \mathbb{R}^r$.
For a natural number $n \in \mathbb{N}$, we use $[n]$ to denote the set $\{1,\dots,n\}$. Let $\bi_p^n$ denote the set of all $p$-tuples drawn without replacement from $[n]$, and let $\bj_p^n$ denote the $p$-fold Cartesian product of $[n]$. Let $\mathbf{\Pi}_n$ denote the set of all permutations of $[n]$. 
For two sequences of real numbers $a_n$ and $b_n$, we write $a_n \lesssim b_n$ (resp. $a_n \gtrsim b_n$) if there exists a positive constant $C > 0$ independent of $n$ such that $a_n \leq C b_n$ (resp. $a_n \geq C b_n$) for all $n \geq 1$. 
We write $a_n \asymp b_n$ if $a_n \lesssim b_n$ and $a_n \gtrsim b_n$. 
Denote $\xi_{\epsilon, \delta} := \epsilon + \log({1}/{(1-\delta)})$ for simplicity. Throughout the paper, a superscript $j \in [d]$ always denotes an index rather than an exponent.

\section{Joint Independence Test via dHSIC}\label{sec_dhsic_resampling}

\subsection{Non-private Statistic}
Kernel-based methods have gained increasing popularity in variable independence testing due to their remarkable ability to capture various types of dependencies, including non-monotonic and nonlinear relationships, as well as their flexibility in implementation. 
We refer the reader to Appendix~\ref{sec_embed} for the basic concepts of RKHS used in this paper. 
In this section, we briefly review dHSIC as proposed in \cite{Pfister2018KernelbasedTF}. 

The goal of dHSIC is to test whether the components of a random vector $\bX = (X^1,\dots,X^d)$ defined on $\mbX^1 \times \dots \times \mbX^d$ are jointly independent, which holds if and only if the joint distribution equals the product of the marginals.
Formally, the hypothesis of interest can be formulated as 
\begin{equation}\label{eq_30}
    H_0: P_{X^1,\dots, X^d} = \otimes_{j=1}^d P_{X^j} \quad \text{versus} \quad H_1: P_{X^1,\dots, X^d} \neq \otimes_{j=1}^d P_{X^j}.
\end{equation}

The central idea of dHSIC is to embed $P_{X^1,\dots, X^d}$ and $\otimes_{j=1}^d P_{X^j}$ into a proper RKHS and compare their discrepancy. Assume that the tensor product $\bk := \otimes_{j=1}^d k^j$ is characteristic, where $k^j: \mbX^j \times \mbX^j \to \mbR$ is a continuous, bounded, positive semi-definite kernel, and $\otimes_{j=1}^d k^j((x^1, x'^{1}),\dots, (x^d, x'^{d})) = \prod_{j=1}^d k^j(x^j, x'^{j})$ for all $x^j, x'^{j} \in \mbX^j$ and $j \in [d]$. Let $\mcH^j$ be the RKHS associated with $k^j$ and $\mcH_{\bk} = \mcH^1 \otimes \dots \otimes \mcH^d$ be the projective tensor product of the RKHSs $\mcH^j$. Denoting $\mcF_{\bk}$ as the unit ball in $\mcH_{\bk}$, dHSIC among the $d$ groups is defined as 
\begin{equation}\label{df_dhsic}
    \dHSIC_{\bk}(P_{X^1,\dots, X^d}) := \sup_{f \in \mcF_{\bk}} \big\{
    \mathbb{E}_{P_{X^1,\dots, X^d}}[f(X^1,\dots, X^d)] - \mathbb{E}_{\otimes_{j=1}^d P_{X^j}}[f(X^1,\dots, X^d)]
    \big\}.
\end{equation}
Here we define dHSIC in MMD form \citep{Gretton2012AKT} to facilitate later derivations. 
Given an i.i.d.\ sample $\mcX_n := (\bX_1,\dots,\bX_n)$, an empirical estimator for \eqref{df_dhsic} is given by
\begin{equation}\label{dhsic_empirical}
    \widehat{\dHSIC}(\mcX_n) := \sup_{f \in \mcF_{\bk}} \bigg\{
    \frac{1}{n} \sum_{i=1}^n f(X^1_i,\dots, X^d_i)
    - \frac{1}{n^d} \sum_{\bj^n_d} f(X^1_{i_1},\dots, X^d_{i_d})
    \bigg\},
\end{equation}
where $\sum_{\bj^n_d}$ abbreviates summation over all $(i_1,\dots,i_d) \in \bj^n_d$. 
Squaring both sides in \eqref{dhsic_empirical} yields the same closed form of the V-statistic in \cite{Pfister2018KernelbasedTF} as
\begin{equation}\label{eq_001}
    \widehat{\dHSIC}^2(\mcX_n) = \frac{1}{n^2} \sum_{\bj_2^n} \prod_{j=1}^d
k^j(X_{i_1}^j, X_{i_2}^j)  
+ \frac{1}{n^{2d}} \sum_{\bj_{2d}^n} \prod_{j=1}^d
k^j(X_{i_{2j-1}}^j, X_{i_{2j}}^j) 
- \frac{2}{n^{d+1}} \sum_{\bj_{d+1}^n} \prod_{j=1}^d
k^j(X_{i_1}^j, X_{i_{j+1}}^j).
\end{equation}

\subsection{Resampling Tests}
 
By the classical theory of V-statistics \citep[Chapter~5]{Serfling1980ApproximationTO}, the limiting distribution of \eqref{eq_001} is a complicated Gaussian chaos. 
To conduct tests, \cite{Pfister2018KernelbasedTF} used a gamma approximation and  two resampling methods (permutation and bootstrap) to obtain the critical value. 
They demonstrated that the bootstrap method achieves pointwise asymptotic level and pointwise power, while the permutation test enjoys valid level but its power consistency has not be established. Moreover, the gamma approximation has been supported only empirically, without theoretical guarantees. 
To facilitate our proposal, we review the non-private resampling test in this subsection.

Consider the test of $H_0: P \in \mathcal{P}_0$ versus $H_1: P \in \mathcal{P}_1$ in \eqref{eq_30} based on random observations $\mcX_n = (\bX_1,\dots,\bX_n)$. 
For $\bphi = (\varphi^1,\dots,\varphi^d)$ satisfying $\varphi^j: [n] \to [n]$ for all $j \in [d]$, define the resampled data as $\mcX^{\bphi}_n := (\bX^{\bphi}_1,\dots,\bX^{\bphi}_n)$, where $\bX^{\bphi}_i = (X^1_{\varphi^1(i)},\dots,X^d_{\varphi^d(i)})$.  
In particular, $\bphi$ yields permutation resampling when each $\varphi^j$ is bijective, and bootstrap resampling when the maps $\varphi^j$ are arbitrary.

Let $\bphi_1,\dots,\bphi_B$ be i.i.d.\ random resamplings of $[n]$. 
We denote by $T(\mcX_n)$ the statistic $\widehat{\dHSIC}(\mcX_n)$ defined in \eqref{dhsic_empirical}, and the resampled statistics are $T(\mcX^{\bphi_1}_n),\dots,T(\mcX^{\bphi_B}_n)$. 
Then the resampling $p$-value is given by
\begin{equation}\label{eq_p}
    \hat{p} := \frac{1}{B+1} \bigg\{
    \sum_{i=1}^B \mathbbm{1}(T(\mcX^{\bphi_i}_n) \geq T(\mcX_n)) + 1
    \bigg\}.
\end{equation}
Suppose that the resampling method has a group structure and that $\mcX^{\bphi}_n$ is equal in distribution to $\mcX_n$ (i.e., \emph{exchangeable}) under the null. Then the resampling $p$-value in \eqref{eq_p} has valid level \citep[Chapter~17]{Lehmann2022TestingSH}. This guarantees the validity of permutation tests under exchangeability, whereas the bootstrap does not generally admit the same argument because bootstrap resampling does not possess a group structure \citep{Pfister2018KernelbasedTF}.

\section{Differentially Private dHSIC}\label{sec_dpdhsic}

Before proceeding, due to space constraints, we refer the reader to Appendix~\ref{sec_DP} for the fundamental concepts and results on differential privacy, including its definition, sensitivity, privacy composition, and the Laplace mechanism.

\subsection{Differentially Private Resampling Tests}

The main challenge of DP resampling tests is that each released resampled statistic incurs an additional privacy cost, so repeated resampling can multiply the privacy cost. 
Formally, suppose that we have obtained the sensitivity of $T$ or its upper bound $\Delta_T$. If we directly add noise to each $T(\mathcal{X}^{\boldsymbol{\phi}_i}_n)$ in \eqref{eq_p} to obtain a DP $p$-value, the privacy budget must be allocated across the $B+1$ statistics by Lemma~\ref{lemma_com}.  
This leads to a severe loss of power when $B$ is large, as the added noise may dominate the signal.

Fortunately, since we only need to return whether we reject the null at level $\alpha$ or not as the testing result, it is unnecessary to obtain a DP $\hat p$ for an $\alpha$-level test.  
In other words, it is sufficient to guarantee the privacy of the test function $\mathbbm{1}(\hat p \leq \alpha)$ rather than the resampling $p$-value $\hat p$. 
Consequently, the undesirable dependence on $B$ is removed by representing the test function via the quantile, where both the statistic $T(\mathcal{X}_n)$ and the quantile of $T(\mathcal{X}^{\boldsymbol{\phi}_i}_n)$ have sensitivity $\Delta_T$. This representation eliminates the factor $B$ and introduces a factor $2$.

For each $i \in \{0\} \cup [B]$, define the private resampling statistic
\begin{equation}\label{eq_22}
M_i := T(\mathcal{X}^{\boldsymbol{\phi}_i}_n) + \frac{2\Delta_T}{\xi_{\epsilon,\delta}} \zeta_i,
\end{equation}
where $\boldsymbol{\phi}_0$ denotes the identity map and $\zeta_i$ are i.i.d.\ standard Laplace random variables generated independently of $\mathcal{X}_n$. 
Given $\{M_i\}_{i=0}^B$, define the private resampling $p$-value as
\begin{equation}\label{dp_p}
    \hat{p}_{\text{DP}} := \frac{1}{B+1} \left\{ \sum_{i=1}^B \mathbbm{1}(M_i \geq M_0) + 1 \right\}.
\end{equation}
The null hypothesis is rejected at level $\alpha$ if and only if $\hat{p}_{\text{DP}} \leq \alpha$. The full procedure is summarized in Algorithm~\ref{al_dp_per}.

\begin{algorithm}
\caption{Differentially Private Resampling Test}\label{al_dp_per}
    \renewcommand{\algorithmicrequire}{\textbf{Input:}}
    \renewcommand{\algorithmicensure}{\textbf{Output:}}
    \begin{algorithmic}[1]
        \REQUIRE Data $\mcX_n$, level $\alpha \in (0,1)$, privacy parameters $\epsilon > 0$ and $\delta \in [0,1)$, test statistic $T$, sensitivity (or its upper bound) $\Delta_T$, resampling method $\bphi$ and number $B \in \mathbb{N}$.  
        \ENSURE Reject $H_0$ under $\alpha$-level or not.  
        
        \FOR{$i \in [B]$}
            \STATE Generate a random sample $\bphi_i$ of $[n]$.
            \STATE Generate $\zeta_i \sim \Lap(0,1)$.
            \STATE Set $M_i \gets T(\mcX^{\bphi_i}_n) + 2\Delta_T \xi_{\epsilon, \delta}^{-1} \zeta_i$, where $\xi_{\epsilon, \delta} := \epsilon + \log( 1/(1 - \delta) )$.
        \ENDFOR
        \STATE Generate $\zeta_0 \sim \text{Laplace}(0,1)$ and set $M_0 \gets T(\mcX_n) + 2\Delta_T \xi_{\epsilon, \delta}^{-1} \zeta_0$.
        \STATE Compute the permutation $p$-value $\hat{p}_{\DP}$ as in \eqref{dp_p}.
        
        \RETURN  $\mathbbm{1}(\hat{p}_{\DP} \leq \alpha)$. 
    \end{algorithmic}
\end{algorithm}

We next present theoretical guarantees for the DP resampling tests. 
Lemma~\ref{lemma_dp_re} presents privacy, level, and pointwise power guarantees for Algorithm~\ref{al_dp_per}, and also identifies a condition for inconsistent power. 

\begin{lemma}\label{lemma_dp_re}
For any $ \alpha \in (0,1) $ and $ B,n \geq 1 $, the resampling test $ \mathbbm{1}(\hat{p}_{\DP} \leq \alpha) $ from Algorithm~\ref{al_dp_per} has the following properties: \\
(i) (Privacy Guarantee) For any $\alpha \in (0,1)$, it is $(\epsilon, \delta)$-DP.\\
(ii) (Validity) Assume that $\mcX^{\bphi_i}_n$ has a group structure and $\mcX_n$ is exchangeable under $H_0$. Then the type I error satisfies
\begin{equation*}
    \sup_{P \in \mathcal{P}_0} \mathbb{P}_P({\hat{p}}_{\DP} \leq \alpha) = \frac{\lfloor (B + 1) \alpha \rfloor}{B + 1} \leq \alpha.
\end{equation*}\\
(iii) (Power) Suppose that $B_n$ is a positive sequence such that $\min_{n \geq 1} B_n > \alpha^{-1} - 1$ and $P$ is a given alternative. If $\lim_{n \to \infty} \mbP_P(M_0 \leq M_1) = 0$, then the DP resampling test is consistent in power as $\lim_{n \to \infty} \mathbb{P}_P(\hat{p}_{\DP} \leq \alpha) = 1$. If $\lim_{n \to \infty} \mbP_P(M_0 \leq M_1) > \alpha$, then the test is inconsistent against $P$, i.e., $\limsup_{n \to \infty} \mathbb{P}_P(\hat{p}_{\DP} \leq \alpha) <1$. 
\end{lemma}

As illustrated earlier, the indicator function $\mathbbm{1}(\hat p \leq \alpha)$ can be represented by the quantile, where both the statistic $T(\mathcal{X}^{\boldsymbol{\phi}_i}_n)$ and the quantile have sensitivity $\Delta_T$. Together with the factor $2$ in \eqref{eq_22}, the algorithm is $(\epsilon, \delta)$-DP.  

Regarding the testing level, methods such as permutation inherently possess an algebraic group structure, which ensures their validity non-asymptotically. 
For testing power, it suffices to show that $M_0 > M_1$ to establish pointwise consistency. In many cases, including our dpdHSIC, the noise strength approaches $0$ as $n \to \infty$. Thus, proving consistency under privacy constraints typically reduces to showing $T(\mathcal{X}_n) > T(\mathcal{X}_n^{\boldsymbol{\phi}})$ in the non-private regime.

\subsection{The Private Permutation dHSIC}

In this subsection, we focus on the DP joint independence testing based on dHSIC with permutation and Algorithm~\ref{al_dp_per}. 
To this end, we first derive the sensitivity of the permutation dHSIC. 

\begin{proposition}[Sensitivity of Empirical Permutation $\dHSIC$] \label{sen_dhsic}
    Assume that the kernels $k^j$  are bounded as $0\leq k^j(x,x')\leq K^j$ for all $x,x'\in \mbX^j$ and $j\in[d]$. Then the sensitivity of the empirical permutation dHSIC satisfies
    \begin{equation}\label{eq_sensi}
       \sup_{\bpi\in\bPi_n} \sup_ {\substack{\mcX_n, \widetilde{\mcX}_n:\\ d_{{\ham}}(\mcX, \widetilde{\mcX})\leq1}}
        |\widehat{\dHSIC}(\mcX^{\bpi}_n)-\widehat{\dHSIC}(\widetilde{\mcX}^{\bpi}_n)| \leq \frac{2d}{n}\bigg(\prod_{j=1}^dK^j\bigg)^{1/2}.
    \end{equation}
\end{proposition}

The bound assumption holds for a wide class of kernels including the popular Gaussian kernel $k(x,y) = e^{-\nu\|x-y\|^2_2}$ and the Laplacian kernel $k(x, y)=e^{-\nu\|x-y\|_1}$ for $\nu>0$, which are naturally bounded as $0\leq k(x,y)\leq 1$. 
Some kernels, such as the linear and polynomial kernels, are unbounded on their natural domains but may become bounded on restricted domains. For example, in single nucleotide polymorphism association studies, $x\in\{0,1,2\}$, so the linear kernel is bounded between $0$ and $4$.  

When $d=2$, our sensitivity bound differs from that of dpHSIC \citep{Kim2026DPP} only up to $O(n^{-2})$ terms. Although the bound can be slightly sharpened through case-specific refinements, the improvement is negligible, so we use the simpler form throughout. We also show in Appendix~\ref{app_pr_sen} that, when the kernels $k^j$ are translation invariant and have non-empty level sets on $\mbX^j$, the order of the sensitivity cannot be improved, i.e., it is bounded below by a constant multiple of $1/n$. 


Set $\Delta_T=2d\sqrt{K_0}/n$ with $K_0:=\prod_{j=1}^dK^j$ and apply Algorithm~\ref{al_dp_per} with the empirical dHSIC in \eqref{dhsic_empirical} as the test statistic. We refer to the resulting test as the dpdHSIC test and denote its output by $\phi_{\dpdHSIC}$. Its privacy guarantees and level control follow immediately from Lemma~\ref{lemma_dp_re}, while pointwise power requires additional efforts. 
We formulate the three points as Theorem~\ref{property_dpdH}. 

\begin{theorem}[Properties of dpdHSIC Test]\label{property_dpdH}
Let $\alpha\in(0,1)$ be a fixed constant and assume that the kernels $k^j$  are bounded as $0\leq k^j(x,x')\leq K^j$ for all $x,x'\in \mbX^j$ and $j\in[d]$. Then $\phi_{\dpdHSIC}$ satisfies \\
(i) (Privacy Guarantee) For $\epsilon > 0$ and $\delta \in [0, 1)$, $\phi_{\dpdHSIC}$ is $(\epsilon, \delta)$-DP.\\
(ii) (Validity) The type I error of $\phi_{\dpdHSIC}$ is controlled at level $\alpha$ non-asymptotically. \\
(iii) (Pointwise Consistency) Assume that $n^{-1} \xi_{\epsilon, \delta}^{-1} \to 0$ as $n \to \infty$. Then for any sequence $B_n$ such that $\min_{n \geq 1} B_n > \alpha^{-1} - 1$, we have
$\lim_{n \to \infty} \mbE_{P_{X^1,\dots, X^d}}[\phi_{\dpdHSIC}] = 1$ under a fixed alternative. 
\end{theorem}

We briefly outline the proof of the pointwise consistency result and highlight the main difficulty. By definition of dpdHSIC, the added noise vanishes in probability whenever $n^{-1}\xi_{\epsilon,\delta}^{-1}\to0$. Hence, consistency of the private test essentially reduces to showing that the non-private permutation dHSIC test is pointwise consistent, namely that $\hat\dHSIC(\mcX_n)$ dominates $\hat\dHSIC(\mcX_n^{\bpi})$ under fixed alternatives.
This step is non-trivial and was left open in \citet[Remark~2]{Pfister2018KernelbasedTF}. 
To address it, Lemma~\ref{le_ineq_e_dhsic} in Appendix~\ref{app_le_dhsic} establishes an exponential inequality showing that the empirical dHSIC converges to its population counterpart. 
It remains to show that the permuted statistic converges to zero. A related argument for dpHSIC was developed by \citet{Kim2026DPP}, relying on a concentration inequality for a second-order permuted U-statistic \citep[Theorem~6.2]{Kim2022MinimaxOO}. 
The dHSIC statistic is instead a high-order V-statistic, so this second-order U-statistic inequality cannot be applied directly. 
One of our theoretical contributions is to develop techniques for controlling high-order permuted U-statistics and transferring the resulting bounds to their V-statistic counterparts. 
These techniques yield concentration inequalities for permuted dHSIC, stated as Lemmas~\ref{le_u_dhsic_con_per} and~\ref{le_dhsic_con_per} in Appendix~\ref{app_le_dhsic}. 

\subsection{Inconsistent Power of Bootstrap dHSIC}
Let $\bb$ denote a bootstrap resampling map. 
Although the non-private bootstrap dHSIC test has pointwise asymptotic validity and consistency \citep{Pfister2018KernelbasedTF}, its DP counterpart suffers from inconsistent power even for fixed alternatives. 
To see this, we first show that the sensitivity of bootstrap dHSIC does not vanish with $n$, as stated in the following proposition.

\begin{proposition}[Sensitivity of Empirical Bootstrap dHSIC]\label{prop_bootstrap_sen}
Consider $d\geq2$ and assume that the kernels $k^j$ are translation invariant and have non-empty level sets on $\mbX^j$ for $j\in[d]$. The sensitivity of bootstrap dHSIC satisfies
\begin{equation*}
    \Delta_T^{\bb}:=\sup_{\bb}\sup_{\substack{\mcX_n,\tilde{\mcX}_n:\\ d_{\ham}(\mcX_n,\tilde{\mcX}_n)\leq1}}
    |\hat\dHSIC(\mcX_n^{\bb})-\hat\dHSIC(\tilde{\mcX}_n^{\bb})|
    \geq \frac{3}{8}\sqrt{K_0},
\end{equation*}
where $\sup_{\bb}$ is taken over all bootstrap resamplings.
\end{proposition}

Proposition~\ref{prop_bootstrap_sen} shows that the sensitivity of bootstrap dHSIC is of constant order, in contrast to the $O(n^{-1})$ sensitivity of the permutation dHSIC in Proposition~\ref{sen_dhsic}. 
The DP bootstrap dHSIC test is based on Algorithm~\ref{al_dp_per} with sensitivity $\Delta_T^{\bb}$. 
Consequently, whenever $\xi_{\epsilon,\delta}$ is fixed, the privacy noise in \eqref{eq_22} remains of constant order and may dominate the signal even under fixed alternatives. 
This leads to the following power inconsistency result. 

\begin{proposition}[Power Inconsistency of DP Bootstrap dHSIC]\label{prop_bootstrap_inconsistency}
    Suppose that $\xi_{\epsilon,\delta}$ is fixed and $B_n$ satisfies $\min_{n\geq1}B_n>\alpha^{-1}-1$. Under the assumptions of Proposition~\ref{prop_bootstrap_sen} and $\alpha<1/2$, the DP bootstrap dHSIC test has inconsistent power against some fixed alternatives. Moreover, the DP bootstrap dHSIC test has inconsistent power against all alternatives when $\alpha<C_{\epsilon,\delta}$, where $C_{\epsilon,\delta}$ is some constant that depends only on $\epsilon$ and $\delta$. 
\end{proposition}

Although bootstrap dHSIC is effective in non-private settings, its DP version is compromised by a non-vanishing sensitivity and the consequent unavoidable noise. This limitation is particular to bootstrap dHSIC and does not rule out the use of bootstrap methods for other kernel-based tests. For example, bootstrap MMD can attain the same order of sensitivity as permutation MMD \citep[Lemma~2]{Kim2026DPP}.

\section{Uniform Power and Minimax Optimality}\label{sec_power_minimax}

In the previous section, we established the asymptotic power of the DP permutation test $\dpdHSIC$ against fixed alternatives. It is also important to obtain uniform power guarantees over sequences of alternatives. In this section, we study uniform power under the dHSIC metric induced by the kernel $\bk$ and under the $L_2$ norm for smooth alternatives. 

\subsection{Uniform Power in dHSIC Metric}\label{sec_powerdhsic}

We first characterize the minimum separation required for $\phi_{\dpdHSIC}$ in the dHSIC metric. Let $\mcP_{\mbX^1\times\dots\times\mbX^d}$ denote the class of distributions on $\mbX^1\times\dots\times\mbX^d$. For $\rho>0$, define the alternative class
$$
    \mcP_{\dHSIC_{\bk}}(\rho) := \bigg\{P_{X^1,\dots, X^d} \in \mcP_{\mbX^1\times\dots\times\mbX^d}: \dHSIC_{\bk}(P_{X^1,\dots, X^d}) \geq \rho \bigg\}.
$$
For a target type II error level $\beta\in(0,1-\alpha)$, the minimum separation of the dpdHSIC test against $\mcP_{\dHSIC_{\bk}}(\rho)$ is defined as
\begin{equation}
    \rho_{\phi_{\dpdHSIC}}(\alpha, \beta, \epsilon, \delta, n) := \inf \bigg\{ \rho > 0 : \sup_{P \in \mcP_{\dHSIC_{\bk}}(\rho)} \mbE_{P} [ 1 - \phi_{\dpdHSIC} ] \leq \beta \bigg\}.
\end{equation}

The quantity $\rho_{\phi_{\dpdHSIC}}$ is the smallest dHSIC signal strength for which the dpdHSIC test detects any alternative in the class with probability at least $1-\beta$. In other words, the test has power at least $1-\beta$ uniformly over all alternatives whose dHSIC signal exceeds this threshold. The following theorem provides an upper bound for $\rho_{\phi_{\dpdHSIC}}$ as a function of the parameters $\alpha$, $\beta$, $\epsilon$, $\delta$ and $n$. 

\begin{theorem}[Minimum Separation of $\phi_{\dpdHSIC}$]
\label{th_mini_sepa_dpdhsic}
    Assume that the kernels $k^j$ are bounded as $0\leq k^j(x,x')\leq K^j$ for all $x,x'\in \mbX^j$ and $j\in[d]$.
    For any $\alpha \in (0,1)$, $\beta \in (0, 1-\alpha)$, $\epsilon > 0$, $\delta \in [0, 1]$ and $B \geq 6\alpha^{-1} \log(2\beta^{-1})$, the minimum separation for $\phi_{\dpdHSIC}$ satisfies
$$
\rho_{\phi_{\dpdHSIC}}(\alpha, \beta, \epsilon, \delta, n) \leq C_{K} \max \bigg\{ \sqrt{\frac{\max \{ \log(1/\alpha), \log(1/\beta) \}}{n}}, \frac{\max \{ \log(1/\alpha), \log(1/\beta) \}}{n \xi_{\epsilon,\delta}} \bigg\},
$$
where $C_{K}$ is a positive constant that only depends on $K^j$ for $j\in[d]$.
\end{theorem}

To simplify the discussion, suppose that $\alpha$ and $\beta$ are fixed. The bound in Theorem~\ref{th_mini_sepa_dpdhsic} clarifies how privacy affects the uniform detection boundary. As $\xi_{\epsilon,\delta}\to\infty$, $\phi_{\dpdHSIC}$ reduces to the non-private permutation dHSIC test \citep{Pfister2018KernelbasedTF}, and the bound recovers the non-private rate $\sqrt{{\max \{ \log(1/\alpha), \log(1/\beta) \}}/{n}}$. 
This also gives a uniform power guarantee for the non-private permutation dHSIC test, addressing the open question raised in \citet[Remark~2]{Pfister2018KernelbasedTF}. 

More generally, the leading term is determined by the privacy level. 
In the low-privacy regime $\xi_{\epsilon,\delta}\gtrsim n^{-1/2}$, the rate remains $n^{-1/2}$, so the added privacy noise has no first-order effect on the separation rate. In the high-privacy regime $\xi_{\epsilon,\delta}\lesssim n^{-1/2}$, the rate becomes $n^{-1}\xi_{\epsilon,\delta}^{-1}$ and is governed by the privacy noise. This bound is meaningful only when $n^{-1}\xi_{\epsilon,\delta}^{-1}\to0$, in line with condition (iii) of Theorem~\ref{property_dpdH}. 

Compared with the general uniform-power result of \citet[Theorem~4]{Kim2026DPP}, restated in Lemma~\ref{lemma_uniformpower}, our argument uses a concentration inequality for the empirical dHSIC rather than Chebyshev's and Markov's inequalities. Consequently, in the low-privacy regime, the polynomial dependence on $\alpha$ and $\beta$ can be improved to a logarithmic dependence. Theorem~\ref{th_minimax_sepa} further shows that this logarithmic dependence cannot, in general, be improved when $\alpha\asymp\beta$ in the low-privacy regime.

We next investigate the minimax optimality of $\phi_{\dpdHSIC}$ when $\mbX^j=\mbR^{p_j}$ for $j\in[d]$. To this end, we establish a lower bound for the minimax separation measured by the dHSIC metric. Let $\phi:\mcX_n\to\{0,1\}$ be a test function, and let $\mathcal{P}_0 := \bigg\{P_{X^1,\dots,X^d}\in\mcP_{\mathbb{R}^{p_1}\times\dots\times\mathbb{R}^{p_d}}: P_{X^1,\dots,X^d}=\otimes_{j=1}^d P_{X^j}\bigg\}$ denote the class of null distributions on $\mathbb{R}^{p_1}\times\dots\times\mathbb{R}^{p_d}$. The class of $(\epsilon,\delta)$-DP tests with level at most $\alpha$ is defined by
$$
\Phi_{\alpha, \epsilon, \delta} := \bigg\{ \phi : \sup_{P \in \mathcal{P}_0} \mathbb{E}_{P}[\phi] \leq \alpha \text{ and } \phi \text{ is } (\epsilon, \delta)\text{-DP} \bigg\}. 
$$

It is of theoretical interest to develop the lower bound on minimum separation for any valid and $(\epsilon, \delta)$-DP test, commonly referred to as minimax separation in the literature. The minimax separation in terms of dHSIC metric is defined as
$$
\rho^*_{\dHSIC}(\alpha, \beta, \epsilon, \delta, n) := \inf \bigg\{ \rho > 0 : \inf_{\phi \in \Phi_{\alpha, \epsilon, \delta}} \sup_{P \in \mathcal{P}_{{\dHSIC}_{\bk}}(\rho)} \mathbb{E}_{P}[1 - \phi] \leq \beta \bigg\}.
$$
This quantity characterizes the smallest dHSIC signal strength above which some valid $(\epsilon,\delta)$-DP $\alpha$-level test can achieve power at least $1-\beta$ uniformly over the alternative class. Equivalently, signals below this separation cannot be uniformly detected with the desired power by any valid $(\epsilon,\delta)$-DP test. The next theorem establishes a lower bound for $\rho^*_{\dHSIC}$ under DP constraints.
\begin{theorem}[Minimax Separation in dHSIC]\label{th_minimax_sepa}
Let $\alpha$ and $\beta$ be positive numbers satisfying $\alpha+\beta<c$ for any fixed $c<1/2$, $\epsilon > 0$ and $\delta \in [0, 1)$. Assume that the kernel functions $k^j$ are translation invariant on $\mathbb{R}^{p_j}$ for $j\in[d]$. In particular, there exist some functions $\kappa^j$ such that $k^j(x,x') = \kappa^j(x - x')$ for all $x, x' \in \mathbb{R}^{p_j}$. Moreover, assume that there exist positive constants $\eta^j$ such that $\kappa^j(0) - \kappa^j(x_0) \geq \eta^j$ for some $x_0^j \in \mathbb{R}^{p_j}$. Then the minimax separation over $\mathcal{P}_{\dHSIC_{\bk}}(\rho)$ is lower bounded as
$$
\rho^*_{\dHSIC}(\alpha, \beta, \epsilon, \delta, n) \geq C_{\eta} \max \bigg\{ \min \bigg\{ \sqrt{\frac{\log(1/(\alpha + \beta))}{n}}, 1 \bigg\}, \min\bigg\{\frac{1}{n\xi_{\epsilon,\delta}},1\bigg\} \bigg\},
$$
where $C_{\eta}$ is a positive constant that only depends on $\eta^j$ for $j\in[d]$.
\end{theorem}

We present some comments as follows. First, the assumptions in Theorem~\ref{th_minimax_sepa} are mild. The restriction $\alpha+\beta<c$ is natural in testing problems, since the practically relevant setting usually takes both error probabilities to be small. Moreover, many standard kernels, including Gaussian, Laplacian, and inverse multiquadric kernels, are translation invariant. If $k^j$ is further characteristic, then the separation condition involving $\eta^j$ is automatically satisfied for a suitable $x_0^j$.

Second, when $\alpha$ and $\beta$ are fixed, our proposed dpdHSIC achieves the minimax rate optimality against the class of alternatives determined by the dHSIC metric in $\mathbb{R}^{p_1} \times\dots\times \mathbb{R}^{p_d}$. Indeed, comparing Theorem~\ref{th_mini_sepa_dpdhsic} with Theorem~\ref{th_minimax_sepa} shows that the upper bound for the minimum separation of $\phi_{\dpdHSIC}$ matches the minimax lower bound $\rho^*_{\dHSIC}$. This conclusion holds for all privacy regimes, since no additional restriction is imposed on $\epsilon$ and $\delta$. 

Third, the lower bound also covers the case where $\alpha$ and $\beta$ vary with $n$. When $\alpha\asymp\beta$, the efficiency of the proposed test can be summarized by
$$
\frac{\rho_{\phi_{\dpdHSIC}}}{\rho^*_{\dHSIC}}\lesssim
\begin{cases}
1, & \text{if } \log^{1/2}(1/\alpha)n^{-1/2} \lesssim \xi_{\epsilon,\delta}, \\
\sqrt{\log(1/\alpha)}n^{-1/2}\xi_{\epsilon,\delta}^{-1}, & \text{if } \log^{-1/2}(1/\alpha)n^{-1/2} \lesssim \xi_{\epsilon,\delta} \lesssim \log^{1/2}(1/\alpha)n^{-1/2}, \\
\log(1/\alpha), & \text{if } \xi_{\epsilon,\delta} \lesssim \log^{-1/2}(1/\alpha)n^{-1/2}.
\end{cases}
$$
Thus, even when $\alpha$ and $\beta$ decrease to zero, dpdHSIC remains minimax optimal up to at most logarithmic factors. 

The minimax analysis in Theorem~\ref{th_minimax_sepa} also reveals the intrinsic trade-off between privacy and statistical power: stronger privacy increases the minimum separation rate for uniform detection, whereas weak privacy does not change the order of separation rate.

\subsection{Uniform Power in \texorpdfstring{$L_2$}{L2} Metric}\label{sec_powerL2}
We next study the minimum separation of the dpdHSIC test in the $L_2$ metric. This result follows the framework of \cite{Albert2022AdaptiveTO}, with additional care devoted to the bias induced by replacing the associated U-statistic with the V-statistic based on the $(2d)$th-order kernel. We restrict attention to a smooth class of alternatives described by Sobolev balls. For a smoothness parameter $s>0$ and radius $R>0$, define
\begin{equation*}
    \mcS_d^s(R) := \bigg\{ f \in L_1(\mathbb{R}^d) \cap L_2(\mathbb{R}^d):\int_{\mathbb{R}^d} \|w\|_2^{2s} |\hat{f}(w)|^2 \dd w \leq (2\pi)^d R^2 \bigg\},
\end{equation*}
where $\hat f$ is the Fourier transform of $f$, defined on $\mbR^d$ by $\hat f(w)=\int_{\mbR^d}f(x) e^{-ix^Tw}\dd x$. The condition $f \in L_1(\mathbb{R}^d) \cap L_2(\mathbb{R}^d)$ restricts the function to be both integrable and square-integrable with respect to the Lebesgue measure on $\mbR^d$. 

For $\rho>0$, $\mcP_{L_2}(\rho)$ specifies the collection of distributions $P_{X^1,\dots, X^d}$ on $\mbR^{p_1}\times\dots\times\mbR^{p_d}$, where $P_{X^1,\dots, X^d}$ is equipped with the Lebesgue density function $p_{X^1,\dots, X^d}$ and the product of the marginals $p_{X^1}\dots p_{X^d}$ such that $\|p_{X^1,\dots, X^d}-p_{X^1}\dots p_{X^d}\|_{L_2}\geq\rho$. Here $\|\cdot\|_{L_2}$ refers to the $L_2$ norm on $\mbR^d$ with the Lebesgue measure. 
We focus on the smooth and uniformly bounded subclass
\begin{align*}
    {\mcP}^s_{L_2}(\rho) := \bigg\{ P_{X^1,\dots, X^d} \in {\mcP}_{L_2}(\rho) : p_{X^1,\dots, X^d} - p_{X^1}\dots p_{X^d} \in \mcS_{\sum p_j}^s(R), \\
    \max(\|p_{X^1,\dots, X^d}\|_{L_\infty}, \|p_{X^1}\dots p_{X^d}\|_{L_\infty}) \leq M \bigg\},
\end{align*}
Our aim is to derive the minimum value of $\rho$ such that the dpdHSIC test has desired uniform power over ${\mcP}^s_{L_2}(\rho)$. 

Following \cite{Albert2022AdaptiveTO}, we focus on Gaussian kernels. For $x^j,x'^j\in\mbR^{p_j}$, let $\blambda_j=(\lambda_{j,1},\dots,\lambda_{j,p_j})\in(0,\infty)^{p_j}$ and set
$$
k^j(x^j,x'^j)=\prod_{i=1}^{p_j}\frac{1}{\sqrt{2\pi}\lambda_{j,i}}\exp\bigg\{-\frac{(x_i^j-x_i'^j)^2}{2\lambda_{j,i}^2}\bigg\}.
$$
The bandwidths $\blambda_j$ may depend on $n$, and consequently $K_0$ also varies throughout this subsection. The minimum separation of the dpdHSIC test over $L_2$ alternatives is
$$
\rho_{\phi_{\dpdHSIC},L_2}(\alpha, \beta, \epsilon, \delta, n, s,R,M, p_1,\dots,p_d) := \inf \bigg\{ \rho > 0 : \sup_{P \in  {\mcP}^s_{L_2}(\rho)} \mbE_{P} [ 1 - \phi_{\dpdHSIC} ] \leq \beta \bigg\}.
$$
Technically speaking, the proof proceeds by lower-bounding the corresponding U-statistic uniformly over the alternative class for sufficiently large $\rho$, and then controls the gap between the U- and V-statistics. We denote $\sum_{j=1}^dp_j=p$, $\prod_{i=1}^{p_j}\lambda_{j,i}$ by $\lambda_j$ for $j\in[d]$ and $\prod_{j=1}^d\lambda_j$ as $\lambda_0$ for simplicity. 
The next theorem presents an upper bound for $\rho_{\phi_{\dpdHSIC},L_2}$ with a set of parameters including $\lambda_{j,i}$ and $n$.
\begin{theorem}[Minimum Separation of $\phi_{\dpdHSIC}$ over ${\mcP}^s_{L_2}$]
\label{th_L2separation}
    Assume that $\alpha \in (0,1)$, $\beta\in(0, 1-\alpha)$ are fixed, and $\epsilon>0$, $\delta \in [0, 1)$, 
    $B \geq 6\alpha^{-1} \log(2\beta^{-1})$, $\lambda_j \leq 1$ for $j\in [d]$. The minimum separation of the dpdHSIC with the Gaussian kernels over ${\mcP}^s_{L_2}$ is upper bounded as
\begin{align*}
    \rho^2_{\phi_{\dpdHSIC},L_2} \leq C_{\alpha, \beta, s,R,M, p_1,\dots,p_d} \bigg\{ \sum_{j=1}^d\sum_{i=1}^{p_j} \lambda_{j,i}^{2s} + \frac{1}{n \sqrt{\lambda_0}}+ \frac{1}{n^{3/2} \lambda_0 \xi_{\epsilon, \delta}} + \frac{1}{n^2 \lambda_0 \xi^2_{\epsilon, \delta}}  \\
    +\sum_{t=1}^{d-2}\sum_{\bi^d_t}\frac{1}{n^t\lambda_{i_1}\dots\lambda_{i_t}} 
    \bigg\}
\end{align*}
where $ C_{\alpha, \beta, s,R,M, p_1,\dots,p_d} $ is some positive constant, depending only on $\alpha, \beta, s,R,M, p_1,\dots,p_d$.
\end{theorem}

\noindent Several comments are in order. For convenience, write
$$
\text{(I)}=\frac{1}{n\sqrt{\lambda_0}},\quad
\text{(II)}=\frac{1}{n^{3/2}\lambda_0\xi_{\epsilon,\delta}},\quad
\text{(III)}=\frac{1}{n^2\lambda_0\xi_{\epsilon,\delta}^2},\quad
\text{(IV)}=\sum_{t=1}^{d-2}\sum_{\bi_t^d}\frac{1}{n^t\lambda_{i_1}\cdots\lambda_{i_t}}.
$$

The bandwidth condition in Theorem~\ref{th_L2separation} is mild. In nonparametric testing, the bandwidths usually shrink to zero as $n\to\infty$ \citep{Albert2022AdaptiveTO}, while the theorem only requires $\lambda_j=\prod_{i=1}^{p_j}\lambda_{j,i}\lesssim1$. The additional term (IV) arises from the discrepancy between the U- and V-statistics when $d>2$. This term is specific to the high-order structure of dHSIC, since both statistics are built from a $(2d)$th-order kernel. 

The leading term depends on the privacy regime. Suppose first that $2p_j\leq p$ for $j\in[d]$, so the vector dimensions are relatively balanced. In this case, term (IV) is absorbed by term (I) in the low-privacy regime, while the terms with $t=1$ and $t\geq2$ in term (IV) are dominated by terms (I) and (III), respectively, in the mid- and high-privacy regimes. 
Thus, under balanced dimensions, term (IV) does not affect the resulting rate. 
In the low-privacy regime, where term (I) is dominant, choosing $\lambda_{j,i}=n^{-2/(4s+p)}$ yields the separation rate $n^{-2s/(4s+p)}$ over the Sobolev ball, which also gives the non-private dHSIC rate as a byproduct. 
In the mid-privacy regime, where term (II) dominates, the bandwidth choice $\lambda_{j,i}=n^{-3/(4s+2p)}\xi_{\epsilon,\delta}^{-1/(2s+p)}$ leads to the rate $n^{-3s/(2s+p)}\xi_{\epsilon,\delta}^{-s/(2s+p)}$. 
In the high-privacy regime, where term (III) dominates, taking $\lambda_{j,i}=(n\xi_{\epsilon,\delta})^{-2/(2s+p)}$ leads to the rate $(n\xi_{\epsilon,\delta})^{-2s/(2s+p)}$. 

These choices can be summarized as
 $$
    \rho_{\phi_{\dpdHSIC},L_2}\lesssim
    \begin{cases}
    n^{-\frac{2s}{4s+p}}, & \text{if } n^{-\frac{2s-p/2}{4s+p}} \lesssim \xi_{\epsilon, \delta} \text{ (low privacy)}, \\
    (n^{\frac{3}{2}}\xi_{\epsilon, \delta})^{-\frac{s}{2s+p}}, & \text{if } n^{-\frac{1}{2}} \lesssim \xi_{\epsilon, \delta} \lesssim n^{-\frac{2s - p/2}{4s+p}} \text{ (mid privacy)}, \\
    (n\xi_{\epsilon, \delta})^{-\frac{2s}{2s+p}}, & \text{if } \xi_{\epsilon, \delta} \lesssim n^{-\frac{1}{2}} \text{ (high privacy)}.
    \end{cases}
    $$
Therefore, the proposed test attains the classical Sobolev testing rate $n^{-2s/(4s+p)}$ in the low-privacy regime. When $\xi_{\epsilon,\delta}\asymp n^{-1/2}$, the rate becomes $n^{-s/(2s+p)}$, matching the minimax rate for density estimation under the $L_2$ loss. 

If $2p_j>p$ for some $j\in[d]$, the V-statistic correction term (IV) may affect the separation rate. In particular, the contribution of term (IV) becomes relevant when $\xi_{\epsilon,\delta}\gtrsim n^{(2p-3p_j-2s)/(2s+p_j)}$, whereas it is negligible outside this range. 
Since this critical value lies in the mid-privacy regime, the interaction between noise and V-statistic bias can be non-monotone in unbalanced dimensions. 
As $\xi_{\epsilon,\delta}$ crosses this threshold, the leading contribution can switch between the V-statistic correction bias and the privacy noise. 
Thus, the resulting rate is governed by which of these two effects is larger, rather than by the privacy noise alone.

We also consider the DP permutation test based on the U-statistic version of dHSIC and derive its minimum separations in both the dHSIC and $L_2$ metrics. The proof follows the same general strategy as that for the V-statistic, but is technically simpler. The resulting minimax rates are close to those in \cite{Kim2026DPP} and confirm the limitation of the U-statistic approach: because the statistic cannot be represented in a squared form, its sensitivity is larger, leading to a less favorable privacy-power trade-off. We defer the detailed statements and proofs to Appendix~\ref{sec_u_dhsic}.

\section{Simulations}\label{sec_simulation}
\subsection{Competing Methods}
For a comprehensive comparison, we consider three competing methods: (i) multiple testing with DP HSIC (MdpHSIC) proposed by \cite{Kim2026DPP}, (ii) dHSIC based on the test of tests (TOT dHSIC) proposed by \cite{Kazan2023TOT}, and (iii) dHSIC based on the subsampled and aggregated randomized response mechanism (SAR dHSIC) proposed by \cite{Pena2025DifferentiallyPH}. We first describe the procedure for method (i), and then briefly review the two general methods in (ii) and (iii). 

\textbf{(i) MdpHSIC.} 
To conduct DP joint independence testing, we consider a multiple pairwise version of the two-variable dpHSIC test. Joint independence holds if and only if $X^j$ is independent of $(X^1, \dots, X^{j-1})$ for all $j \in \{2, \dots, d\}$. This characterization reduces the joint independence problem to a sequence of $d-1$ pairwise independence tests. We run dpHSIC $d-1$ times and combine the results using a Bonferroni correction. Since the procedure privatizes several pairwise tests rather than one direct joint-independence statistic, the privacy budget is split across these tests and additional noise is introduced. For example, setting $K^j=1$ for all $j\in[d]$, the standard error of the total noise in MdpHSIC is about $2d$ times that in dpdHSIC. The procedure is summarized in Algorithm~\ref{al_mul_dphsic}. By Lemma~\ref{lemma_com} and the Bonferroni correction, the algorithm is $(\epsilon,\delta)$-DP and valid.

\textbf{(ii) TOT dHSIC.} TOT is based on the subsample-and-aggregate principle \citep{Nissim2007} and provides a black-box way to privatize a hypothesis test. TOT dHSIC divides the data into $m$ subsets and computes public $p$-values using the permutation dHSIC test based on \eqref{eq_p}.
Under the null, these $p$-values are uniformly distributed on $[0,1]$. Hence the number of $p$-values below the sub-test level $\alpha_0$ follows a binomial distribution with parameters $(m,\alpha_0)$. 
The optimal DP test for binomial data \citep{Awan2018} is then applied to this count to obtain the final DP test statistic. 
This method guarantees differential privacy and non-asymptotic level, but it is sensitive to the choice of $(m,\alpha_0)$. Following \cite{Kim2026DPP}, we set $m=\sqrt{n}$ and $\alpha_0=5\alpha$ throughout the paper. 

\textbf{(iii) SAR dHSIC.} Similar to TOT, SAR is also based on subsample-and-aggregate, but it aggregates the subset-level results differently. After public $p$-values are computed on the subsets, they induce a sequence of sub-test functions. The randomized response mechanism is then applied to each sub-test function. Specifically, it returns the true value of the indicator function with probability $p$ and the flipped value with probability $1-p$, where $p=e^{\epsilon}/(1+e^{\epsilon})$. However, the differential privacy guarantees and valid level for SAR hold only for sufficiently large sample sizes. We use the heuristic method in \citet[Section~3.3]{Pena2025DifferentiallyPH} to choose the tuning parameters. 

\begin{algorithm}
\caption{Multiple Testing Based on dpHSIC}\label{al_mul_dphsic}
    \renewcommand{\algorithmicrequire}{\textbf{Input:}}
    \renewcommand{\algorithmicensure}{\textbf{Output:}}
    \begin{algorithmic}[1]
        \REQUIRE Data $\mcX_n$, level $\alpha \in (0,1)$, privacy parameters $\epsilon > 0$ and $\delta \in [0,1)$, kernels $k^j$ and their upper bounds $K^j$ for $j\in[d]$, permutation number $B \in \mathbb{N}$.  
        \ENSURE Reject $H_0$ under $\alpha$-level or not.  
        \STATE Set $(\epsilon',\delta')=(\epsilon/(d-1),\delta/(d-1))$ and $\alpha'=\alpha/(d-1)$. 
        \FOR{$j \in \{2,3,\dots,d\}$}
            \STATE Calculate the sensitivity $\Delta_T^j=4(\prod_{l=1}^jK^l)^{1/2}n^{-1}$.  
            \STATE Test whether $X^j$ is independent of $(X^1,\dots,X^{j-1})$ using Algorithm~\ref{al_dp_per}, with privacy parameter $(\epsilon',\delta')$, level $\alpha'$, HSIC statistic with kernels $k^j$ and $\otimes_{l=1}^{j-1}k^l$, sensitivity $\Delta_T^j$, and $B$ permutations.
            \STATE Obtain $\mathbbm{1}(\hat{p}^j_{\DP}\leq\alpha')$, indicating whether the pairwise test rejects or not. 
        \ENDFOR
        \RETURN  $\max_j\mathbbm{1}(\hat{p}^j_{\DP} \leq \alpha')$. 
    \end{algorithmic}
\end{algorithm}

\subsection{Numerical Studies on Testing Joint Independence}
Our first goal is to test the joint independence among the variables $\{X^1,\dots,X^d\}$ by our proposed method and the competing methods. 
Without loss of generality, we consider Gaussian kernels as 
\begin{equation}
    k^j(x,x')=\exp\{-\|x-x'\|^2_2/(2\nu^2_j)\}, 
\end{equation}
which are bounded between $0$ and $1$. Following \cite{Pfister2018KernelbasedTF}, we use the median heuristic for choosing the bandwidth $\nu_j$ for the $j$-th kernel such that $\text{median}\{\|X^j_{i_1}-X^j_{i_2}\|^2_2:i_1<i_2\}=2\nu_j^2$. 
We set level $\alpha=0.05$ and number of permutations $B=200$ throughout the simulations. We fix the privacy parameter $\delta = 0$ and vary the parameter $\epsilon$ to control the privacy level.

\textbf{Simulation 1 (Testing level).} Consider the data $\bX=(X^1,\dots,X^d)$ generated as $\bX\sim N(0,I_d)$. Then it holds that 
$$
P_{X^1,\dots, X^d}= \otimes_{j=1}^d P_{X^j}\in H_0,
$$
in \eqref{eq_30}. We conduct $500$ repetitions to study the empirical level. We consider settings where the sample size $ n $ varies from $100$ to $1000$ in steps of $100$, and the privacy parameter $ \epsilon $ varies over the discrete set $\{10^{-4}, 10^{-3}, 10^{-2}, 10^{-1}, 1, 10, 20, 30, 40,50\}$. 
As SAR dHSIC requires a minimum sample size to ensure both privacy and validity, we set the test to not reject the null hypothesis when the sample size does not meet the required standard. 
The empirical sizes are summarized in Table~\ref{table_c1}. Overall, the empirical sizes are generally close to the nominal level for all methods. 

\begin{table}[htbp]
\centering \tabcolsep 4pt \LTcapwidth 5in
\caption{
In Simulation 1, the empirical sizes are reported at the significance level $\alpha=0.05$ and dimension $d=3$. The left panel fixes privacy at $\epsilon=1$ and varies the sample size $n$; the right panel fixes $n=300$ and varies privacy $\epsilon$. 
}
\label{table_c1}{\footnotesize
    \setlength{\tabcolsep}{4pt}
    \begin{threeparttable}
    \resizebox{\textwidth}{!}{%
	\begin{tabular}{cccccccccc}
		\toprule
		& \multicolumn{4}{c}{$\epsilon=1$ and $d=3$} &  & \multicolumn{4}{c}{$n=300$ and $d=3$} \\
		\cmidrule(lr){2-5}\cmidrule(lr){7-10}
        $n$ & dpdHSIC & MdpHSIC & TOT dHSIC & SAR dHSIC & $\epsilon$ & dpdHSIC & MdpHSIC & TOT dHSIC & SAR dHSIC \\
        \midrule
        100 & 0.044 & 0.046 & 0.044 & 0.000 & $10^{-4}$ & 0.030 & 0.074 & 0.056 & 0.000 \\ 
        200 & 0.046 & 0.044 & 0.042 & 0.002 & $10^{-3}$ & 0.032 & 0.074 & 0.056 & 0.000 \\ 
        300 & 0.056 & 0.048 & 0.044 & 0.052 & $10^{-2}$ & 0.044 & 0.048 & 0.046 & 0.000 \\ 
        400 & 0.048 & 0.050 & 0.050 & 0.052 & $10^{-1}$ & 0.040 & 0.044 & 0.044 & 0.000 \\ 
        500 & 0.050 & 0.040 & 0.040 & 0.038 & 1  & 0.056 & 0.048 & 0.044 & 0.052 \\ 
        600 & 0.046 & 0.044 & 0.048 & 0.046 & 10 & 0.042 & 0.054 & 0.040 & 0.032 \\ 
        700 & 0.046 & 0.048 & 0.038 & 0.058 & 20 & 0.046 & 0.056 & 0.040 & 0.032 \\ 
        800 & 0.040 & 0.060 & 0.046 & 0.044 & 30 & 0.052 & 0.054 & 0.040 & 0.032 \\ 
        900 & 0.054 & 0.068 & 0.040 & 0.052 & 40 & 0.058 & 0.036 & 0.030 & 0.058 \\ 
        1000 & 0.054 & 0.048 & 0.044 & 0.028 & 50 & 0.052 & 0.056 & 0.048 & 0.040 \\ 
		\bottomrule            
	\end{tabular}%
    }
    \end{threeparttable}}
\end{table}

\textbf{Simulation 2 (Joint dependence).} Consider data $\bX=(X^1,X^2,X^3)$ generated as $X^1,X^2\overset{i.i.d.}{\sim}N(0,1)$ and $X^3=X^1X^2+e$, where $e\sim N(0,\sigma^2)$ is generated independently of $X^1X^2$. Since the performance of MdpHSIC depends on the ordering of the variables, we randomly shuffle the order of $X^1,X^2,X^3$. 
We perform $200$ repetitions to investigate the empirical powers across different privacy regimes and sample sizes. Specifically, we consider varying privacy $\epsilon$ with fixed sample size $n=1000$, and privacy $\epsilon=10/\sqrt{n},1,\sqrt{n}/10$ with varying sample sizes from $250$ to $1500$. The results are presented in Figures~\ref{fig_s21} and~\ref{fig_s22}, respectively. 

In general, dpdHSIC has higher power than the three competing methods, especially in the mid- and low-privacy regimes. Because the dependence in this setting is primarily joint rather than pairwise, this design may favor dpdHSIC over MdpHSIC. 
In addition, because SAR dHSIC is set to not reject the null hypothesis when the minimum sample-size requirement is not met, its power is $0$ in the high-privacy regime. 
The power of TOT dHSIC is close to that of dpdHSIC in the high-privacy regime. However, as the privacy restrictions are relaxed, this advantage gradually disappears. Compared with the other three methods, TOT dHSIC is less sensitive to changes in the privacy level, possibly because it uses an optimal DP mechanism tailored to binomial variables. Nevertheless, for a fixed sample size, the power of TOT dHSIC is inherently limited by data splitting. Consequently, TOT dHSIC maintains relatively high power under strong privacy constraints, but has diminished power in low-privacy regimes.


\begin{figure}[htbp]
    \centering
\includegraphics[width=1\textwidth]{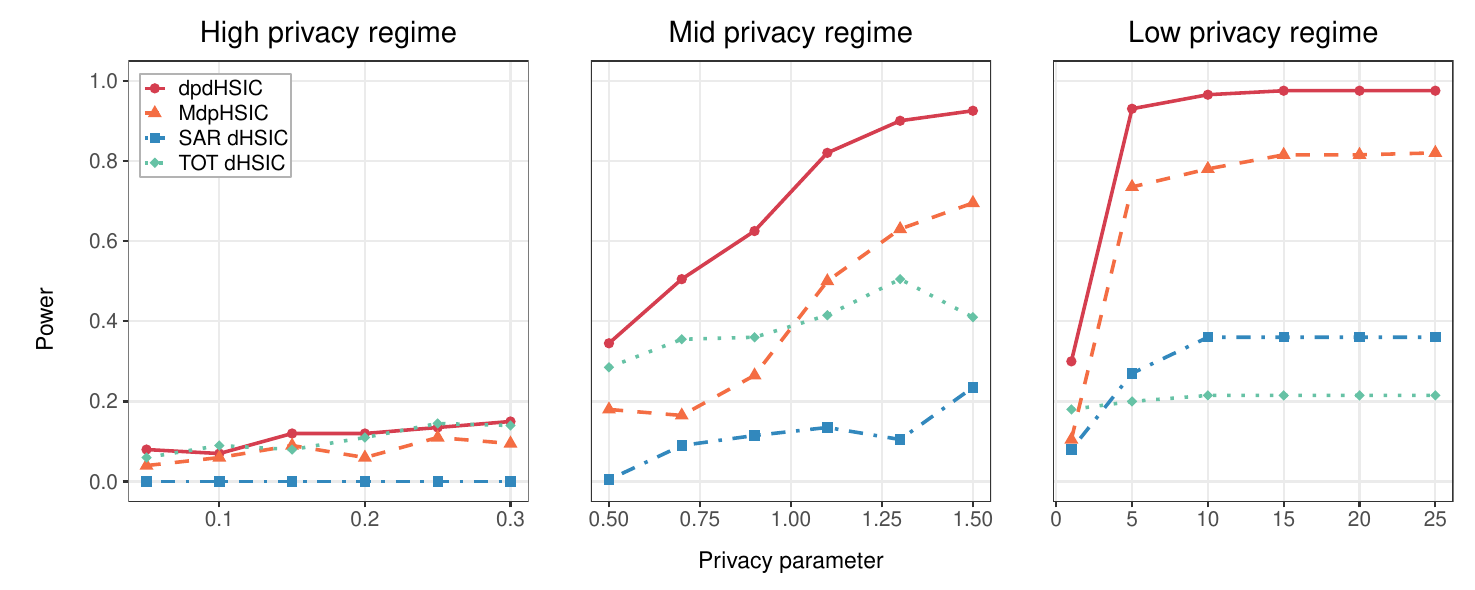}
    \caption{In Simulation~2 with $n=1000$, we vary privacy and the noise level from left to right: $\epsilon$ from $50/n$ to $10/\sqrt{n}$ with $\sigma=2$, $\epsilon$ from $0.5$ to $1.5$ with $\sigma=2$, and $\epsilon$ from $1$ to $25$ with $\sigma=3$.}
    \label{fig_s21}
\end{figure}
\begin{figure}[htbp]
    \centering
\includegraphics[width=1\textwidth]{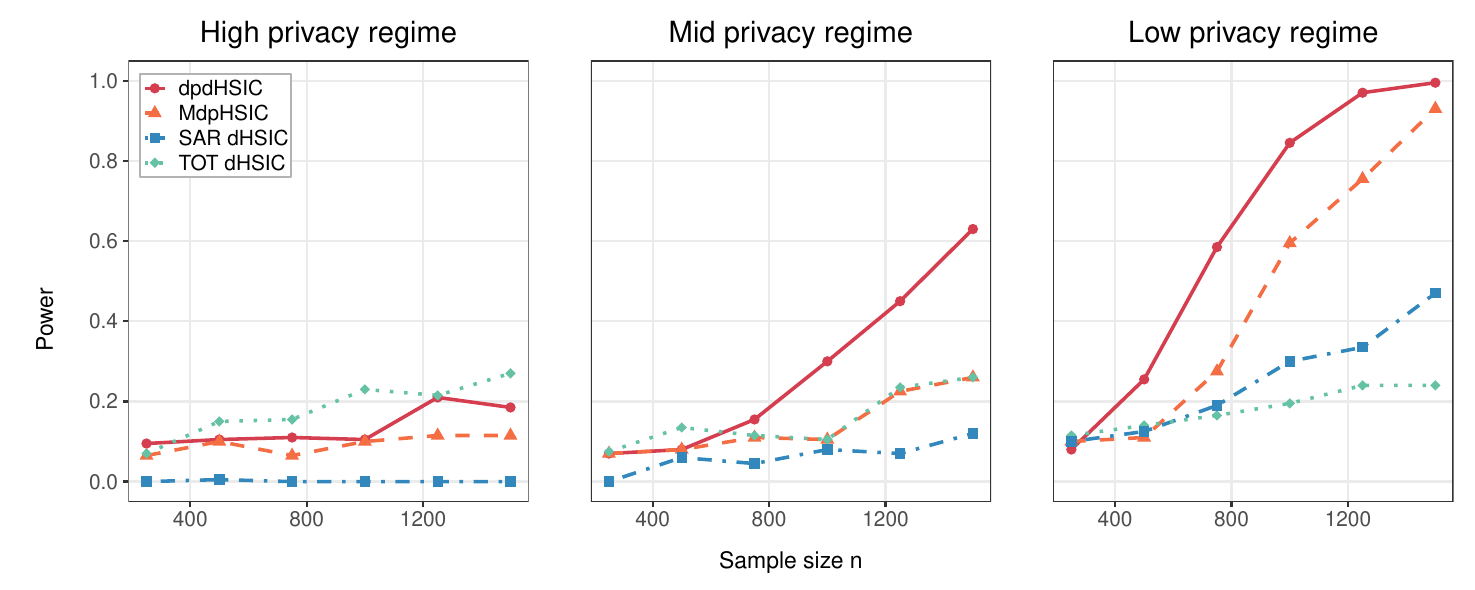}
    \caption{In Simulation~2, we vary the sample size under three privacy and noise settings from left to right: $\epsilon=10/\sqrt{n}$ with $\sigma=2$, $\epsilon=1$ with $\sigma=3$, and $\epsilon=\sqrt{n}/10$ with $\sigma=3$.}
    \label{fig_s22}
\end{figure}

\textbf{Simulation 3 (Marginal dependence).} Consider data $\bX=(X^1,\dots,X^d)\sim N(0,\Sigma_d)$ with a Toeplitz covariance matrix $\Sigma_d$, whose $(i, j)$ entry is $\rho^{|i-j|}$ for $i, j = 1,  \dots, d$. We conduct $200$ repetitions to investigate the empirical powers across different privacy regimes and dimensions $d$. We consider two cases as the dimension varies. The first tests whether the $d$ variables are jointly independent. The second tests whether three vectors are jointly independent, where the components of the three vectors are randomly sampled from the $d$ variables without replacement. 

The empirical powers of the two cases are presented in Figures~\ref{fig_s3} and ~\ref{fig_s32}, respectively. The results are somewhat interesting. One may assume that MdpHSIC outperforms the proposed dpdHSIC, since the dependence structure is primarily driven by marginal dependence. While this assumption holds in the low-privacy regime, the situation becomes more complex in the high and mid-privacy regimes. As noted in the last remark following Proposition~\ref{sen_dhsic}, MdpHSIC faces inherent limitations in the non-private regime and requires additional noise to ensure privacy. These drawbacks are exacerbated under more restrictive privacy constraints, leading to poorer performance in such regimes. 
Consequently, compared to MdpHSIC, our proposed dpdHSIC demonstrates higher power in the high and mid-privacy regimes. 
Similar to Simulation~2, TOT dHSIC performs well under strong privacy constraints, but it lacks power when the privacy constraints are relaxed. Although TOT dHSIC uses the ``optimal'' method to privately aggregate public results in the subsample-and-aggregate framework, it has even lower statistical power than SAR dHSIC in the low-privacy regime. This discrepancy may stem from the different parameter-selection recommendations in the respective original papers.

For the second case, where the $d$ variables are grouped into three sets for testing, the relative performance of the methods remains similar. However, the relationship between dimension and power changes. The power of dpdHSIC and MdpHSIC decreases as $d$ increases, whereas the power of TOT dHSIC and SAR dHSIC increases with $d$. We conjecture that this phenomenon is caused by the difference between the noise-addition mechanisms used by the two groups of methods. As the dimension increases, many weak but dense signals appear. dpdHSIC and MdpHSIC may obscure these signals by directly adding noise to both the test statistic and its permuted counterparts. In contrast, the two sample-splitting methods first perform testing on non-perturbed data and then combine the results privately. This approach may allow weak signals to persist without being completely masked by the added noise.

\begin{figure}[htbp]
    \centering
\includegraphics[width=1\textwidth]{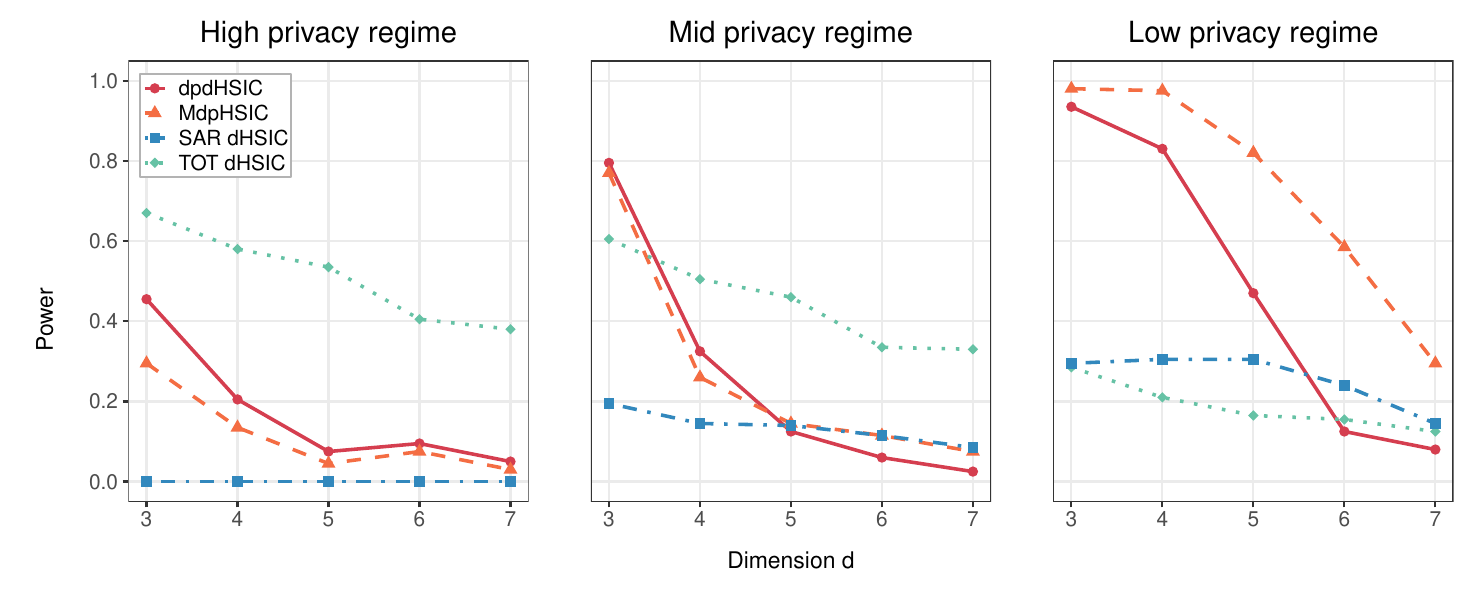}
    \caption{Testing joint independence among $d$ variables in Simulation~3 as the dimension $d$ varies. We set $\alpha=0.05$ and vary the privacy parameter and correlation as follows: (Left) $\epsilon=10/\sqrt{n}$ and $\rho=0.3$. (Middle) $\epsilon=1$ and $\rho=0.3$. (Right) $\epsilon=\sqrt{n}/5$ and $\rho=0.2$.}
    \label{fig_s3}
\end{figure}

\begin{figure}[htbp]
    \centering
\includegraphics[width=1\textwidth]{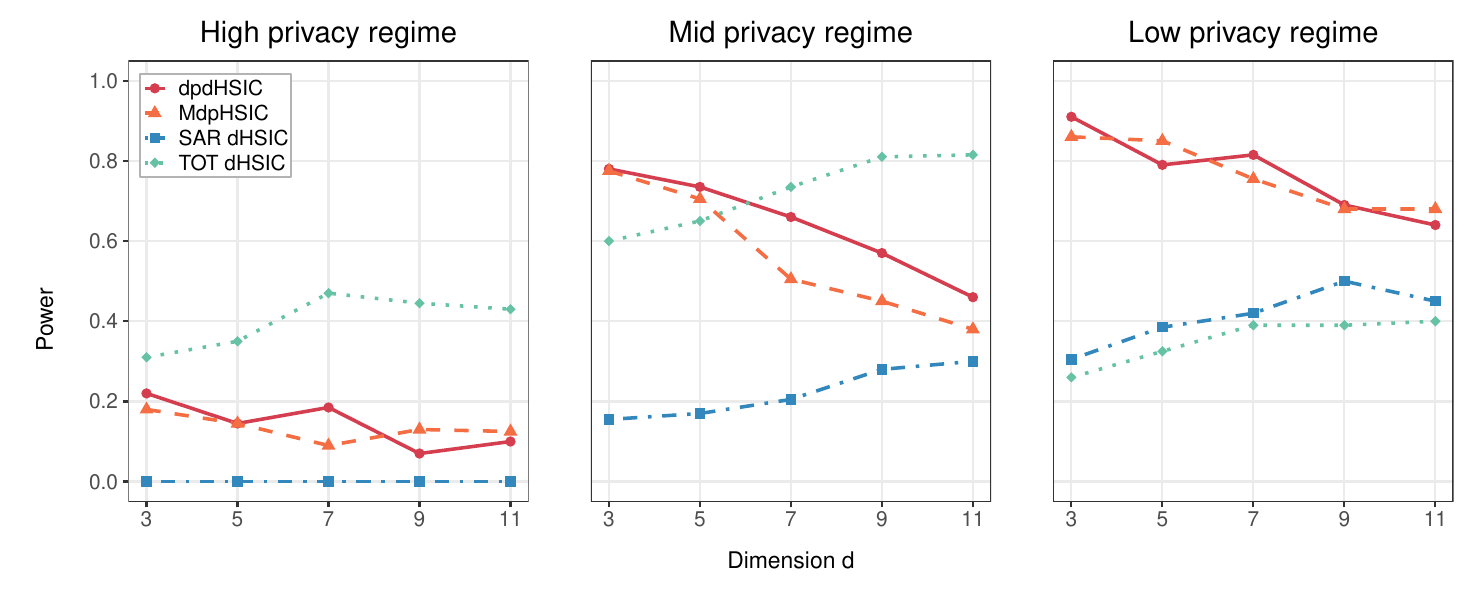}
    \caption{Testing joint independence among three vectors in Simulation~3 as the dimension $d$ varies. We set $\alpha=0.05$ and vary the privacy parameter and correlation as follows: (Left) $\epsilon=10/\sqrt{n}$ and $\rho=0.3$. (Middle) $\epsilon=1$ and $\rho=0.3$. (Right) $\epsilon=\sqrt{n}/10$ and $\rho=0.2$.}
    \label{fig_s32}
\end{figure}


\section{Application to Causal Inference}\label{sec_realdata}
\subsection{Model Diagnostic Checking for Directed Acyclic Graph}\label{sec_61}
In this section, our primary aim is to perform model diagnostic checking for DAGs under privacy constraints. 
To this end, we first review the DAG verification method proposed by \cite{Pfister2018KernelbasedTF}. Following \cite{Peter2014CAM}, assume that the joint distribution is induced by $d$ structural equations
\begin{equation}\label{eq_31}
    X^j:=\sum_{k\in {\rm PA}^j}f^{j,k}(X^k)+N^j,\quad j\in\{1,\dots,d\},
\end{equation}
with ${\rm PA}^j$ being the parents of $j$ in the associated DAG $\mathcal{G}_0$. The noise variables $N^1,\dots,N^d$ are normally distributed and are assumed to be jointly independent. 
Model diagnostic checking for a candidate DAG $\mathcal{G}$ in \eqref{eq_31} assesses the goodness of fit of $\mathcal{G}$, which reduces to testing the joint independence of the residuals from the fitted model \citep{Pfister2018KernelbasedTF}.  
Specifically, given observations $(\bX_1,\dots,\bX_n)$ and candidate DAG $\mathcal{G}$, consider the following procedure. 
\begin{itemize}
    \item Use generalized additive model regression \citep{Wood2002GAMsWI} to regress each node $X^j$ on its parents ${\rm PA}^j$ and denote the resulting residual vector by ${\rm res}^j$.
    \item Perform dHSIC to test whether $({\rm res}^1,\dots,{\rm res}^d)$ is jointly independent.  
    \item If $({\rm res}^1,\dots,{\rm res}^d)$ is not jointly independent, then the DAG $\mathcal{G}$ is rejected. 
\end{itemize}
In the first step, the residuals of the fitted models are tested instead of the true noise values in model \eqref{eq_31}. 
This does not affect the asymptotic ordering of $\hat{\dHSIC}$ \citep[Theorem~E.2]{Pfister2018KernelbasedTF}, and thus does not affect the ordering of dpdHSIC as the added noise is higher-order.  


\subsection{Portuguese Bank Marketing Campaign Data}
The Bank Marketing data from a Portuguese banking institution records direct marketing campaigns conducted by phone calls from $2008$ to $2010$ \citep{Moro2014ADA}. The dataset contains clients' sensitive personal information, such as age, job type, education level, credit default status, account balance, and personal loan status. The original classification goal is to predict whether a client subscribes to a bank term deposit and is also analyzed in \cite{Bi2023DistributionInvariantDP}. The dataset is available at \url{http://archive.ics.uci.edu/ml/datasets/Bank+Marketing#}. 

We use these data to construct a model diagnostic problem for candidate DAGs in private regimes. Specifically, we use a random subsample of size $n=1000$ and focus on five processed variables: 
account balance, debt burden, campaign contacts, last-call duration, and the binary subscription outcome. We consider three candidate DAGs as shown in Figure~\ref{fig_bank_dag}. For each candidate DAG, we adopt the procedure described in Section~\ref{sec_61} to perform model diagnostics under privacy constraints. 

\begin{figure}[htbp] 
    \centering
    \includegraphics[width=0.9\linewidth]{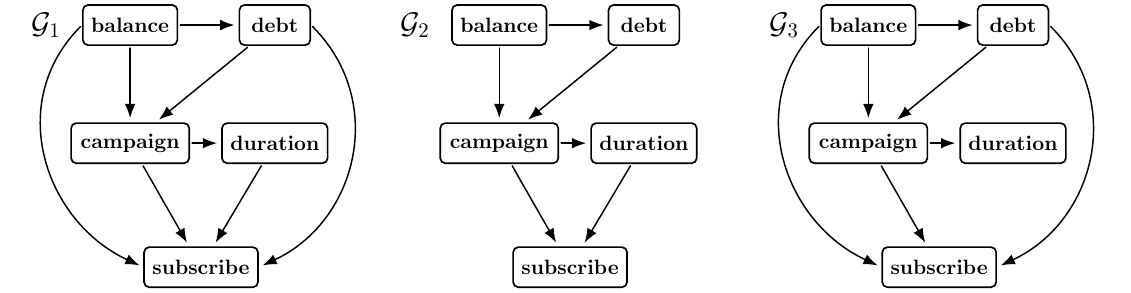} 
    \caption{Three candidate DAGs for the bank dataset.}
    \label{fig_bank_dag}
\end{figure}

Table~\ref{tab_bank_debt} reports empirical rejection rates at level $\alpha=0.05$ for the four methods considered in Section~\ref{sec_simulation}. We use $M=100$ repetitions and $B=500$ permutations in each repetition, and vary $\epsilon\in\{1,3,10\}$ to cover privacy levels from stronger to weaker protection.

In this diagnostic setting, a lower rejection rate indicates less detectable joint dependence among the residuals under a candidate DAG. It should therefore be read as relative support among the proposed working DAGs, rather than as evidence that a candidate is the true data-generating graph. 
For dpdHSIC, the same qualitative ordering appears at all three privacy levels: $\mathcal{G}_1$ has the smallest rejection rate, $\mathcal{G}_2$ is intermediate, and $\mathcal{G}_3$ is rejected most often. This ordering becomes more pronounced as the privacy constraint is relaxed. 
It suggests that the direct debt-related effects help explain part of the dependence structure, and that omitting the call-duration edge creates the largest residual diagnostic conflict. 
Moreover, rejection rates greater than the nominal level imply that this five-variable DAG remains a simplified working model of the Bank Marketing process, but it is the best supported candidate among the three. 

Compared with MdpHSIC, dpdHSIC yields a more consistent result across the considered privacy levels. Its qualitative conclusions are also closer to those of the two dHSIC-based aggregation methods, SAR HSIC and TOT HSIC, which generally favor the same relative ordering of the candidate DAGs. Meanwhile, dpdHSIC retains a clearer gradation among the three candidates, making the relative evidence for the working DAGs easier to interpret under privacy constraints.

\begin{table}[htbp]
    \centering
    \caption{Empirical rejection rates for DAG checking in the Bank Marketing data.}
    \label{tab_bank_debt}
    {\footnotesize
    \begin{threeparttable}
    \begin{tabular}{cccccc}
        \toprule
        $\epsilon$ & Candidate DAG & dpdHSIC & MdpHSIC & SAR dHSIC & TOT dHSIC \\
        \midrule
        \multirow{3}{*}{$1$}
        & $\mathcal{G}_1$ & 0.03 & 0.10 & 0.06 & 0.07 \\
        & $\mathcal{G}_2$ & 0.06 & 0.04 & 0.04 & 0.11 \\
        & $\mathcal{G}_3$ & 0.08 & 0.10 & 0.07 & 0.17 \\
        \midrule
        \multirow{3}{*}{$3$}
        & $\mathcal{G}_1$ & 0.07 & 0.18 & 0.09 & 0.07 \\
        & $\mathcal{G}_2$ & 0.13 & 0.14 & 0.13 & 0.10 \\
        & $\mathcal{G}_3$ & 0.31 & 0.33 & 0.24 & 0.26 \\
        \midrule
        \multirow{3}{*}{$10$}
        & $\mathcal{G}_1$ & 0.15 & 0.82 & 0.15 & 0.12 \\
        & $\mathcal{G}_2$ & 0.74 & 0.88 & 0.17 & 0.18 \\
        & $\mathcal{G}_3$ & 0.99 & 1.00 & 0.38 & 0.22 \\
        \bottomrule
    \end{tabular}
    \end{threeparttable}}
\end{table}

\subsection{Pima Indians Diabetes Data}
In clinical treatment scenarios, patients may be particularly sensitive to the privacy of health-related measurements, especially body indices that are associated with specific diseases. 
We consider the Pima Indians Diabetes dataset, which contains observations on women who were at least $21$ years old, of Pima Indian heritage, and residing near Phoenix, Arizona. These individuals were tested for diabetes according to the World Health Organization criteria. The dataset, available from the \url{https://www.kaggle.com/datasets/uciml/pima-indians-diabetes-database}, was originally collected by the U.S. National Institute of Diabetes and Digestive and Kidney Diseases. 
Following the analysis of \cite{Chakraborty2019DistanceMF}, we focus on five variables: age, body mass index (BMI), 2-hour serum insulin (SI), plasma glucose concentration (glu), and diastolic blood pressure (DBP). After removing observations with zero values, $n=392$ observations remain.

\begin{figure}[htbp] 
    \centering
    \includegraphics[width=0.6\linewidth]{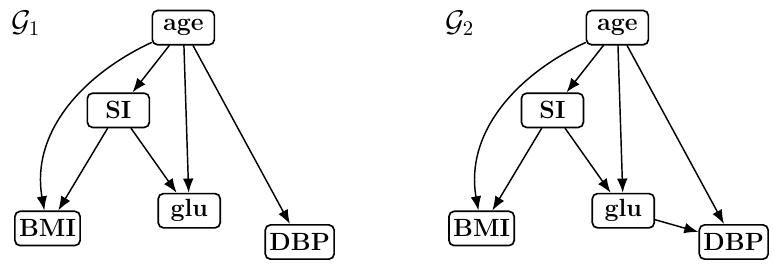} 
    \caption{Two candidate DAGs for the diabetes data.}
    \label{fig_dag}
\end{figure}

Our aim is to perform model diagnostic checking for candidate DAGs under privacy constraints. The two DAGs in Figure~\ref{fig_dag} are the two most plausible candidates as demonstrated in  \cite{Chakraborty2019DistanceMF}. They differ only by the presence of an additional edge from glu to DBP, which represents a possible direct effect of plasma glucose concentration on diastolic blood pressure. 
We perform $1000$ repetitions and use $B=500$ permutations in each repetition to estimate the empirical rejection rate at level $\alpha=0.05$. Results of the four methods are reported in Figure~\ref{fig_real}. 
For dpdHSIC, both candidate DAGs have rejection rates close to the nominal level, suggesting that neither graph exhibits strong residual joint dependence after the DAG-based residual adjustment. Moreover, dpdHSIC remains the closest to the nominal level across the privacy budgets considered here. This conclusion is consistent with the non-private analysis of \cite{Chakraborty2019DistanceMF}. 

\begin{figure}[htbp] 
    \centering
    \includegraphics[width=0.6667\linewidth]{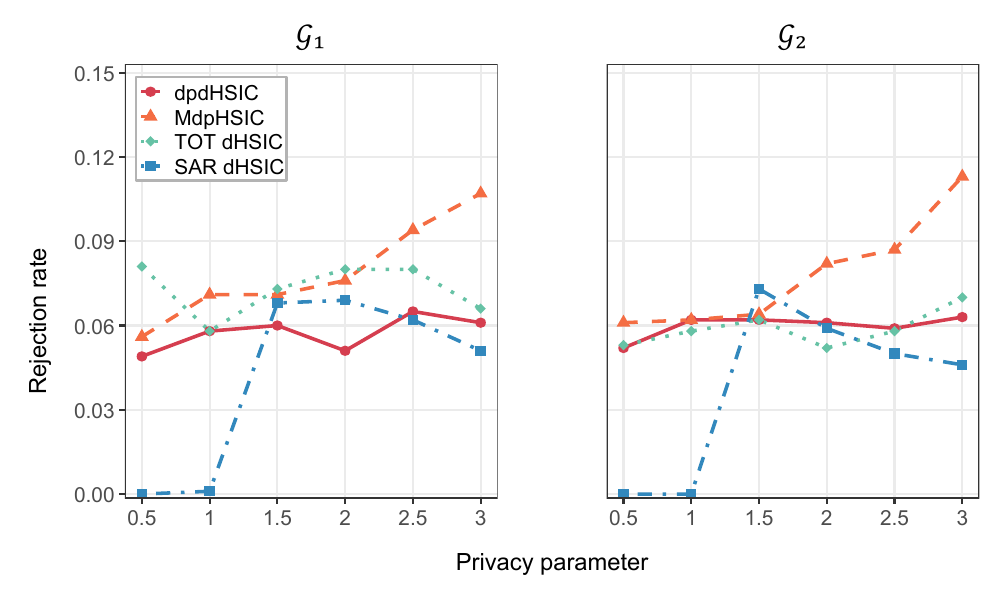} 
    \caption{Empirical rejection rates of the two DAGs in the real data analysis. The left panel reports rejection rates for DAG $\mathcal{G}_1$, and the right panel reports rejection rates for DAG $\mathcal{G}_2$.}
    \label{fig_real}
\end{figure}

\section{Conclusion and Discussion} \label{sec_discussion}
In this work, we propose a DP joint independence test based on dHSIC. Given that the limiting distribution of the empirical estimate of dHSIC is a complex Gaussian chaos, we consider two resampling methods, permutation and bootstrap under privacy constraints. 
We show  that the proposed permutation-based dpdHSIC enjoys DP guarantees, a non-asymptotic level and pointwise consistency, while the bootstrap counterpart suffers from power inconsistency. We further investigate the uniform power of the proposed dpdHSIC in terms of dHSIC metric and $L_2$ metric. The minimum separation of the proposed test matches the minimax separation up to constant factors across different privacy regimes. Numerical simulations and real data analysis validate our theoretical findings. 

Although the DP bootstrap for dHSIC suffers from intractable sensitivity, it shares the same sensitivity as permutation for MMD \citep{Gretton2012AKT}. We conjecture that the large sensitivity is caused by the tensor product kernel. The further comparison between DP bootstrap and DP permutation would be an interesting direction for future investigation. 

\bibliographystyle{apalike}
\bibliography{DPJointIT}
\newpage

\begin{appendices}
\label{app}
\allowdisplaybreaks[3]

\renewcommand{\theequation}{\thesection.\arabic{equation}}

\setcounter{table}{0}
\renewcommand{\thetable}{\thesection\arabic{table}}
\setcounter{figure}{0}
\renewcommand{\thefigure}{\thesection\arabic{figure}}
\setcounter{algorithm}{0}
\renewcommand{\thealgorithm}{\thesection\arabic{algorithm}}

\section{Overview of Appendix}
This supplementary material includes technical details and proofs omitted from the main text. The remaining material is organized as follows. 
\begin{itemize}
    \item Appendix~\ref{sec_pre} reviews the basic concepts of differential privacy and reproducing kernel Hilbert spaces used throughout the paper.
    \item Appendix~\ref{sec_u_dhsic} presents the theoretical results for the U-statistic version of dHSIC, including its sensitivity and its separation behavior in the dHSIC and $L_2$ metrics.
    \item Appendix~\ref{app_technical_proof} contains the proofs for the main results. Specifically, Appendix~\ref{app_lemma_dp_re} proves the properties of the DP resampling test; Appendix~\ref{app_pr_sen} proves the sensitivity results for permutation dHSIC; Appendix~\ref{app_property_dpdH} proves the properties of the dpdHSIC test; Appendices~\ref{app_a23} and~\ref{app_a24} prove the negative results for bootstrap dHSIC; and Appendices~\ref{app_th_mini_sepa_dpdhsic},~\ref{app_th_minimax_sepa} and~\ref{app_th_L2separation} prove the separation results for dpdHSIC.
    \item Appendix~\ref{app_proof_u} gives the proofs for the theoretical results on the U-statistic version of dHSIC.
    \item Appendix~\ref{app_lemma} collects auxiliary lemmas, including lemmas on resampling tests in Appendix~\ref{app_lemma_re} and lemmas on dHSIC in Appendix~\ref{app_le_dhsic}.
\end{itemize}

\section{Preliminaries}\label{sec_pre}
\subsection{Brief Review of Differential Privacy}\label{sec_DP}
We briefly review some basic concepts and properties of differential privacy \citep{Dwork2006CalibratingNT,Dwork2014TheAF}. It is a probabilistic property of randomized algorithms that limits the information that can be learned when replacing an individual in the dataset. 
Differential privacy with additional auxiliary variables requires the probabilistic property hold for all possible values of auxiliary variable. In the paper, the resampling procedure can be treated as the auxiliary variables independent of the dataset. 

\begin{definition}[Differential Privacy \citep{Dwork2014TheAF}]
\label{def_dp}
    Let $\epsilon > 0$ and $\delta \in [0, 1)$. 
     A randomized algorithm $\mathcal{M}$ satisfies $(\epsilon, \delta)$-DP if, for all $S \in \text{range}(\mathcal{M})$, any additional auxiliary variable $w \in \mathcal{W}$ and all input datasets $\mcX_n$ and $\tilde{\mcX}_n$ with $d_{\ham}(\mcX_n, \tilde{\mcX}_n) \leq 1$, we have
     $$
     \mathbb{P}(\mathcal{M}(\mcX_n; w) \in S \mid \mcX_n, w) \leq e^\epsilon \mathbb{P}(\mathcal{M}(\tilde{\mcX}_n; w) \in S \mid \tilde{\mcX}_n, w) + \delta.
     $$ 
\end{definition}

The $\epsilon$ and $\delta$ in Definition~\ref{def_dp} are referred to as privacy parameters. 
The smaller $\epsilon$ and $\delta$ indicate the stricter privacy guarantees, where the probability distributions of $\mathcal{M}(\mcX_n; w)$ and $\mathcal{M}(\tilde\mcX_n; w)$ are forced to be similar, so the output of $\mathcal{M}$ is not very sensitive to small changes in $\mcX_n$. 
It is usually desirable to take $\delta$ approaching $0$ to ensure higher privacy. 
When $\delta=0$ and without auxiliary variables, we retrieve $\epsilon$-differential privacy, which is the most popular formalization of privacy in the literature. 

We review that the composition theorem \citep[Theorem~3.16]{Dwork2014TheAF}, which presents the overall privacy guarantees for a composition of multiple DP mechanisms and serves as a foundation for combining DP pairwise independence testing results. 

\begin{lemma}[Composition]\label{lemma_com}
    Suppose that each algorithm $\mathcal{M}_i$ is $(\epsilon_i, \delta_i)$-DP for $i \in [m]$. Then, the composed algorithm $\mathcal{M}_{1:m}$ defined as $\mathcal{M}_{1:m} := (\mathcal{M}_1, \dots, \mathcal{M}_m)$ is $(\sum_{i=1}^m \epsilon_i, \sum_{i=1}^m \delta_i)$-DP.
\end{lemma}

In this article, $\mathcal{M}(\mcX)$ are perturbed versions of confidential statistics $T(\mcX)$.
A keystone to guarantee differential privacy is the sensitivity \citep{Dwork2014TheAF}, which determines the scaling factor of perturbation. 
As a generalization of commonly used definition in the literature, the sensitivity with additional auxiliary variables is stated as follows. 

\begin{definition}[$\ell_p$-Sensitivity]
    Consider a function $T:\mcX_n\to\mathbb{R}^r$ and an additional auxiliary variable $w \in \mathcal{W}$. For $p \geq 1$, the global $\ell_p$-sensitivity of $T$ is defined as
    $$
    \Delta_T^p := \sup_{w \in \mathcal{W}} \sup_{d_{\ham}(\mcX_n, \tilde{\mcX}_n) \leq 1} 
    \|T(\mcX_n; w) - T(\tilde{\mcX}_n; w)\|_p.
    $$
\end{definition}

The most commonly used mechanisms to guarantee differential privacy may be the Laplace mechanism and Gaussian mechanisms \citep{Dwork2006CalibratingNT}. 
In this paper, there is little difference between the two mechanisms, as the outcome of $T$ is also one-dimensional. Therefore we simply focus on Laplace mechanism working with $\ell_1$-sensitivity, and $\Delta_T^1$ can be abbreviated as $\Delta_T$. 
The Laplace mechanism is formulated as follows.

\begin{lemma}[The Laplace Mechanism]\label{Laplace}
    For $T:\mcX_n\to\mathbb{R}^r$ and an additional auxiliary variable $w \in \mathcal{W}$, the following mechanism $\mathcal{M}$ is $(\epsilon,\delta)$-DP:
    $$
    \mathcal{M}(\mcX_n;w) = T(\mcX_n;w) +\frac{\Delta_T}{\xi_{\epsilon,\delta}} (\zeta_1, \ldots, \zeta_r)^T,
    $$
    where $\zeta_i \overset{i.i.d.}{\sim}\Lap(0,1)$ generated independently of $\mcX_n$.
\end{lemma}

\subsection{Brief Review of Reproducing Kernel Hilbert Spaces}\label{sec_embed}
We briefly review RKHSs and mean embeddings here; see \cite{2004Reproducing} for a comprehensive treatment. 
For a reproducing kernel $k:\mbX\times\mbX\to\mbR$ defined on a separable topological space $\mbX$, let $\mcH_k$ be a RKHS endowed with $k$. 
It is possible to embed complicated objects into an RKHS and to analyze them by using the Hilbert space structure. 
\cite{Pfister2018KernelbasedTF} embedded probability distributions in an RKHS, where the mean embedding associated with $k$ is defined as the function 
$$
\Pi(\mu):=\int_{\mbX}k(x,\cdot)\mu(\mathrm{d}x)\in\mcH_k. 
$$
To infer that two distributions are equal given that their embeddings coincide, it is necessary that the mean embedding is injective, where such a kernel is called characteristic. A kernel $k$ is said to be \emph{translation invariant}, if there exists a symmetric positive definite function $\kappa$ such that $k(x,x')=\kappa(x-x')$ for all $x,x'\in \mbX$. Assuming that $0\leq k(x,x')\leq K$ for all $x,x'\in\mbX$, then the kernel $k$ has \emph{non-empty level} sets on $\mbX$ if, for any $\varepsilon\in(0,K)$, there exist $x,x'\in\mbX$ such that $k(x,x')\leq\varepsilon$. Both the Gaussian kernel and the Laplace kernel have non-empty level sets on $\mbR^d$ due to the continuity of the kernel function.

In addition, we define product kernels on $\mbX^1\times\dots\times\mbX^d$.
For $j\in[d]$, let $k^j:\mbX^j\times\mbX^j\to\mbR$ be a continuous, bounded, positive semi-definite kernel and denote by $\mcH^j$ the corresponding RKHS. Throughout the paper, a superscript $j\in[d]$ always denotes an index rather than an  exponent. Assume that the tensor product
$\bk:=\otimes_{j=1}^d k^j$ is characteristic, which is defined as $\otimes_{j=1}^d k^j((x^1,x'^{1}),\dots, (x^d,x'^{d}))=\prod_{j=1}^dk^j(x^j,x'^{j})$ for all $x^j,x'^{j}\in \mbX^j$ and $j\in[d]$. Let $\mcH_{\bk}=\mcH^1\otimes\dots\otimes\mcH^d$ be the projective tensor product of the RKHSs $\mcH^j$. 

\section{Theoretical Results on the U-statistic}\label{sec_u_dhsic}
U-statistics often enjoy the optimal separation rate in non-private regimes \citep{Albert2022AdaptiveTO,Kim2022MinimaxOO}. In this section, we explore the private dHSIC framework based on a U-statistic, which exhibits sub-optimal separation performance in some cases, especially in high-privacy regimes.

To begin with, we recall the closed form of the dHSIC U-statistic, which is an unbiased estimator of the squared population dHSIC in \eqref{df_dhsic} and given by
\begin{align*}
    U_{\dHSIC} :=& \frac{(n-2)!}{n!} \sum_{\bi_2^n} \prod_{j=1}^d
    k^j(X_{i_1}^j, X_{i_2}^j)  
    + \frac{(n-2d)!}{n!} \sum_{\bi_{2d}^n} \prod_{j=1}^d
    k^j(X_{i_{2j-1}}^j, X_{i_{2j}}^j) \nonumber\\
    &- \frac{2(n-d-1)!}{n!} \sum_{\bi_{d+1}^n} \prod_{j=1}^d
k^j(X_{i_1}^j, X_{i_{j+1}}^j).
\end{align*}
The following lemma computes the sensitivity of the permutation-based $U_{\dHSIC}$.
\begin{lemma}\label{le_sen_u}
    Assume that the kernels $k^j$  are bounded as $0\leq k^j(x,x')\leq K^j$ for all $x,x'\in \mbX^j$ and $j\in[d]$. 
    In addition, assume that $k^j$ are translation invariant and have non-empty level sets on $\mbX^j$. Then there exists a positive sequence $c_n \in [2, 4d^2+4d]$ such that for all $n \geq 2d$,
$$
\sup_{\bpi\in\bPi_n} \sup_ {\substack{\mcX_n, \widetilde{\mcX}_n:\\ d_{{\ham}}(\mcX, \widetilde{\mcX})\leq1}} | U_{\dHSIC}(\mcX^{\bpi}_n) - U_{\dHSIC}(\tilde\mcX^{\bpi}_n) | = \frac{c_n}{n}\prod_{j=1}^dK^j.
$$
\end{lemma}

It is worth highlighting that the U-statistic is an estimator of $\dHSIC^2_{\bk}$, while the V-statistic is an estimator of $\dHSIC_{\bk}$. As a result, the impact of the $n^{-1}$ order sensitivity on their power is different. 
We denote the test function using $U_{\dHSIC}$ based on Algorithm~\ref{al_dp_per} and Lemma~\ref{le_sen_u} by $\phi^u_{\dpdHSIC}$. We next devote attention to the power analysis of $\phi^u_{\dpdHSIC}$. To begin with, we consider the minimum separation in the dHSIC metric. The following theorem shows that $\phi^u_{\dpdHSIC}$ fails to achieve the minimax separation rate over $\mathcal{P}_{\dHSIC_{\bk}}(\rho)$.

\begin{theorem}[Sub-optimality of $\phi^u_{\dpdHSIC}$ against dHSIC Alternatives]\label{th_a1}
Assume that the kernels $k^j$ fulfill the conditions specified in Lemma~\ref{le_sen_u}. Moreover, assume that if $P_{X^1,\dots, X^d} \in \mcP_{\mbX^1\times\dots\times\mbX^d}$, then $wP_{X^1,\dots, X^d} + (1-w)P_{X^1}\dots P_{X^d} \in \mcP_{\mbX^1\times\dots\times\mbX^d}$ for all $w \in [0, 1]$, and there exist $P_{X^1,\dots, X^d} \in \mcP_{\mbX^1\times\dots\times\mbX^d}$ such that $\dHSIC_{\bk}(P_{X^1,\dots, X^d}) = \varrho_0 > 0$ for some fixed $\varrho_0 > 0$. Let $\alpha \in ((B+1)^{-1}, 1)$, $\beta \in (0, 1-\alpha)$ be fixed values. Consider the high-privacy regime where $\xi_{\epsilon, \delta} \asymp n^{-1/2-r}$ with fixed $r \in (0, 1/2)$, for $\xi_{\epsilon, \delta}$. Then the uniform power of $\phi^u_{\dpdHSIC}$ is asymptotically at most $\alpha$ over $\mcP_{\dHSIC_{\bk}}(\rho)$ where
\begin{equation} \label{eq_32} 
\rho = \log(n) \times \max \left\{ \sqrt{\frac{\max\{\log(1/\alpha), \log(1/\beta)\}}{n}}, \frac{\max\{\log(1/\alpha), \log(1/\beta)\}}{n\xi_{\epsilon, \delta}} \right\}.
\end{equation}
In other words, it holds that
$$
\limsup_{n\to\infty} \inf_{P_{X^1,\dots, X^d} \in \mcP_{\dHSIC_{\bk}}(\rho)} \mathbb{E}_{P_{X^1,\dots, X^d}}[\phi^u_{\dpdHSIC}] \le \alpha.
$$
\end{theorem}
It is clear that $\phi^u_{\dpdHSIC}$ is not minimax optimal in dHSIC metric as $\rho/\rho^*_{\dHSIC}\to\infty$ as $n\to\infty$. The sub-optimality is mainly caused by larger sensitivity and additional noise. Roughly speaking, $\phi^u_{\dpdHSIC}$ requires the signal $\dHSIC^2_{\bk}\gtrsim(n\xi_{\epsilon,\delta})^{-1}$ for consistent power, i.e. $\dHSIC_{\bk}\gtrsim(n\xi_{\epsilon,\delta})^{-1/2}$. Otherwise, the signal is governed by the added Laplace noise. However, the minimax separation in the high-privacy regime requires signal $\dHSIC_{\bk}\gtrsim\min\{(n\xi_{\epsilon,\delta})^{-1},1\}$ by Theorem~\ref{th_minimax_sepa}, which is less than $(n\xi_{\epsilon,\delta})^{-1/2}$. 

Next we move our attention to $L_2$ alternatives. Denote the minimum separation of $\phi^u_{\dpdHSIC}$ with Gaussian kernel over ${\mcP}^s_{L_2}(\rho)$ as $\rho_{\phi^u_{\dpdHSIC},L_2}$, which is defined similarly to $\rho_{\phi_{\dpdHSIC},L_2}$. Establishing the lower bound for $\rho_{\phi^u_{\dpdHSIC},L_2}$ is a much easier task than that for $\rho_{\phi_{\dpdHSIC},L_2}$, and we formulate it as Theorem~\ref{th_a2}.
\begin{theorem}[Minimum Separation of $\phi_{\dpdHSIC}$ over ${\mcP}^s_{L_2}$]
\label{th_a2}
    Assume that $\alpha \in (0,1)$, $\beta\in(0, 1-\alpha)$ are fixed, and $\epsilon>0$, $\delta \in [0, 1)$, 
    $B \geq 6\alpha^{-1} \log(2\beta^{-1})$, $\lambda_j \leq 1$ for $j\in [d]$. The minimum separation of $\phi^u_{\dpdHSIC}$ with the Gaussian kernels over ${\mcP}^s_{L_2}$ is upper bounded as
\begin{align*}
    \rho^2_{\phi^u_{\dpdHSIC},L_2} \leq C_{\alpha, \beta, s,R,M, p_1,\dots,p_d} \bigg\{ \sum_{j=1}^d\sum_{i=1}^{p_j} \lambda_{j,i}^{2s} + \frac{1}{n \sqrt{\lambda_0}} + \frac{1}{n \lambda_0 \xi_{\epsilon, \delta}}
    \bigg\}
\end{align*}
where $ C_{\alpha, \beta, s,R,M, p_1,\dots,p_d} $ is some positive constant, depending only on $\alpha, \beta, s,R,M, p_1,\dots,p_d$.
\end{theorem}

By comparing the results in Theorem~\ref{th_L2separation} and Theorem~\ref{th_a2}, the upper bound for $\rho_{\phi^u_{\dpdHSIC},L_2}$ is less than $\rho_{\phi_{\dpdHSIC},L_2}$ if $n\xi_{\epsilon, \delta}\lesssim1$, which indicates a condition on extremely high privacy restrictions. 
Since we assume $\lambda_0\leq1$, $n\xi_{\epsilon, \delta}\lesssim1$ implies that the signal $\|P_{X^1,\dots,X^d}-P_{X^1}\dots P_{X^d}\|_{L_2}$ is sufficiently large. However, the $L_2$ distance cannot be made sufficiently large as $\|P_{X^1,\dots,X^d}\|_{L_{\infty}}$ and $\|P_{X^1}\dots P_{X^d}\|_{L_{\infty}}$ are assumed to be bounded. Therefore, the proposed test dpdHSIC is more favorable than the one based on the U-statistic when it comes to achieving a tight separation rate in the $L_2$ distance. 
When $\xi_{\epsilon, \delta}\gtrsim\lambda_0^{-1/2}$, the test based on the U-statistic also achieves the optimal separation rate $n^{-2s/(4s+p)}$ over the Sobolev ball by setting $\lambda_{j,i}= n^{-2/(4s+p)}$ like our dpdHSIC.


\section{Technical Proofs}\label{app_technical_proof}
\subsection{Proof for Lemma~\ref{lemma_dp_re}}\label{app_lemma_dp_re}
\begin{proof}
    By Lemma~\ref{lemma_repre1},
    \[
    \mathbbm{1}(\hat{p}_{\DP} \leq \alpha)=\mathbbm{1}(M_0>q_{1-\alpha}),
    \]
    where $q_{1-\alpha}$ is the $1-\alpha$ quantile of $M_0,M_1,\dots,M_B$. By Lemma~\ref{lemma_repre2},
    \[
    \mathbbm{1}(M_0>q_{1-\alpha})
    =\mathbbm{1}(M_0 > r_{1-\alpha_\star}) \mathbbm{1}(\alpha \geq {1}/{n+1}),
    \]
    where $\alpha_\star = \max\{({(n+1)}\alpha/{n}  - {1}/{n}), 0\}$ and $r_{1-\alpha_\star}$ is the $1-\alpha_\star$ quantile of $\{M_i\}_{i=1}^n$. The sensitivity of $T(\mcX_n)-r_{1-\alpha_\star}$ is upper bounded by $2\Delta_T$ by Lemma~\ref{lemma_sen_quan} along with triangle inequality. Therefore $M_0-r_{1-\alpha_\star}$ is $(\epsilon,\delta)$-DP, and $\mathbbm{1}(M_0 > r_{1-\alpha_\star}) \mathbbm{1}(\alpha \geq {1}/{n+1})$ is a post-processing quantity. The desired statement (i) follows as post-processing does not leak more information \citep{Dwork2006CalibratingNT}. 
    Given that $\mcX^{\bphi_i}_n$ have a group structure and $\mcX_n$ is exchangeable, it holds that $T(\mcX^{\bphi_i}_n)$ and $M_i$ also have a group structure and are exchangeable. 
    Statement (ii) follows immediately from \citet[Proposition~B.4]{Pfister2018KernelbasedTF}. The first part of statement (iii) is a direct result of \citet[Lemma~S.4]{Kim2026DPP}.

    For the second part of (iii), we denote $\lim_{n \to \infty} \mbP_P(M_0 \leq M_1)=p>\alpha$, and assume for $n>N$ we have $\mbP_P(M_0 \leq M_1)>(\alpha+p)/2$.
    We abbreviate $B_n$ as $B$, such that $B>\alpha^{-1}-1$ and may be fixed, diverge with $n$ or be non-monotone. 
    Noting that $\mathbbm{1}(M_{0} \leq M_{i})$ for $i\in[B]$ are i.i.d.\ random variables conditional on $\mcX_n$, the power is 
    \begin{equation}\label{eq_24}
    \begin{aligned}
        &\mathbb{P}_P\bigg[ \frac{1}{B+1} \bigg( \sum_{i=1}^B \mathbbm{1}(M_{0} \leq M_{i}) + 1 \bigg) \leq \alpha \bigg] \\
        =\ &\mathbb{P}_P\bigg[ \sum_{i=1}^B \mathbbm{1}(M_{0} \leq M_{i}) \leq \alpha(B+1) - 1 \bigg]\\
        =\ &\mbE\bigg[
        \mathbb{P}\bigg\{ \sum_{i=1}^B \mathbbm{1}(M_{0} \leq M_{i}) \leq \alpha(B+1) - 1 \mid \mcX_n\bigg\}
        \bigg]\\        
        =\ &\mbE\bigg[\sum_{k=0}^{\lfloor\alpha(B+1)-1\rfloor}C_B^k\mbP(M_0 \leq M_1\mid \mcX_n)^k\mbP(M_0 > M_1\mid \mcX_n)^{B-k}\bigg]\\
        =\ &\int_{{\mcA_n}\cup{\mcA_n^c}}\sum_{k=0}^{\lfloor\alpha(B+1)-1\rfloor}C_B^k\mbP(M_0 \leq M_1\mid \mcX_n)^k\mbP(M_0 > M_1\mid \mcX_n)^{B-k}\dd\mu(\mcX_n)\\
        \leq\ & \int_{{\mcA_n}}\sum_{k=0}^{\lfloor\alpha(B+1)-1\rfloor}C_B^k\mbP(M_0 \leq M_1\mid \mcX_n)^k\mbP(M_0 > M_1\mid \mcX_n)^{B-k}\dd\mu(\mcX_n)+\mu({\mcA_n^c}), 
    \end{aligned}
    \end{equation}
    where ${\mcA_n}$ is defined as 
    $$
    {\mcA_n}:=\{\mcX_n:\mbP(M_0 \leq M_1\mid\mcX_n)>(\alpha+p)/2\}.
    $$
    When $n>N$ we have $\mu({\mcA_n})>0$ as $\mbP_P(M_0 \leq M_1)>(\alpha+p)/2$. For any $\mcX_n\in{\mcA_n}$, we denote $\mbP(M_0 \leq M_1\mid \mcX_n)=p'>(\alpha+p)/2$. 
    By denoting $X\sim\textrm{Bin}(B,p')$, for any $t>0$ we have
    \begin{align*}
        \sum_{k=0}^{\lfloor\alpha(B+1)-1\rfloor}C_B^k\mbP(M_0 \leq M_1\mid \mcX_n)^k\mbP(M_0 > M_1\mid \mcX_n)^{B-k}&\leq\mbP(X\leq\alpha B\mid \mcX_n)\\
        &\leq e^{-t\alpha B}\cdot\mbE[e^{tX}\mid \mcX_n]\\
        &=e^{-t\alpha B}\cdot(p'e^t+(1-p'))^B.
    \end{align*}
    Optimizing over $t$, the last quantity attains its minimum at $t^*=\ln(\alpha(1-p)/(p(1-\alpha))$. Thus we have Chernoff bound for $B$ as 
    $$
    \mbP(X\leq\alpha B)\leq \exp\bigg\{-B\bigg(\alpha\ln\frac{\alpha}{p'}+(1-\alpha)\ln\frac{1-\alpha}{1-p'}\bigg)\bigg\}<c<1,
    $$
    where $\alpha\ln\frac{\alpha}{p'}+(1-\alpha)\ln\frac{1-\alpha}{1-p'}>0$ as $p'>(\alpha+p)/2$, and $c$ is some constant not depending on $n$ and $p'$. Hence the probability is less than $1$ as $p'>\alpha$ and $B+1>\alpha^{-1}$. The probability above along with \eqref{eq_24} yields 
    $$
    \mathbb{P}_P\bigg[ \frac{1}{B+1} \bigg( \sum_{i=1}^B \mathbbm{1}(M_{0} \leq M_{i}) + 1 \bigg) \leq \alpha \bigg]<c\mu({\mcA_n})+\mu({\mcA_n^c})<1.
    $$
    Taking $n\to\infty$ yields the desired result. 
\end{proof}

\subsection{Proof for Proposition~\ref{sen_dhsic} and Examples} \label{app_pr_sen}
We briefly overview this section here. We establish the upper bound for the sensitivity of permutation dHSIC in Appendix~\ref{app_a21}, and the lower bound for it in Appendix~\ref{app_a22}.
The lower bound for the sensitivity of bootstrap dHSIC is given in Appendix~\ref{app_a23}, and the inconsistency of DP bootstrap dHSIC is included in Appendix~\ref{app_a24}. 

\subsubsection{Upper Bound for Permutation dHSIC}\label{app_a21}
\begin{proof}
We first demonstrate the upper bound for the sensitivity. Without loss of generality, we consider the dataset $\mcX_n=\{(X_1^1,\dots,X_1^d),\dots,(X_n^1,\dots,X_n^d)\}$ and its neighboring dataset is 
$$
\widetilde{\mcX}_n=\{({X_1^1}',\dots,{X_1^d}'),\dots,(X_n^1,\dots,X_n^d)\}:=\{(\widetilde X_1^1,\dots,\widetilde X_1^d),\dots,(\widetilde X_n^1,\dots,\widetilde X_n^d)\},
$$
where $({X_1^1}',\dots,{X_1^d}')$ is a copy of $({X_1^1},\dots,{X_1^d})$ independent of everything else. That is, $\widetilde \mcX_n$ and $\mcX_n$ only differ in their first component. 

For a given permutation $\bpi$, $\mcX_n^{\bpi}=\{(X_{\pi_1^1}^1,\dots, X_{\pi_1^d}^d),\dots, (X_{\pi_n^1}^1,\dots, X_{\pi_n^d}^d)\}$, where each $(\pi_1^j,\dots,\pi_n^j)$ is a permutation on $[n]$ for $j=1,\dots,d$. Such permutation on $\widetilde{\mcX}_n$ leads to $\widetilde{\mcX}_n^{\bpi}$. For simplicity, we denote $k^j(x^j,\cdot)=\psi^j(x^j)$ for $j=1,\dots,d$. We further write the sample mean of $\psi^j(X_1^j),\dots,\psi^j(X_n^j)$ as $\bar\psi^j$, and the sample mean of $\psi^j(\widetilde X_1^j),\dots,\psi^j(\widetilde X_n^j)$ as $\widetilde{\psi}^j$, $j=1,\dots,d$. Then we can establish the connection between $\widehat{\dHSIC}(\mcX^{\bpi}_n)$ and $\hat{\dHSIC}(\mcX_n)$ as:
\begin{align}\label{connection}
  \hat{\dHSIC}(\mcX_n)
=&\sup_{f\in\mcF_{\bk}}\bigg\{
\frac{1}{n}\sum_{i=1}^nf(X^1_{\pi^1_i},\dots, X^d_{\pi^d_i})
-\frac{1}{n^d}\sum_{i_1,\dots, i_d}^nf(X^1_{\pi^1_{i_1}},\dots, X^d_{\pi^d_{i_d}})
\bigg\}\nonumber\\
=&\sup_{f\in\mcF_{\bk}}
\bigg\langle f, 
\frac{1}{n}\sum_{i=1}^n\prod_{j=1}^d \psi^j(X_{\pi^j_i}^j)
-\frac{1}{n^d}\prod_{j=1}^d\sum_{i=1}^n\psi^j(X^j_{\pi^j_i}) \bigg\rangle_{\mcH_{\bk}}\nonumber\\
=&\sup_{f\in\mcF_{\bk}}
\bigg\langle f, 
\frac{1}{n}\sum_{i=1}^n\prod_{j=1}^d \psi^j(X_{\pi^j_i}^j)
-\frac{1}{n^d}\prod_{j=1}^d\sum_{i=1}^n\psi^j(X^j_{i}) \bigg\rangle_{\mcH_{\bk}}\nonumber\\
:=&\sup_{f\in\mcF_{\bk}}
\langle f, 
S_1-S_2 \rangle_{\mcH_{\bk}}. 
\end{align}

First, we prove $S_2$ can be represented as \eqref{S2} for $d\geq 2$ by  mathematical induction. 
\begin{align}\label{S2}
    S_2=& \frac{1}{n^d}\prod_{j=1}^d\sum_{i=1}^n \psi^j(\tilde X^j_{i})
    + \frac{1}{n}( \psi^1(X^1_{1}) -\psi^1(\tilde X^1_{1}) )  \tilde\psi^2\dots\tilde\psi^d\nonumber\\
    &+ \frac{1}{n}\bar\psi^1( \psi^2(X^2_{1}) -\psi^2(\tilde X^2_{1}) )  \tilde\psi^3\dots\tilde\psi^d 
    +\dots+ \frac{1}{n}\bar\psi^1\dots\bar\psi^{d-1}( \psi^d(X^d_{1}) -\psi^d(\tilde X^d_{1}) ). 
\end{align}
In fact, for the case with $d=2$, we have
\begin{align*}
    S_2=& \frac{1}{n^2}\sum_{i_1,i_2=1}^n\psi^1(X^1_{i_1})\psi^2(X^2_{i_2}) \\
    =& \frac{1}{n^2}\sum_{i_1,i_2=1}^n \{\psi^1(\tilde X^1_{i_1})+(\psi^1(X^1_{i_1}) -\psi^1(\tilde X^1_{i_1}))\} 
    \{\psi^2(\tilde X^2_{i_2})+(\psi^2(X^2_{i_2}) -\psi^2(\tilde X^2_{i_2}))\}\\
    =& \frac{1}{n^2}\sum_{i_1,i_2=1}^n \{ \psi^1(\tilde X^1_{i_1})\psi^2(\tilde X^2_{i_2}) 
    + (\psi^1(X^1_{i_1}) -\psi^1(\tilde X^1_{i_1}))\psi^2(\tilde X^2_{i_2})
    +\psi^1(X^1_{i_1}) (\psi^2(X^2_{i_2}) -\psi^2(\tilde X^2_{i_2}))\}\\
    =& \frac{1}{n^2}\sum_{i_1,i_2=1}^n  \psi^1(\tilde X^1_{i_1})\psi^2(\tilde X^2_{i_2}) 
    + (\psi^1(X^1_{1}) -\psi^1(\tilde X^1_{1}))\tilde\psi^2
    +\bar\psi^1 (\psi^2(X^2_{1}) -\psi^2(\tilde X^2_{1})), 
\end{align*}
which completes the induction foundation. 

Second, suppose \eqref{S2} holds for a certain $d\geq2$, then we are to prove it also holds for $d+1$. In fact, simple calculation leads to 
\begin{align*}
   S_2 =& \frac{1}{n^{d+1}}\prod_{j=1}^{d+1}\sum_{i=1}^n\psi^j(X^j_{i})
    =\frac{1}{n^{d}}\prod_{j=1}^{d}\sum_{i=1}^n\psi^j(X^j_{i})\cdot
   \frac{1}{n}\sum_{i=1}^n\psi^{d+1}(X^{d+1}_{i})\\
   =& \frac{1}{n^{d}}\prod_{j=1}^{d}\sum_{i=1}^n\psi^j(X^j_{i})\cdot
   \frac{1}{n}\sum_{i=1}^n\{\psi^{d+1}(\tilde X^{d+1}_{i})+(\psi^{d+1}(X^{d+1}_{i}) -\psi^{d+1}(\tilde X^{d+1}_{i}))\}\\
   =& \frac{1}{n^{d}}\prod_{j=1}^{d}\sum_{i=1}^n\psi^j(X^j_{i})\cdot\tilde\psi^{d+1}
    + 
    \prod_{j=1}^{d}\bar\psi^j\cdot
   \frac{1}{n}(\psi^{d+1}(X^{d+1}_{i}) -\psi^{d+1}(\tilde X^{d+1}_{i}))\\
   \overset{(*)}{=}& \frac{1}{n^{d+1}}\prod_{j=1}^{d+1}\sum_{i=1}^n \psi^j(\tilde X^j_{i})
    + \frac{1}{n}( \psi^1(X^1_{1}) -\psi^1(\tilde X^1_{1}) )  \tilde\psi^2\dots\tilde\psi^{d+1}\\
    &+ \frac{1}{n}\bar\psi^1( \psi^2(X^2_{1}) -\psi^2(\tilde X^2_{1}) )  \tilde\psi^3\dots\tilde\psi^{d+1} 
    +\dots+ \frac{1}{n}\bar\psi^1\dots\bar\psi^{d}(\psi^{d+1}(X^{d+1}_{1}) -\psi^{d+1}(\tilde X^{d+1}_{1}) ), 
\end{align*}
where step $(*)$ uses that induction assumption that \eqref{S2} holds for $d$. In all, we complete the mathematical induction and thus \eqref{S2} holds for any $d\geq2$.


Second, we expand $S_1$ by adding and subtracting the same term. Note that $\psi^j(X_{\pi^j_i}^j)= \psi^j(\tilde X _{\pi^j_i}^j)$ holds unless subscript satisfies $\pi^j_i=1$. Let there be $d'\leq d$ different $i$ such that $\pi^j_i=1$ for $j\in[d]$. $[d]$ is divided into $d'$ disjoint sets $C_1,\dots,C_{d'}$ by the values of those $i$s.  $j_1$ and $j_2$ are in the same set if and only if $\pi^{j_1}_i =\pi^{j_2}_i=1$ for a certain $i$. 
Without loss of generality, 
$C_k=\{j:\pi^j_k=1\}$, $k=1,\dots,d'$. 

With these notations in mind, we expand $S_1$ as: 
\begin{align}\label{S1}
    S_1 =& \frac{1}{n}\sum_{i=1}^n\prod_{j=1}^d \{\psi^j(\tilde X_{\pi^j_i}^j)+ (\psi^j(X_{\pi^j_i}^j)-\psi^j(\tilde X _{\pi^j_i}^j))\}\nonumber\\
    =& \frac{1}{n}\sum_{i=1}^{d'}\prod_{j\notin C_i}\psi^j(\tilde X_{\pi^j_i}^j)\prod_{j\in C_i} \{\psi^j(\tilde X_{\pi^j_i}^j)+ (\psi^j(X_{\pi^j_i}^j)-\psi^j(\tilde X _{\pi^j_i}^j))\}+\frac{1}{n}\sum_{i=d'+1}^{n}\prod_{j=1}^d\psi^j(\tilde X_{\pi^j_i}^j)\nonumber\\
    \overset{(i)}{=}& \frac{1}{n}\sum_{i=1}^{d'}\prod_{j\notin C_i}\psi^j(\tilde X_{\pi^j_i}^j)
    \bigg\{\prod_{j\in C_i} \psi^j(\tilde X_{\pi^j_i}^j)+ \sum_{j\in C_i} \{(\psi^j(X_{\pi^j_i}^j)-\psi^j(\tilde X _{\pi^j_i}^j))
    \prod_{\substack{j'\in C_i\\j'\neq j}}\psi^{j'}(\dot X_{\pi^{j'}_i}^{j'})\}
    \bigg\}\nonumber\\
    &+\frac{1}{n}\sum_{i=d'+1}^{n}\prod_{j=1}^d\psi^j(\tilde X_{\pi^j_i}^j)\nonumber\\
    =& \frac{1}{n}\sum_{i=1}^n\prod_{j=1}^d\psi^j(\tilde X_{\pi^j_i}^j) + 
    \frac{1}{n}\bigg\{ \sum_{i=1}^{d'}\sum_{j\in C_i}(\psi^j(X_{1}^j)-\psi^j(\tilde X _{1}^j))
    \prod_{j'\neq j}\psi^{j'}(\dot X_{\pi^{j'}_i}^{j'})
    \bigg\}\nonumber\\
    \overset{(ii)}{=}& \frac{1}{n}\sum_{i=1}^n\prod_{j=1}^d\psi^j(\tilde X_{\pi^j_i}^j) + 
    \frac{1}{n}\bigg\{ \sum_{j=1}^{d}(\psi^j(X_{1}^j)-\psi^j(\tilde X _{1}^j))
    \prod_{j'\neq j}\psi^{j'}(\dot X_{\tau(j,j')}^{j'})
    \bigg\}. 
\end{align}
Here, the notation $\psi^{j'}(\dot X_{\pi^{j'}_i}^{j'})$ denotes either $\psi^{j'}(\tilde X_{\pi^{j'}_i}^{j'})$ or $\psi^{j'}(X_{\pi^{j'}_i}^{j'})$, which makes no difference when deriving the upper bound. Step $(i)$ uses similar techniques in the proof of \eqref{S2} to handle the product. Step $(ii)$ rewrites the expression based on $j$, where $\tau(j,j')$ is the subscript that depends on $j,j'$. 

Combining the \eqref{S2} with \eqref{S1} and using $\dot\psi^j$ to denote either $\bar\psi^j$ or $\tilde\psi^j$, it follows that 
\begin{align*}
    S_1-S_2= &
    \frac{1}{n}\sum_{i=1}^n\prod_{j=1}^d\psi^j(\tilde X_{\pi^j_i}^j) 
    -\frac{1}{n^d}\prod_{j=1}^d\sum_{i=1}^n \psi^j(\tilde X^j_{i})\\
    &+ \frac{1}{n} \sum_{j=1}^{d}(\psi^j(X_{1}^j)-\psi^j(\tilde X _{1}^j))
    \bigg\{
    \prod_{j'\neq j}\psi^{j'}(\dot X_{\tau(j,j')}^{j'})
    - \prod_{j'\neq j}\dot\psi^{j'} 
    \bigg\}.
\end{align*}
The inner product defined on ${\mcH_{\bk}}$ between $f$ and the result of the first line equals to $\widehat{\dHSIC}(\widetilde{\mcX}^{\bpi}_n)$. Recalling \eqref{connection}, the  
the fact $\|f\|_{\mcH_{\bk}} \leq1$ along with Cauchy--Schwartz inequality yields that 
\begin{align*}
    &|\widehat{\dHSIC}(\mcX^{\bpi}_n)-\widehat{\dHSIC}(\widetilde{\mcX}^{\bpi}_n)|\\
    \leq &
    \sup_{f\in\mcF_{\bk}}\bigg|\bigg\langle f, 
    \frac{1}{n} \sum_{j=1}^{d}(\psi^j(X_{1}^j)-\psi^j(\tilde X _{1}^j))
    \bigg\{
    \prod_{j'\neq j}\psi^{j'}(\dot X_{\tau(j,j')}^{j'})
    - \prod_{j'\neq j}\dot\psi^{j'} 
    \bigg\}
    \bigg\rangle_{\mcH_{\bk}}\bigg|\\
    \leq & \frac{1}{n} \sum_{j=1}^{d}
    \sup_{f\in\mcF_{\bk}}\bigg|\bigg\langle f, 
    (\psi^j(X_{1}^j)-\psi^j(\tilde X _{1}^j))
    \bigg\{
    \prod_{j'\neq j}\psi^{j'}(\dot X_{\tau(j,j')}^{j'})
    - \prod_{j'\neq j}\dot\psi^{j'} 
    \bigg\}
    \bigg\rangle_{\mcH_{\bk}}\bigg|\\
    \leq & \frac{1}{n} \sum_{j=1}^{d}
    \bigg\|
    (\psi^j(X_{1}^j)-\psi^j(\tilde X _{1}^j))
    \bigg\{
    \prod_{j'\neq j}\psi^{j'}(\dot X_{\tau(j,j')}^{j'})
    - \prod_{j'\neq j}\dot\psi^{j'} 
    \bigg\} \bigg\|_{\mcH_{\bk}}\\
    =& \frac{1}{n} \sum_{j=1}^{d}
    \|
    \psi^j(X_{1}^j)-\psi^j(\tilde X _{1}^j)\|_{\mcH_{k^j}}
    \bigg\|
    \prod_{j'\neq j}\psi^{j'}(\dot X_{\tau(j,j')}^{j'})
    - \prod_{j'\neq j}\dot\psi^{j'} 
    \bigg\|_{\mcH_{\otimes_{j\neq j'} k^{j'}}}\\
    \overset{(iv)}{\leq} & \frac{1}{n} \sum_{j=1}^{d} (2K^j)^{1/2}\bigg(2\prod_{j'\neq j}K^{j'}\bigg)^{1/2}=\frac{2d}{n}\bigg(\prod_{j=1}^dK^j\bigg)^{1/2}.
\end{align*}
In the above calculation, Step $(iv)$ follows due to the fact that 
\begin{align*}
    \|
    \psi^j(x^j)-\psi^j({x^j}')\|_{\mcH_{k^j}}&=\sqrt{k^j(x^j,x^j)+k^j({x^j}',{x^j}')-2k^j(x^j,{x^j}')}\leq \sqrt{2K^j}\quad \text{and}\\
    \bigg\|
    \prod_{j}\psi^{j}({x^j})
    - \prod_{j}{\psi^j}({x^j}')
    \bigg\|_{\mcH_{\bk}}
    &\leq
    \sqrt{\bk(\prod_{j}{x^j},\prod_{j}{x^j})+\bk(\prod_{j}{x^j}',\prod_{j}{x^j}')}\\
    &= \sqrt{\prod_{j}k^j({x^j},{x^j})+\prod_{j}k^j({x^j}',{x^j}')} \leq \sqrt{2\prod_{j}K^j}.
\end{align*}
\end{proof}

Up to now, we have completed the proof of the upper bound the permutation dHSIC. It is worth remarking that, a sharper inequality in Step $(iv)$ can be derived by discussing the second term more meticulously. One can reformulate the expression and expand the averages to identify and cancel out identical terms, leading to a sharper bound. 
However, improvements can be achieved incrementally, with a constant factor sharpened by subtracting $O(n^{-d})$, while the discussion may be rather miscellaneous. 
Therefore, we only state a briefer and more elegant result in this paper.

\subsubsection{Lower Bounds for the Permutation dHSIC}\label{app_a22}
Next we construct a lower bound for sensitivity of permutation dHSIC. 
Further assume that kernels $k^j$ are translation invariant, and have non-empty level sets on their domain. Here, the kernel $k^j$ has non-empty level sets on $\mbX^j$ if, for any $\varepsilon\in(0,K^j)$, there exist $x,y\in\mbX^j$ such that $k^j(x,y)\leq\varepsilon$. Both the Gaussian kernel and the Laplace kernel have non-empty level sets on $\mbR^d$ due to the continuity of the kernel function. 

We first reduce the general $d$-variable case to the bivariate case. 
Let $\Delta_T^d$ denote the sensitivity of dHSIC for $d$ random vectors with kernels $k^1,\dots,k^d$, and define $\Delta_T^j$ similarly for $j\in[d]$. 
For a translation invariant kernel $k^d$, write $k^d(x,y)=\kappa^d(x-y)$. Since $K^d$ is the uniform upper bound and $k^d(x,x)=\kappa^d(0)$ for all $x\in\mbX^d$, we may choose the bound such that $k^d(x,x)=K^d$. 
Take two neighboring datasets in the $(d-1)$-variable problem whose dHSIC values differ by nearly $\Delta_T^{d-1}$, and append to both datasets the same constant $d$th coordinate. 
In the closed-form expression of dHSIC in Lemma~\ref{le_sq_dhsic}, every kernel factor from the appended coordinate is then equal to $K^d$, so the squared dHSIC difference is multiplied by $K^d$. 
Equivalently, for our square-root dHSIC statistic,
\[
    \Delta_T^d\geq \sqrt{K^d}\Delta_T^{d-1}.
\]
Iterating this reduction shows that it is enough to construct the lower bound for $d=2$. 
The bivariate construction below gives $\Delta_T^2\geq4(n-2.5)n^{-2}\sqrt{K^1K^2}$, and hence $\Delta_T^d\gtrsim K_0^{1/2}/n$.

Without loss of generality, we consider $d=2$ and $k^1$ and $k^2$ are translation invariant with non-empty level sets on $\mbX^1$ and $\mbX^2$, respectively. For a given $\varepsilon\in(0,\min\{
K^1,K^2\})$, we assume that there exist $x^i_a,x^i_b\in\mbX^i$ such that $k^i(x^i_a,x^i_b)=\varepsilon_i\leq\varepsilon$ and $k^i(x^i,x^i)=K^i$ for all $x^i$ and $i=1,2$. Consider datasets
$$
\mathcal{X}_n = \begin{bmatrix} 
x^1_a & x^2_a \\ 
x^1_a & x^2_a \\ 
x^1_b & x^2_b \\ 
\vdots & \vdots \\ 
x^1_b & x^2_b 
\end{bmatrix} 
\quad \text{and} \quad 
\tilde{\mathcal{X}}_n = \begin{bmatrix} 
x^1_a & x^2_a \\ 
x^1_b & x^2_b \\ 
x^1_b & x^2_b \\ 
\vdots & \vdots \\ 
x^1_b & x^2_b 
\end{bmatrix},
$$
where $d_{\ham}(\mcX_n,\tilde{\mcX}_n)=1$. Consider a permutation $\bpi$ resampling $\bpi=(2,1,3,4,\dots,n)$ and corresponding permuted datasets are
$$
\mathcal{X}_n = \begin{bmatrix} 
x^1_a & x^2_a \\ 
x^1_a & x^2_a \\ 
x^1_b & x^2_b \\ 
\vdots & \vdots \\ 
x^1_b & x^2_b 
\end{bmatrix} 
\quad \text{and} \quad 
\tilde{\mathcal{X}}_n = \begin{bmatrix} 
x^1_a & x^2_b \\ 
x^1_b & x^2_a \\ 
x^1_b & x^2_b \\ 
\vdots & \vdots \\ 
x^1_b & x^2_b 
\end{bmatrix},
$$
Using the closed-form expression of the squared dHSIC in Lemma~\ref{le_sq_dhsic}, one can calculate that 
\begin{align*}
    \hat{\dHSIC}^2(\mcX_n^{\bpi})=&\frac{4+(n-2)^2}{n^2}K^1K^2+\frac{4(n-2)\varepsilon_1\varepsilon_2}{n^2}\\
    &+\bigg\{\frac{4+(n-2)^2}{n^2}K^1+\frac{4(n-2)\varepsilon_1}{n^2}\bigg\}
    \bigg\{\frac{4+(n-2)^2}{n^2}K^2+\frac{4(n-2)\varepsilon_2}{n^2}\bigg\}\\
    &-4\bigg\{
    \frac{4}{n^3}K^1K^2+\frac{2(n-2)\varepsilon_1K^2}{n^3}+\frac{2(n-2)\varepsilon_2K^1}{n^3}+\frac{(n-2)^2\varepsilon_1\varepsilon_2}{n^3}\bigg\}\\
    &-2(n-2)\bigg\{
    \frac{(n-2)^2}{n^3}K^1K^2+\frac{2(n-2)\varepsilon_1K^2}{n^3}+\frac{2(n-2)\varepsilon_2K^1}{n^3}+\frac{4\varepsilon_1\varepsilon_2}{n^3}\bigg\}\\
    =&\frac{16(n-2)^2}{n^4}K^1K^2+C_1\varepsilon_1\varepsilon_2+C_2\varepsilon_1+C_3\varepsilon_2, 
\end{align*}
where $C_1, C_2, C_3$ are some constants that only depend on $K^1, K^2, n$. A similar calculation shows
$$
\hat{\dHSIC}^2(\tilde{\mcX}_n^{\bpi})=\frac{4}{n^4}K^1K^2+C_1'\varepsilon_1\varepsilon_2+C_2'\varepsilon_1+C_3'\varepsilon_2.
$$
Combining them together yields that
$$
|\hat{\dHSIC}({\mcX}_n^{\bpi})-\hat{\dHSIC}(\tilde{\mcX}_n^{\bpi})|\geq\frac{4(n-2.5)}{n^2}\sqrt{K^1K^2}+h(\varepsilon_1,\varepsilon_2,K^1,K^2,n),
$$
where $h(\varepsilon_1,\varepsilon_2,K^1,K^2,n)$ is some function of $\varepsilon_1,\varepsilon_2,K^1,K^2,n$, which goes to zero as $\varepsilon_1\to0$ and $\varepsilon_2\to0$ for any fixed $K^1,K^2,n$. By letting $\varepsilon_1,\varepsilon_2\to0$, the sensitivity of permutation dHSIC is at least $4(n-2.5)n^{-2}\sqrt{K^1K^2}$ for any fixed $n$. It follows that the sensitivity of dpdHSIC $\Delta_T$ satisfying $\Delta_T\gtrsim K_0^{1/2}/n$. 

\subsection{Proof for Theorem~\ref{property_dpdH}}\label{app_property_dpdH}
\begin{proof} 
    (i) and (ii) are direct results from (i) and (ii) in Lemma~\ref{lemma_dp_re}, respectively.
    To demonstrate (iii), by (iii) in Lemma~\ref{lemma_dp_re}, it is sufficient to verify $\lim_{n\to\infty}\mbP(M_{0,n}\leq M_{1,n})=0$, where
    \begin{align*}
        M_{i,n}&:=M_i=\hat{\dHSIC}(\mathcal{X}_n^{\bpi_i})+\frac{4d\sqrt{K_0}}{n\xi_{\epsilon,\delta}}\zeta_i\quad\text{for } i\in\{0,1\}.
    \end{align*}
    In this theorem, we hold that the upper bound for kernel products $K_0$ keeps as a constant.    
    Since $1/(n\xi_{\epsilon,\delta})\to0$ as $n\to\infty$, then $4d\sqrt{K_0}/(n\xi_{\epsilon,\delta})\zeta_i=o_P(1)$, 
    which leads to $M_{i,n}\overset{p}{\to}\hat{\dHSIC}(\mathcal{X}_n^{\bpi_i})$ as $n\to\infty$. 
    Recalling that \citet[Proposition~1]{Pfister2018KernelbasedTF}, the $d$ groups are not jointly independent if and only if the population $\dHSIC>0$.     
    Therefore, it is sufficient 
    to verify $\hat{\dHSIC}(\mcX_n^{\bpi_1})\overset{p}{\to}0$ and $\hat{\dHSIC}(\mcX_n^{\bpi_0})\overset{p}{\to}\dHSIC_{\bk}(P_{X^1,\dots, X^d})$ to complete the proof. 
    In fact, Lemma~\ref{le_ineq_e_dhsic} indicates that the unpermuted dHSIC statistic satisfies $\hat{\dHSIC}(\mcX_n^{\bpi_0})\overset{p}{\to}\dHSIC_{\bk}(P_{X^1,\dots, X^d})$, while a direct application to Lemma~\ref{le_dhsic_con_per} yields that $\hat{\dHSIC}(\mcX_n^{\bpi_1})\overset{p}{\to}0$. By the definition of convergence in probability, we have
    $$
    \lim_{n\to\infty}\mbP(M_{0,n}\leq M_{1,n})=\lim_{n\to\infty}\mbP(\dHSIC_{\bk}(P_{X^1,\dots, X^d})\leq0)=0,
    $$
    which completes the proof. 
\end{proof}

\subsection{Proof of Proposition~\ref{prop_bootstrap_sen}}\label{app_a23}
\begin{proof}
We first show that it suffices to construct the example for $d=2$. 
Let $\Delta_{\bb}^d$ denote the sensitivity of bootstrap dHSIC for $d$ random vectors with kernels $k^1,\dots,k^d$, and define $\Delta_{\bb}^2$ similarly for the first two kernels. 
Take two neighboring datasets in the bivariate construction and append the same constant coordinates to dimensions $3,\dots,d$ in both datasets. 
For each appended coordinate $j\geq3$, the corresponding kernel factor equals $k^j(x^j,x^j)=K^j$ throughout the closed-form expression in Lemma~\ref{le_sq_dhsic}. 
Therefore, the squared bootstrap dHSIC discrepancy in the $d$-variable construction is multiplied by $\prod_{j=3}^dK^j$, and hence the square-root statistic satisfies
\[
    \Delta_{\bb}^d\geq \bigg(\prod_{j=3}^dK^j\bigg)^{1/2}\Delta_{\bb}^2.
\]
Thus the bivariate lower bound below immediately yields the desired lower bound $3K_0^{1/2}/8$ for general $d$.

Without loss of generality, we consider $d=2$ and $k^1$ and $k^2$ are translation invariant with non-empty level sets on $\mbX^1$ and $\mbX^2$, respectively. For a given $\varepsilon\in(0,\min\{
K^1,K^2\})$, we assume that there exist $x^i_a,x^i_b\in\mbX^i$ such that $k^i(x^i_a,x^i_b)=\varepsilon_i\leq\varepsilon$ and $k^i(x^i,x^i)=K^i$ for all $x^i$ and $i=1,2$. Consider datasets
$$
\mathcal{X}_n = \begin{bmatrix} 
x^1_a & x^2_a \\ 
x^1_a & x^2_a \\ 
x^1_b & x^2_b \\ 
\vdots & \vdots \\ 
x^1_b & x^2_b 
\end{bmatrix} 
\quad \text{and} \quad 
\tilde{\mathcal{X}}_n = \begin{bmatrix} 
x^1_a & x^2_a \\ 
x^1_b & x^2_b \\ 
x^1_b & x^2_b \\ 
\vdots & \vdots \\ 
x^1_b & x^2_b 
\end{bmatrix},
$$
where $d_{\ham}(\mcX_n,\tilde{\mcX}_n)=1$ and $n=4m$ for some $m\in\mathbb{Z}^+$ for simplicity. Consider a bootstrap $\bb$ resampling $((1,1),(1,2),(2,1),(3,3))$ for $m$ times, and corresponding datasets are
$$
\mathcal{X}_n^{\bb} = \begin{bmatrix} 
x^1_a & x^2_a \\ 
x^1_a & x^2_a \\ 
x^1_a & x^2_a \\ 
x^1_b & x^2_b \\
\vdots & \vdots 
\end{bmatrix} 
\quad \text{and} \quad 
\tilde{\mathcal{X}}_n^{\bb} = \begin{bmatrix} 
x^1_a & x^2_a \\ 
x^1_a & x^2_b \\ 
x^1_b & x^2_a \\ 
x^1_b & x^2_b \\
\vdots & \vdots 
\end{bmatrix},
$$
where the whole datasets are given by the first four data points repeating $m$ times.
Using the closed-form expression of the squared dHSIC in Lemma~\ref{le_sq_dhsic}, one can calculate that 
\begin{align*}
    \hat{\dHSIC}^2(\mcX_n^{\bb})=&\frac{m^2+(3m)^2}{n^2}K^1K^2+\frac{2m\cdot3m\varepsilon_1\varepsilon_2}{n^2}\\
    &+\bigg\{\frac{m^2+(3m)^2}{n^2}K^1+\frac{2m\cdot3m\varepsilon_1}{n^2}\bigg\}
    \bigg\{\frac{m^2+(3m)^2}{n^2}K^2+\frac{2m\cdot3m\varepsilon_2}{n^2}\bigg\}\\
    &-2m\bigg\{
    \frac{m^2}{n^3}K^1K^2+\frac{2m\cdot3m\varepsilon_1K^2}{n^3}+\frac{2m\cdot3m\varepsilon_2K^1}{n^3}+\frac{(3m)^2\varepsilon_1\varepsilon_2}{n^3}\bigg\}\\
    &-6m\bigg\{
    \frac{(3m)^2}{n^3}K^1K^2+\frac{2m\cdot3m\varepsilon_1K^2}{n^3}+\frac{2m\cdot3m\varepsilon_2K^1}{n^3}+\frac{m^2\varepsilon_1\varepsilon_2}{n^3}\bigg\}\\
    =&\frac{9}{64}K^1K^2+C_1\varepsilon_1\varepsilon_2+C_2\varepsilon_1+C_3\varepsilon_2, 
\end{align*}
where $C_1, C_2, C_3$ are constants that only depend on $K^1, K^2, n$. Similarly, a simple calculation shows
$$
\hat{\dHSIC}^2(\tilde{\mcX}_n^{\bb})=0.
$$
Combining them together yields that
$$
|\hat{\dHSIC}({\mcX}_n^{\bb})-\hat{\dHSIC}(\tilde{\mcX}_n^{\bb})|\geq\frac{3}{8}\sqrt{K^1K^2}+h(\varepsilon_1,\varepsilon_2,K^1,K^2,n),
$$
where $h(\varepsilon_1,\varepsilon_2,K^1,K^2,n)$ is some function of $\varepsilon_1,\varepsilon_2,K^1,K^2,n$, which goes to zero as $\varepsilon_1\to0$ and $\varepsilon_2\to0$ for any fixed $K^1,K^2,n$. By letting $\varepsilon_1,\varepsilon_2\to0$, the sensitivity of bootstrap dHSIC is at least $3\sqrt{K_0}/8$ for any fixed $n$, which is significantly larger than $4\sqrt{K_0}/n$ for permutation dHSIC.  
\end{proof}

The primary reason why permutation dHSIC exhibits manageable sensitivity lies in the fact that the sum of data points remains constant after permutation. This invariance effectively eliminates certain terms that could otherwise contribute to increased sensitivity. In contrast, bootstrap dHSIC lacks this property, leading to unbounded sensitivity.
However, it does not necessarily imply that permutation holds a universal advantage over bootstrap in the context of differential privacy. In certain scenarios, both approaches enjoy equivalent sensitivity. For instance, a simple calculation can show that the sensitivity of bootstrap MMD \citep{Gretton2012AKT} matches that of permutation MMD \citep[Lemma~2]{Kim2026DPP}. A more comprehensive comparison between these two techniques remains an open avenue for future research.

\subsection{Proof of Proposition~\ref{prop_bootstrap_inconsistency}}\label{app_a24}
\begin{proof}
In this subsection, we first demonstrate that the DP bootstrap dHSIC is inconsistent against some fixed alternatives, and next further show that it is inconsistent against all alternatives for some $\alpha,\xi_{\epsilon,\delta}$. Denote the sensitivity of $\hat\dHSIC(\mcX_n^{\bb})$ as $\Delta_T^{\bb}\geq3\sqrt{K_0}/8$ and noised $\hat\dHSIC(\mcX_n^{\bb_i})$ by \eqref{eq_22} as $M_i^{\bb}$. 
Consider an alternative $P$ close to the null, say $\dHSIC(P_{X^1,\dots,X^d})=\rho_0>0$. Then Lemma~\ref{le_ineq_e_dhsic} yields that $\hat\dHSIC(\mcX_n)\overset{p}{\to}\dHSIC(P_{X^1,\dots,X^d})=\rho_0$, and note that $\hat\dHSIC(\mcX_n^{\bb_i})\geq0$. 
Then 
\begin{align*}
   \lim_{n \to \infty} \mbP_P(M_0^{\bb_0} \leq M_1^{\bb_1})&=\lim_{n \to \infty} \mbP_P\bigg(\hat\dHSIC(\mcX_n) +\frac{2\Delta_T^{\bb}}{\xi_{\epsilon,\delta}}\zeta_0\leq \hat\dHSIC(\mcX_n^{\bb_1})+\frac{2\Delta_T^{\bb}}{\xi_{\epsilon,\delta}}\zeta_1\bigg)\\
   &=\lim_{n \to \infty} \mbP_P\bigg(\hat\dHSIC(\mcX_n)-\hat\dHSIC(\mcX_n^{\bb_1})\leq\frac{2\Delta_T^{\bb}}{\xi_{\epsilon,\delta}}\zeta_1-\frac{2\Delta_T^{\bb}}{\xi_{\epsilon,\delta}}\zeta_0\bigg)\\
   &\geq  \mbP\bigg(\frac{2\Delta_T^{\bb}}{\xi_{\epsilon,\delta}}\zeta_1-\frac{2\Delta_T^{\bb}}{\xi_{\epsilon,\delta}}\zeta_0\geq\rho_0\bigg),
\end{align*}
where the last probability can be made close to $1/2$ as the Laplace distribution is symmetric about $0$ and thus $\zeta_1-\zeta_0$ is symmetric about $0$. Take $\rho_0$ close to $0$ along with the (iii) in Lemma~\ref{lemma_dp_re}, the test is inconsistent for some fixed $P$ under $\alpha<1/2$. 

Next we show that under some conditions on $\alpha,\xi_{\epsilon,\delta}$, the test is inconsistent for all alternatives. In fact, as $\hat\dHSIC(\mcX_n)\leq\sqrt{2K_0}$ and $\hat\dHSIC(\mcX_n^{\bb_1})\geq0$, 
\begin{align*}
   \inf_P\lim_{n \to \infty} \mbP_P(M_0^{\bb_0} \leq M_1^{\bb_1})
   &=\inf_P\lim_{n \to \infty} \mbP_P\bigg(\hat\dHSIC(\mcX_n)-\hat\dHSIC(\mcX_n^{\bb_1})\leq\frac{2\Delta_T^{\bb}}{\xi_{\epsilon,\delta}}\zeta_1-\frac{2\Delta_T^{\bb}}{\xi_{\epsilon,\delta}}\zeta_0\bigg)\\
   &\geq \mbP\bigg(\sqrt{2K_0}\leq\frac{2\Delta_T^{\bb}}{\xi_{\epsilon,\delta}}\zeta_1-\frac{2\Delta_T^{\bb}}{\xi_{\epsilon,\delta}}\zeta_0\bigg)\\
   &\geq \mbP\bigg(\zeta_1-\zeta_0\geq\frac{4\sqrt{2}}{3}\xi_{\epsilon,\delta}\bigg).
\end{align*}
By denoting the last probability as $C_{\epsilon,\delta}$, (iii) in Lemma~\ref{lemma_dp_re} yields that the DP bootstrap dHSIC is inconsistent for all alternatives at any level $\alpha<C_{\epsilon,\delta}$. To be more concrete, taking $\xi_{\epsilon,\delta}=\epsilon+\log(1/(1-\delta))=1$ leads to $C_{\epsilon,\delta}>0.148$. 
\end{proof}

We close this part by pointing out the asymptotic level of DP bootstrap dHSIC. In fact, \citet[Theorem~2]{Pfister2018KernelbasedTF} shows that, under the null hypothesis, 
$$
n\hat\dHSIC(\mcX_n)\overset{d}{\to}\binom{2d}{2}\sum_{i=1}^{\infty}\lambda_iZ_i^2\quad \text{as } n\to\infty, 
$$
where $(\lambda_i)_{i\in\mathbb{N}}$ are the eigenvalues of a square integrable operator. It follows that $\hat\dHSIC(\mcX_n)=O_p(n^{-1})$. Then we have 
$$
M_0^{\bb_0}=\hat\dHSIC(\mcX_n)+\frac{2\Delta_T^{\bb}}{\xi_{\epsilon,\delta}}\zeta_0\to\frac{2\Delta_T^{\bb}}{\xi_{\epsilon,\delta}}\zeta_0 \quad \text{as } n\to\infty,
$$
provided that the privacy parameter $\xi_{\epsilon,\delta}$ is fixed. The proof is then completed by the proof for \citet[Theorem~4]{Pfister2018KernelbasedTF}.

\subsection{Proof for Theorem~\ref{th_mini_sepa_dpdhsic}}\label{app_th_mini_sepa_dpdhsic}
\begin{proof}
    With a little abuse of notations, $C_1, C_2,\dots$ refer to universal constants only associated with the upper bounds for the kernels. By the quantile representation of resampling test, the type II error can be expressed as
    $$
        \mbE[1-\phi_{\dpdHSIC}]=\mbP(\hat{p}_{\DP}>\alpha)=\mbP(M_0\leq q_{1-\alpha,B}),
    $$
    where $q_{1-\alpha,B}$ is the $1-\alpha$ quantile of $M_0,M_1,\dots,M_B$ and $M_i=\hat{\dHSIC}(\mcX^{\bpi_i}_n)+2\Delta_T\xi^{-1}_{\epsilon,\delta}\zeta_i$. The proof is divided into two steps, where the quantile is bounded in the first step and the type II error is controlled in the second step.

\textbf{Step 1 (Bounding the quantile).} We first bound the sample quantile by the population quantile. Let $q_{1-\alpha/6,\infty}$ be the limiting value of $q_{1-\alpha/6,B}$ with $B=\infty$. Then Lemma~S.20 in \cite{Kim2026DPP} yields that
$$
    q_{1-\alpha,B}\leq q_{1-\alpha/6,\infty}
$$
holds with probability at least $1-\beta/2$ provided that $B\geq2\alpha^{-1}\max\{3\log(1/\beta),1-\alpha\}$. By the property of the quantile, we have $q_{1-\alpha/6,\infty}\leq q_{1-\alpha/12,\infty}^a+q_{1-\alpha/12,\infty}^b$, where
$$q_{1-\alpha/12, \infty}^a := \inf \bigg\{ x \in \mbR : \frac{1}{|\bPi_n|} \sum_{\bpi \in \bPi_n} \mathbbm{1}(T(\mathcal{X}_n^{\bpi}) \leq x) \geq 1 - \alpha/12 \bigg\} 
\quad \text{and}$$
$$q_{1-\alpha/12, \infty}^b := 2 \Delta_T \xi_{\epsilon, \delta}^{-1} F_\zeta^{-1}(1 - \alpha/12),$$
where $F_\zeta^{-1}$ is the inverse cumulative distribution function of $\zeta\sim\Lap(0,1)$.

Next we bound $q_{1-\alpha/12, \infty}^a$ with the concentration inequality for permuted dHSIC. Lemma~\ref{le_dhsic_con_per} with simple calculations leads to $$
q_{1-\alpha/12, \infty}^a \leq C_1 \sqrt{\frac{1}{n} \max \bigg\{ \log\bigg(\frac{12}{\alpha}\bigg), \log^{1/2}\bigg(\frac{12}{\alpha}\bigg), 1 \bigg\}}.$$
Therefore, with probability at least $1-\beta/2$, it holds that
\begin{equation}\label{eq_01}
    q_{1-\alpha/6, B} \leq C_1 \sqrt{\frac{1}{n} \max \bigg\{ \log \bigg(\frac{12}{\alpha}\bigg), \log^{1/2}\bigg(\frac{12}{\alpha}\bigg), 1 \bigg\}} + 2 \Delta_T \xi_{\epsilon, \delta}^{-1} F_\zeta^{-1}(1 - \alpha/12). 
\end{equation}

\textbf{Step 2 (Bounding the type II error).} To begin with, we convert the calculations for the empirical dHSIC to the population version by Lemma~\ref{le_ineq_e_dhsic}, which demonstrates the exponential tail bound for their difference. 
Therefore, it holds that
$$
    | \widehat{\dHSIC}(\mathcal{X}_n) - \dHSIC_{\bk}(P_{X_1,\dots,X_d}) | 
    \geq C_2\sqrt{\frac{\log(8/\beta)}{n}}
$$
with probability at least $1-\beta/4$. By the definition of the c.d.f., it can be verified that
$$
    2\Delta_T\xi^{-1}_{\epsilon,\delta}\zeta_0>2 \Delta_T \xi_{\epsilon, \delta}^{-1} F_\zeta^{-1}(\beta/4)
$$
with probability being $1-\beta/4$. Combining the results above together, we have
\begin{align*}
    \mathbb{E}[1 - \phi_{\dpdHSIC}] 
    &\leq \mathbb{P} \bigg( \dHSIC_{\bk}(P_{X_1,\dots,X_d}) \leq q_{1-\alpha,B} + C_2 \sqrt{\frac{\log(8/\beta)}{n}} - 2 \Delta \xi_{\epsilon, \delta}^{-1} F_\zeta^{-1}(\beta/4) \bigg) + \frac{\beta}{2} \\
    &\leq \mathbb{P} \bigg( \dHSIC_{\bk}(P_{X_1,\dots,X_d}) \leq C_1 \sqrt{\frac{1}{n} \max \bigg\{ \log\bigg(\frac{12}{\alpha}\bigg), \log^{1/2}\bigg(\frac{12}{\alpha}\bigg), 1 \bigg\}} \\
    &\quad + 2 \Delta_T \xi_{\epsilon, \delta}^{-1} F_\zeta^{-1}(1-\alpha/12) + C_2 \sqrt{\frac{\log(8/\beta)}{n}} - 2 \Delta_T \xi_{\epsilon, \delta}^{-1} F_\zeta^{-1}(\beta/4)\bigg) + \beta.
\end{align*}
For $\alpha\in(0,1)$ and $\beta\in(0,1-\alpha)$, it holds
\begin{equation}\label{eq_08}
    \max \bigg\{ \log\bigg(\frac{12}{\alpha}\bigg), \log^{1/2}\bigg(\frac{12}{\alpha}\bigg), 1, \log\bigg(\frac{8}{\beta}\bigg)\bigg\}
    \leq C_3 \max\bigg\{\log\bigg(\frac{1}{\alpha}\bigg),\log\bigg(\frac{1}{\beta}\bigg)\bigg\}.
\end{equation}
On the other hand, simple calculations yield  
$$
    2 \Delta_T \xi_{\epsilon, \delta}^{-1} F_\zeta^{-1}(1-\alpha/12) - 2 \Delta_T \xi_{\epsilon, \delta}^{-1} F_\zeta^{-1}(\beta/4) \leq C_4 \max\bigg\{\log\bigg(\frac{1}{\alpha}\bigg),\log\bigg(\frac{1}{\beta}\bigg)\bigg\}.
$$
By combining the results above, we have
\begin{align*}
    &\quad\mathbb{E}[1 - \phi_{\dpdHSIC}] \\
    &\leq \mathbb{P} \bigg( \dHSIC_{\bk}(P_{X_1,\dots,X_d}) \leq C_5\sqrt{\frac{\max \{ \log(1/\alpha), \log(1/\beta) \}}{n}}+C_6 \frac{\max \{ \log(1/\alpha), \log(1/\beta) \}}{n \xi_{\epsilon,\delta}}\bigg) + {\beta}\\
    &\leq\beta,
\end{align*}
which completes the proof. 
\end{proof}

\subsection{Proof for Theorem~\ref{th_minimax_sepa}}\label{app_th_minimax_sepa}
\begin{proof}
    We establish lower bounds for the minimax separation $\rho^*_{\dHSIC}(\alpha, \beta, \epsilon, \delta, n)=\rho^*_{\dHSIC}$ in the non-private regime and in the privacy regime, separately. The main idea is Le Cam's two point method and coupling method \citep{Lecam1973ConvergenceOE}. 
\subsubsection{Non-privacy regime}
    We first show that $\rho^*_{\dHSIC}\geq C_{\eta}\min\{\sqrt{\log(1/(\alpha+\beta))/n},1\}$ in the non-private regime, where $C_{\eta}$ is a positive constant that only depends on $\eta^j$ for $j\in[d]$. Denote the set of the non-private $\alpha$-level test as
    $$
    \Phi_{\alpha, \infty} := \bigg\{ \phi : \sup_{P \in \mathcal{P}_0} \mathbb{E}_{P}[\phi] \leq \alpha \bigg\},
    $$
    and it includes $(\epsilon, \delta)$-DP tests. It is equivalent to verify that
    $$
    \inf\bigg\{ \rho > 0 : \inf_{\phi \in \Phi_{\alpha, \infty}} \sup_{P \in \mathcal{P}_{{\dHSIC}_{\bk}}(\rho)} \mathbb{E}_{P}[1 - \phi] \leq \beta \bigg\}\geq C_{\eta}\min\{\sqrt{\log(1/(\alpha+\beta))/n},1\}
    $$
    by the definition of the minimax separation. For this direction, it is sufficient to pick a joint distribution $P_{X^1,\dots,X^d,0}$ from $\mcP_{\dHSIC_{\bk}}(\tilde\rho)$ with
    \begin{equation}\label{eq_09}
        \tilde\rho=C_{\eta}\min\bigg\{\sqrt{\frac{\log(1/(\alpha+\beta))}{n}},1\bigg\}
    \end{equation}
    such that the type II error rate of any test is higher than $\beta$. On the other hand, for such a $\tilde\rho$ in \eqref{eq_09}, we have
\begin{align*}
    \inf_{\phi \in \Phi_{\alpha, \epsilon,\delta}} \sup_{P \in \mathcal{P}_{{\dHSIC}_{\bk}}(\rho)} \mathbb{E}_{P}[1 - \phi]&\geq\inf_{\phi \in \Phi_{\alpha, \infty}} \sup_{P \in \mathcal{P}_{{\dHSIC}_{\bk}}(\rho)} \mathbb{E}_{P}[1 - \phi]\\
    &\geq\inf_{\phi \in \Phi_{\alpha, \infty}} \mathbb{E}_{P_{X^1,\dots,X^d,0}}[1 - \phi]=1 - \sup_{\phi \in \Phi_{\alpha, \infty}} \mathbb{E}_{P_{X^1,\dots,X^d,0}}[\phi]\\
    &\overset{(i)}{\geq}1-\alpha-d_{\TV}(P_{X^1,\dots,X^d,0}^{\otimes n},\otimes_{j=1}^d P_{X^j,0}^{\otimes n})\\
    &\overset{(ii)}{\geq}1-\alpha-1+\frac{1}{2}\exp\{-d_{\KL}(P_{X^1,\dots,X^d,0}^{\otimes n}\|\otimes_{j=1}^d P_{X^j,0}^{\otimes n})\}
    \\
    &\overset{(iii)}{\geq}\frac{1}{2}\exp\{-n\times d_{\KL}(P_{X^1,\dots,X^d,0}\|\otimes_{j=1}^d P_{X^j,0})\}-\alpha,
\end{align*}
where $d_{\TV}$ and $d_{\KL}$ denote their total variation distance and KL-divergence, respectively. 
Here, Step $(i)$ is due to the fact that $\mathbb{E}_{P_{X^1,\dots,X^d,0}}[\phi]\leq \mathbb{E}_{\otimes_{j=1}^d P_{X^j,0}}[\phi]+d_{\TV}\leq \alpha+d_{\TV}$ holds for any unbiased test $\phi$. Step $(ii)$ holds by Bretaguolle-Huber inequality \citet[Lemma~B.4]{Canonne2022TopicsAT} and Step $(iii)$ is by the chain rule of the KL divergence.
Therefore, the minimax type II error is lower bounded by $\beta$ provided that $\alpha+\beta<c$ for any fixed $c<1/2$ as well as 
$$
d_{\KL}(P_{X^1,\dots,X^d,0}\|\otimes_{j=1}^d P_{X^j,0})\leq\frac{1}{n}\log\bigg(\frac{1}{2(\alpha+\beta)}\bigg).
$$
Hence it suffices to find an example of $P_{X^1,\dots,X^d,0}$ such that
\begin{align}\label{eq_10}
    &\dHSIC_{\bk}(P_{X^1,\dots, X^d}) \geq C_{\eta}\min \bigg\{ \sqrt{\frac{\log(1/(\alpha + \beta))}{n}}, 1 \bigg\}\quad \text{and}\\
    &\label{eq_11}d_{\KL}(P_{X^1,\dots,X^d,0}\|\otimes_{j=1}^d P_{X^j,0})\leq\frac{1}{n}\log\bigg(\frac{1}{2(\alpha+\beta)}\bigg).
\end{align}
To this end, we pick a discrete distribution and set most probability mass at two points. Consider $X^j\in\{x^j_1,x^j_2\}$ for $j\in[d]$. Suppose that the joint probabilities of $(X^1,\dots, X^d)$ are given by
$$
    \mbP(X^j=x^j_1,\text{for all } j\in[d]) =\mbP(X^j=x^j_2,\text{for all } j\in[d])=\frac{1}{2^d}+v,
$$
else the probabilities equal to ${1}/{2^d}-{v}/{(2^{d-1}-1)}$. Here, $v\in(0,(2^{d-1}-1)/2^d]$. It is sufficient to choose $v$ such that \eqref{eq_10} and \eqref{eq_11} hold. 
In this setting, it is clear that $(X^1,\dots, X^d)$ are not jointly independent. Simple calculations yield
\begin{align*}
    d_{\KL}(P_{X^1,\dots,X^d,0}\|\otimes_{j=1}^d P_{X^j,0})&\leq d_{\chi^2}(P_{X^1,\dots,X^d,0}\|\otimes_{j=1}^d P_{X^j,0})\\
    &=2^d\bigg[\bigg(\frac{1}{2^d}+v\bigg)^2\times2+\bigg(\frac{1}{2^d}-\frac{v}{2^{d-1}-1}\bigg)^2\times(2^d-2)\bigg]-1\\
    &=\frac{2^{2d}}{2^{d-1}-1}v^2.
\end{align*}
On the other hand, the squared population dHSIC can be bounded as
\begin{align*}
    &\dHSIC_{\bk}(P_{X^1,\dots, X^d,0})^2 \\
    &\overset{(i)}{=} \mbE\bigg\{ \prod_{j=1}^d k^j(X_1^j, X_2^j) \bigg\} + \mbE \bigg\{ \prod_{j=1}^d k^j(X_{2j-1}^j, X_{2j}^j) \bigg\} - 2  \mbE \bigg\{ \prod_{j=1}^d k^j(X_1^j, X_{j+1}^j) \bigg\}\\
    &\overset{(ii)}{=}\frac{2^d}{(2^{d-1}-1)^2}v^2(2^{d-1}S_1-S_2)\\
    &\overset{(iii)}{\geq}\frac{2^d}{2^{d-1}-1}v^2\prod_{j=1}^d\eta^j,
\end{align*}
where $S_1:=\prod_{j=1}^d\kappa^j(0)+\prod_{j=1}^d\kappa^j(x^j_1-x^j_2)$ and $S_2:=\prod_{j=1}^d\{\kappa^j(0)+\kappa^j(x^j_1-x^j_2)\}$. Step~$(i)$ holds by Lemma~\ref{le_sq_dhsic} and direct algebra calculations lead to Step~$(ii)$. Choose $x^j_1-x^j_2=x_0^j$ for $j\in[d]$ and note the fact that $2^{d-1}S_1-S_2$ can be divided into $2^{d-1}$ terms, which can be written as $\prod_{l}\{\prod_{j\in A_l}\kappa^j(0)-\prod_{j\in A_l}\kappa^j(x^j_1-x^j_2)\}$ where $A_l$'s are a partition of [d]. Step~$(iii)$ holds by 
$$
\prod_{l}\bigg\{\prod_{j\in A_l}\kappa^j(0)-\prod_{j\in A_l}\kappa^j(x^j_1-x^j_2)\bigg\}\geq
\prod_{l}\prod_{j\in A_l}\eta^j=\prod_{j=1}^d\eta^j.
$$
Based on the above inequalities, condition \eqref{eq_11} is fulfilled by picking 
$$
v=\min\bigg\{\sqrt{\frac{2^{d-1}-1}{2^{2d}n}\log\bigg(\frac{1}{2(\alpha+\beta)}\bigg)},\frac{2^{d-1}-1}{2^{d}}\bigg\}.
$$
It is easy to verify that \eqref{eq_10} holds. Therefore, $\rho^*_{\dHSIC}\geq C_{\eta}\min\{\sqrt{\log(1/(\alpha+\beta))/n},1\}$.

\subsubsection{Privacy regime}
To establish the lower bound under privacy constraints, it is sufficient to pick a joint distribution $P_{X^1,\dots,X^d,0}$ from $\mcP_{\dHSIC_{\bk}}(\tilde\rho)$ with
    \begin{equation}\label{eq_12}
        \tilde\rho=C_{\eta}\min\bigg\{\frac{\log({1}/{\beta})}{n\xi_{\epsilon,\delta}},\frac{1}{n\xi_{\epsilon,\delta}},1\bigg\}
    \end{equation}
    such that the type II error rate of any test is higher than $\beta$.
Bounding the type II error rate in the privacy regime can be mainly based on the private coupling idea, which was introduced in \citet[Theorem 1]{Acharya2018DifferentiallyPT}. 
In particular, let $(\mcX_n,\mcX_n')$ be a coupling between $P_{X^1,\dots,X^d,0}^{\otimes n}$ and $\otimes_{j=1}^d P_{X^j,0}^{\otimes n}$ with $D=10\mbE[d_{\ham}(\mcX_n,\mcX_n')]$. 
Then for $\alpha+\beta<c$ for any fixed $c<1/2$, we have
\begin{align*}
    \inf_{\phi \in \Phi_{\alpha, \epsilon,\delta}} \sup_{P \in \mathcal{P}_{{\dHSIC}_{\bk}}(\rho)} \mathbb{E}_{P}[1 - \phi]&\geq\inf_{\phi \in \Phi_{\alpha, \epsilon,\delta}} \mbE_{P_{X^1,\dots,X^d,0}}[1 - \phi]\\
    &\overset{(i)}{\geq} e^{-D\epsilon}(0.9-\alpha)-D\delta\\
    &>(0.4+\beta)e^{-D\epsilon}-D\delta\\
    &\geq(0.2+\beta)e^{-D\epsilon},
\end{align*}
provided that the condition $D\delta\leq0.2e^{-D\epsilon}$ holds. Here Step~$(i)$ holds by the proof of Theorem~1 in \cite{Acharya2021DifferentiallyPA}.
Therefore, it is desired that 
$$
\inf_{\phi \in \Phi_{\alpha, \epsilon,\delta}} \sup_{P \in \mathcal{P}_{{\dHSIC}_{\bk}}(\rho)} \mathbb{E}_{P}[1 - \phi]\geq(0.2+\beta)e^{-D\epsilon}\geq\beta,
$$
which is fulfilled when
\begin{equation}\label{eq_13}
    D\leq\frac{1}{\epsilon}\log\bigg(\frac{1}{5\beta}+1\bigg)\quad\text{and}\quad D\delta e^{D\epsilon}\leq\frac{1}{5}.
\end{equation}
Without loss of generality, assuming that $\delta>0$, we can show that the second condition holds when $D\leq{1}/{(5\epsilon+5\delta)}$.
In this case, letting $x=\epsilon/\delta$, we have
$$
D\delta e^{D\epsilon}\leq\frac{\delta}{5(\epsilon+\delta)}e^{\frac{\epsilon}{5(\epsilon+\delta)}}=\frac{x}{5(x+1)}e^{\frac{1}{5(x+1)}},
$$
which no more than $1/5$ by simple calculations. Thus \eqref{eq_13} is particularly satisfied when
\begin{equation}\label{eq_14}
    D\leq\frac{1}{5(\epsilon+\delta)}\min\bigg\{\log\bigg(\frac{1}{5\beta}+1\bigg),1\bigg\}.
\end{equation}
Furthermore, by \citet[Lemma~20]{Acharya2021DifferentiallyPA}, we have 
$$
\mbE[d_{\ham}(\mcX_n,\mcX_n')]=n\times d_{\TV}(P_{X^1,\dots,X^d,0},\otimes_{j=1}^d P_{X^j,0})=\frac{n}{2}\|P_{X^1,\dots,X^d,0}-\otimes_{j=1}^d P_{X^j,0}\|_1.
$$
Then the condition \eqref{eq_14} becomes
\begin{equation}\label{eq_15}
    \|P_{X^1,\dots,X^d,0}-\otimes_{j=1}^d P_{X^j,0}\|_1\leq\frac{1}{25n(\epsilon+\delta)}\min\bigg\{\log\bigg(\frac{1}{5\beta}+1\bigg),1\bigg\}.
\end{equation}
We also consider the discrete distribution in the non-private regime. A direct calculation yields that $\|P_{X^1,\dots,X^d,0}-\otimes_{j=1}^d P_{X^j,0}\|_1=4v$. Therefore, we take 
$$
v= \min\bigg\{\frac{1}{100n(\epsilon+\delta)}\log\bigg(\frac{1}{5\beta}+1\bigg),\frac{1}{100n(\epsilon+\delta)},\frac{2^{d-1}-1}{2^d}\bigg\}.
$$
Noting that
$$
    v\gtrsim\min\bigg\{\frac{\log({1}/{\beta})}{n\xi_{\epsilon,\delta}},\frac{1}{n\xi_{\epsilon,\delta}},1\bigg\}\asymp\min\bigg\{\frac{1}{n\xi_{\epsilon,\delta}},1\bigg\}
$$
and combining it with lower bound for squared population dHSIC in the non-private regime yields $\rho^*_{\dHSIC}\geq C_{\eta}\min\{{\log({1}/{\beta})}/{n/\xi_{\epsilon,\delta}},{1}/{n/\xi_{\epsilon,\delta}},1\}$. We then complete the proof by combining the two parts.
\end{proof}

\subsection{Proof for Theorem~\ref{th_L2separation}}\label{app_th_L2separation}
\begin{proof}
    The proof structure 
    is that we first analyze the difference between the U- and V-statistics, and then leverage the existing results of the U-statistic. 
    More attention should be paid to the difference between the V- and U-statistics as the kernels are $(2d)$th-order. 
    We treat $\alpha, \beta, s,R,M, p_1,\dots,p_d$ as fixed constants throughout the proof. 
    Without loss of generality, we consider $d\geq3$ as $d=2$ is developed by \citet[Theorem~14]{Kim2026DPP}. 
    We denote $\hat\dHSIC^2(\mcX_n)$ by $V_{\dHSIC}$ and corresponding U-statistic as $U_{\dHSIC}$. The corresponding statistics based on the permuted data $\mcX_n^{\bpi}$ as $V_{\bpi,\dHSIC}$ and $U_{\bpi,\dHSIC}$, respectively. We keep the notations in Theorem~\ref{sen_dhsic}, i.e., $K^j= \prod_{i=1}^{p_j}1/ (\sqrt{2\pi} \lambda_{j,i})$ for $j\in[d]$ and $K_0$ denotes their product. 

    Note that the difference between $V_{\bpi,\dHSIC}$ and $U_{\bpi,\dHSIC}$ is bounded by $C_1K_0/n$ in the proof of Lemma~\ref{le_dhsic_con_per}, we devote more efforts to the bias. Specifically, we divide the difference into two parts. The first term $D_1$ is invariant to any permutation $\bpi$, and is bounded by $K_0/n$ in the sense of order. The remaining term $D_2$ has a sharper lower bound than $K_0/n$. 
    Formally, the difference between $V_{\pi,\dHSIC}$ and $U_{\pi,\dHSIC}$ can be written as
    $$
    V_{\dHSIC}-U_{\dHSIC}=D_1+D_2,
    $$
    where 
    \begin{align*}
        D_1&=\frac{n^{d-1}-1}{n^d}K_0-\sum_{j=1}^d\frac{K_0}{K^jn^{d+1}}\sum_{\bi^n_2}k^j(X^j_{i_1},X^j_{i_2}) \\
        D_2&\asymp -\frac{1}{n^{3}}\sum_{\bi^n_{2}}\prod_{j=1}^dk^j(X^j_{i_{1}},X^j_{i_{2}})+ \frac{1}{n^{d+2}}\sum_{\bi^n_{d+1}}\prod_{j=1}^dk^j(X^j_{i_{1}},X^j_{i_{2j}})-\frac{1}{n^{2d+1}}\sum_{\bi^n_{2d}}\prod_{j=1}^dk^j(X^j_{i_{2j-1}},X^j_{i_{2j}})\\
        &+\sum_{t=1}^{d-2}\sum_{A\subset[d]:|A|=t}\bigg\{
        \frac{1}{n^{2d}}\sum_{\bi^n_{2d-t}}\prod_{j\in [d]\backslash A}k^j(X^j_{i_{2j-1}},X^j_{i_{2j}})
        -\frac{1}{n^{d+1}}\sum_{\bi^n_{d+1-t}}\prod_{j\in [d]\backslash A}k^j(X^j_{i_{1}},X^j_{i_{2j}})
        \bigg\}\prod_{j\in A}K^j. 
    \end{align*}
Combining the analysis above, the permutation distribution of $V_{\bpi,\dHSIC}$ is equivalent to the permutation distribution of $U_{\bpi,\dHSIC}+D_1+D_{2,\bpi}$ where $D_{2,\bpi}$ has the same form as $D_2$ but computed from permuted data $\mcX_n^{\bpi}$. Denoting the sensitivity of $\sqrt{V_{\bpi,\dHSIC}}$ as $\Delta_{V^1/2}$, Lemma~\ref{lemma_repre1} yields that dpdHSIC test rejects the null if and only if 
$$
\sqrt{V_{\bpi,\dHSIC}}+\frac{2\Delta_{V^1/2}}{\xi_{\epsilon,\delta}}\zeta_0>q_{1-\alpha,B},
$$
where $q_{1-\alpha,B}$ is $1-\alpha$ quantile of $\{\sqrt{V_{\bpi_i,\dHSIC}}+{2\Delta_{V^1/2}}{\xi_{\epsilon,\delta}^{-1}}\zeta_i\}$. 
Denote the conditional distribution of $\sqrt{V_{\bpi,\dHSIC}}$ given $\mcX_n$ as $q^a_{1-\alpha/12,\infty}$. 
By the proof for Theorem~\ref{th_mini_sepa_dpdhsic} in Appendix~\ref{app_th_mini_sepa_dpdhsic}, the type II error of the dpdHSIC can be bounded as 
\begin{align}\label{eq_20}
    &\mathbb{P}(\sqrt{V_{\dHSIC}} + 2\Delta_{V^{1/2}} \xi_{\epsilon,\delta}^{-1} \zeta_0 \leq q_{1-\alpha,B})\nonumber\\
    \leq\ &\mathbb{P}(\sqrt{V_{\dHSIC}} \leq q^a_{1-\alpha/12,\infty} + 14\Delta_{V^{1/2}} \xi^{-1}_{\epsilon,\delta} \max\{\log(1/\alpha), \log(1/\beta)\}) + 5\beta/8.
\end{align}
Similar to the structure in the proof for Theorem~\ref{th_mini_sepa_dpdhsic} in Appendix~\ref{app_th_mini_sepa_dpdhsic}, the remaining proof first focuses on bounding $q^a_{1-\alpha/12,\infty}$ and then continues to upper bound the type II error. In what follows, the  notation $C_1,C_2,\dots$ represent universal constants only associated with $\alpha, \beta, s,R,M, p_1,\dots,p_d$.

\textbf{Step 1 (Bounding the quantile).} 
A simple calculation yields 
\begin{align*}
    q^a_{1-\alpha/12,\infty}&= \sqrt{\text{Quantile}_{1-\alpha/12}(\{U_{\bpi_i,\dHSIC}+D_{2,\bpi_i}\}_{i=0}^{\infty}+D_1)}\\
    &\leq \sqrt{\text{Quantile}_{1-\alpha/24}(\{U_{\bpi_i,\dHSIC}\})_{i=0}^{\infty}+\text{Quantile}_{1-\alpha/24}(\{D_{2,\bpi_i}\}_{i=0}^{\infty})+D_1}.
\end{align*}
Combining the proof for Lemma~\ref{le_u_dhsic_con_per} and \citet[Theorem~5.1]{Kim2022MinimaxOO}, as well as the proof for \citet[Proposition~2]{Albert2022AdaptiveTO}, we have
$$
\mathbb{E}_{\bpi}[U_{\bpi,\dHSIC} | \mathcal{X}_n] = 0 \quad \text{and} \quad
\mathbb{E}[\Var_{\bpi}(U_{\bpi,\dHSIC} | \mathcal{X}_n)] \leq {C_1}{n^{-2} \lambda_0^{-1}}.
$$
Then the Chebyshev's inequality shows the following inequality holds with probability at least $1-\beta/8$
\begin{equation}\label{eq_17}
    \text{Quantile}_{1-\alpha/24}(\{U_{\bpi_i,\dHSIC}\})_{i=0}^{\infty}\leq C_2n^{-1}\lambda_0^{-1/2}. 
\end{equation}
For the $1-\alpha/24$ quantile of $\{D_{2,\bpi_i}\}_{i=0}^{\infty}$, note that
$$
    D_{2,\bpi_i}\lesssim\frac{1}{n^{d+2}}\sum_{\bi^n_{d+1}}\prod_{j=1}^dk^j(X^j_{\pi_{i_{1}}^j},X^j_{\pi_{i_{2j}}^j})+\sum_{t=1}^{d-2}\sum_{A\subset[d]:|A|=t}
    \frac{1}{n^{2d}}\sum_{\bi^n_{2d-t}}\prod_{j\in [d]\backslash A}k^j(X^j_{\pi_{i_{2j-1}}^j},X^j_{\pi_{i_{2j}}^j})\prod_{j\in A}K^j.
$$
We take the expectation over $\bpi$ and $P$ on the two sides and denote the expectations of two terms by $E_1$ and $E_2$, respectively. Similar to the proof of Theorem~14 in \cite{Kim2026DPP}, we have $E_1\lesssim n^{-1}$ under the assumptions of $\max(\|p_{X^1,\dots, X^d}\|_{L_\infty}, \|p_{X^1}\dots p_{X^d}\|_{L_\infty}) \leq M$. A simple calculation with similar technique shows that $E_2\lesssim n^{-1}K_0$. 
An application of Markov's inequality yields
\begin{equation}\label{eq_18}
    \text{Quantile}_{1-\alpha/24}(\{D_{2,\bpi_i}\}_{i=0}^{\infty}) \leq C_3n^{-1}K_0,
\end{equation}
with probability at least $1-\beta/8$. A direct calculation shows that $D_1\lesssim n^{-1}K_0$.

Combining the pieces together, we have 
\begin{equation}\label{eq_19}
    q^a_{1-\alpha/12,\infty}\leq \frac{C_4}{\sqrt{n\lambda_0}}
\end{equation}
with probability at least $1-\beta/4$ as $\lambda_0\leq1$ and $\lambda_0\asymp K_0^{-1}$. 

\textbf{Step 2 (Bounding the type II error).} Back to \eqref{eq_20}, we have
\begin{align*}
&\mathbb{P}\Big( \sqrt{V_{\dHSIC}} \leq q^a_{1-\alpha/12,\infty} + 14 \Delta_{V^{1/2}} \xi^{-1}_{\epsilon, \delta} \max\{ \log(1/\alpha), \log(1/\beta) \} \Big) \\
= &\ \mathbb{P}\Big( V_{\dHSIC} \leq (q^a_{1-\alpha/12,\infty})^2 + 14^2 \Delta^2_{V^{1/2}} \xi^{-2}_{\epsilon, \delta} \max\{ \log^2(1/\alpha), \log^2(1/\beta) \} \\
&\quad + 28 q^a_{1-\alpha/12,\infty} \Delta_{V^{1/2}} \xi^{-1}_{\epsilon, \delta} \max\{ \log(1/\alpha), \log(1/\beta) \}  \Big) \\
=\ &\mathbb{P}\Big( U_{\dHSIC} \leq \text{Quantile}_{1-\alpha/24}(\{U_{\bpi_i, \dHSIC}\}_{i=0}^\infty) + \text{Quantile}_{1-\alpha/24}(\{D_{2, \bpi_i}\}_{i=0}^\infty) - D_2 \\
&\quad + 14^2 \Delta^2_{V^{1/2}} \xi^{-2}_{\epsilon, \delta} \max\{ \log^2(1/\alpha), \log^2(1/\beta) \} + 28 q^a_{1-\alpha/12, \infty} \Delta_{V^{1/2}} \xi^{-1}_{\epsilon, \delta} \max\{ \log(1/\alpha), \log(1/\beta) \} \Big),
\end{align*} 
where the second equality uses the identity $V_{\dHSIC} = U_{\dHSIC} + D_1 + D_2$.  By Proposition~\ref{sen_dhsic}, the sensitivity of $\sqrt{V_{\dHSIC}}$ is bounded as
$$
\Delta_{V^{1/2}} \leq {C_5}{n^{-1} \lambda^{-1/2}}.
$$
By Markov's inequality, we have
\begin{equation}\label{eq_21}
    |D_2|\leq\frac{C_6}{n}+\sum_{\bi^n_1}\frac{C_7K^j}{n}+\sum_{\bi^n_2}\frac{C_8K^{j_1}K^{j_2}}{n^2}+\dots+\sum_{\bi^n_{d-2}}\frac{C_9K^{j_1}\dots K^{j_{d-2}}}{n^{d-2}}
\end{equation}
with probability at least $1-\beta/16$. A sharper upper bound for $\text{Quantile}_{1-\alpha/24}(\{D_{2, \bpi_i}\}_{i=0}^\infty)$ can be derived in the same way. 
For $d\geq4$, combining the pieces \eqref{eq_20}, \eqref{eq_17}, \eqref{eq_18}, \eqref{eq_19} and \eqref{eq_21} together, the type II error can be bounded as  
\begin{align*}
    &\mathbb{P}(\sqrt{V_{\dHSIC}} + 2\Delta_{V^{1/2}} \xi_{\epsilon,\delta}^{-1} \zeta_0 \leq q_{1-\alpha,B})\\
    \leq\ &\mathbb{P}\bigg({U_{\dHSIC}} \leq \frac{C_{10}}{n \sqrt{\lambda_0}} + \frac{C_{11}}{n^2 \lambda_0 \xi^2_{\epsilon, \delta}} + \frac{C_{12}}{n^{3/2} \lambda_0 \xi_{\epsilon, \delta}}+\sum_{t=1}^{d-2}\sum_{\bi^d_t}\frac{C_{13}}{n^t\lambda_{i_1}\dots\lambda_{i_t}}\bigg) + \frac{15}{16}\beta.
\end{align*}
Next we establish the
upper bound the first probability under the sufficient condition on $L_2$-norms. 

A slight modification of \citet[Lemma~1]{Albert2022AdaptiveTO} yields that a sufficient condition for the first probability in the above display to be less than $\beta/16$ is
\begin{align*}
\dHSIC^2_{\bk}(P_{X^1,\dots,X^d}) &\geq \sqrt{\Var(U_{\dHSIC})} + \frac{C_{10}}{n \sqrt{\lambda_0}} + \frac{C_{11}}{n^2 \lambda_0 \xi^2_{\epsilon, \delta}} + \frac{C_{12}}{n^{3/2} \lambda_0 \xi_{\epsilon, \delta}}+\sum_{t=1}^{d-2}\sum_{\bi^d_t}\frac{C_{13}}{n^t\lambda_{i_1}\dots\lambda_{i_t}}.
\end{align*}
With a little abuse of notations, by writing $\psi = p_{X^1,\dots,X^d} - p_{X^1}\dots p_{X^d}$ and the convolution of $\psi$ and $\bk$ as $\psi*\bk$, the proofs for \citet[Theorem~1 and Proposition~4]{Albert2022AdaptiveTO} show that the previous condition is implied by
$$
\|\psi\|_{L_2}^2 \geq \|\psi - \psi * \bk\|_{L_2}^2 + C_{14} \bigg\{ \frac{1}{n \sqrt{\lambda_0}} + \frac{1}{n^2 \lambda_0 \xi^2_{\epsilon, \delta}} + \frac{1}{n^{3/2} \lambda_0 \xi_{\epsilon, \delta}} +\sum_{t=1}^{d-2}\sum_{\bi^d_t}\frac{1}{n^t\lambda_{i_1}\dots\lambda_{i_t}}\bigg\}.
$$
Finally, a small modification of the proof of \citet[Lemma~3]{Albert2022AdaptiveTO} shows that a sufficient condition for the inequality is 
$$
\|\psi\|_{L_2}^2 \geq C_{15} \bigg\{\sum_{j=1}^d\sum_{i=1}^{p_j} \lambda_{j,i}^{2s}+ \frac{1}{n \sqrt{\lambda_0}} + \frac{1}{n^2 \lambda_0 \xi^2_{\epsilon, \delta}} + \frac{1}{n^{3/2} \lambda_0 \xi_{\epsilon, \delta}} +\sum_{t=1}^{d-2}\sum_{\bi^d_t}\frac{1}{n^t\lambda_{i_1}\dots\lambda_{i_t}}\bigg\},
$$
which is the desired result. 
\end{proof}

\section{Proofs for Results in Appendix~\ref{sec_u_dhsic}}\label{app_proof_u}
\subsection{Proof for Lemma~\ref{le_sen_u}}
\begin{proof}
    For the upper bound, we write the permutation U-statistic as $U_{\bpi,\dHSIC}=U_{\bpi,\dHSIC}^a+U_{\bpi,\dHSIC}^b-2U_{\bpi,\dHSIC}^c$, where
    \begin{align*}
        U_{\bpi,\dHSIC}^a&:=\frac{(n-2)!}{n!} \sum_{\bi_2^n} \prod_{j=1}^dk^j(X_{\pi_{i_1}^j}^j, X_{\pi_{i_2}^j}^j),  \\
        U_{\bpi,\dHSIC}^b&:= \frac{(n-2d)!}{n!} \sum_{\bi_{2d}^n} \prod_{j=1}^dk^j(X_{\pi_{i_{2j-1}}^j}^j, X_{\pi_{i_{2j}}^j}^j), \\
        U_{\bpi,\dHSIC}^c&:= \frac{(n-d-1)!}{n!} \sum_{\bi_{d+1}^n} \prod_{j=1}^dk^j(X_{\pi_{i_1}^j}^j, X_{\pi_{i_{j+1}}^j}^j).
    \end{align*}
    To establish an upper bound for the global sensitivity, we first consider a neighboring dataset $\tilde\mcX_n$ and denote the U-statistic of dHSIC based on $\tilde \mcX_n$ as $\tilde U_{\bpi,\dHSIC}=\tilde U_{\bpi,\dHSIC}^a+\tilde U_{\bpi,\dHSIC}^b-2\tilde U_{\bpi,\dHSIC}^c$. Then the difference between $U_{\bpi,\dHSIC}$ and $\tilde U_{\bpi,\dHSIC}$ can be bounded as 
    \begin{align*}
        |U_{\bpi,\dHSIC}-\tilde U_{\bpi,\dHSIC}|&\leq|U_{\bpi,\dHSIC}^a-\tilde U_{\bpi,\dHSIC}^a|+|U_{\bpi,\dHSIC}^b-\tilde U_{\bpi,\dHSIC}^b|+2|U_{\bpi,\dHSIC}^c-\tilde U_{\bpi,\dHSIC}^c|\\
        &\overset{(*)}{\leq}\frac{2d}{n}K_0+\frac{2d^2}{n}K_0+\frac{2d(d+1)}{n}K_0.
    \end{align*}
    Here, we explain $|U_{\bpi,\dHSIC}^a-\tilde U_{\bpi,\dHSIC}^a|{\leq}2dK_0/n$ in detail and the other two hold similarly in step $(*)$. In fact, we consider 
    \begin{equation}\label{eq_29}
        |U_{\bpi,\dHSIC}^a-\tilde U_{\bpi,\dHSIC}^a|\leq\frac{1}{n(n-1)} \sum_{\bi_2^n} \bigg|\prod_{j=1}^dk^j(X_{\pi_{i_1}^j}^j, X_{\pi_{i_2}^j}^j)-\prod_{j=1}^dk^j(\tilde X_{\pi_{i_1}^j}^j, \tilde X_{\pi_{i_2}^j}^j)\bigg|. 
    \end{equation}
    Without loss of generality, suppose that we change $\bX_1$ into $\bX_1'$, which leads to at most $d$ components changing. This implies that at most $2d(n-1)$ terms can be non-zero in the summation on the right-hand side of \eqref{eq_29}. Each non-zero term is bounded by $K_0$, which follows the desired result.  

    For the lower bound, we write $\Delta_T^d$ as $\Delta_T$ for the sensitivity of U-statistic dHSIC testing $d$ vectors with kernels $K^1,\dots,K^d$, and write $\Delta_T^j$ similarly for $j\in[d]$. We claim that $\Delta_T^d\geq {K^d}\Delta_T^{d-1}$ by the same construction as Appendix~\ref{app_a22}. Then the conclusion holds by \citet[Lemma~S.8]{Kim2026DPP}.   
\end{proof}

\subsection{Proof for Theorem~\ref{th_a1}}
\begin{proof}
    By the proof of Lemma~\ref{le_dhsic_con_per}, we have
    $$
    |U_{\dHSIC}-V_{\dHSIC}|\lesssim\frac{K_0}{n}.
    $$
    By treating $K_0$ as some fixed number, Lemma~\ref{le_ineq_e_dhsic} yields that 
    $$
    V_{\dHSIC}=\{\dHSIC_{\bk}(P_{X^1,\dots,X^d})+R_n\}^2\quad \text{where }R_n=O_p(n^{-1/2}).
    $$
    Following the same proof strategy as for \citet[Theorem~S.6]{Kim2026DPP} shows 
    $$
    \limsup_{n\to\infty} \inf_{P_{X^1,\dots, X^d} \in \mcP_{\dHSIC_{\bk}}(\rho)} \mathbb{E}_{P_{X^1,\dots, X^d}}[\phi^u_{\dpdHSIC}] \le \alpha,
    $$
    which completes the proof. 
\end{proof}

\subsection{Proof for Theorem~\ref{th_a2}}
\begin{proof}
    Denote by $C_1, C_2, \dots$ constants that may depend on $\alpha, \beta, R, M, d_Y, d_Z$. Following the proofs of Theorem~\ref{th_L2separation} along with the sensitivity result for $U_{\dHSIC}$ in Lemma~\ref{le_sen_u}, we may arrive at the point where the type II error of $\phi^u_{\dpdHSIC}$ is upper bounded as
$$
\mathbb{E}[1 - \phi^u_{\dpdHSIC}] \le \mathbb{P}\left( U_{\dHSIC} \le \frac{C_1}{n\sqrt{\lambda_0}} + \frac{C_2}{n\lambda_0 \xi_{\epsilon, \delta}} \right) + \frac{15}{16}\beta
$$
We then use the proofs of \citet[Theorem~1]{Albert2022AdaptiveTO} and \citet[Lemma~3]{Albert2022AdaptiveTO}, and show that the probability term in the above display is less than or equal to $\beta/16$ once
$$
\|\psi\|_{L_2}^2 \ge C_3 \left\{ \sum_{i=1}^{d_Y} \lambda_i^{2s} + \sum_{i=1}^{d_Z} \mu_i^{2s} + \frac{1}{n\sqrt{\lambda_0}} + \frac{1}{n\lambda_0 \xi_{\epsilon, \delta}} \right\},
$$
which completes the proof.
\end{proof}

\section{Technical Lemmas}\label{app_lemma}
\subsection{Lemmas on Resampling Tests}\label{app_lemma_re}
We first present three lemmas used in DP resampling tests. 
The following one can be found in \citet[Lemma~S.16]{Kim2026DPP}, we also provide a proof here for completeness. 
\begin{lemma}[Quantile Representation]
\label{lemma_repre1}
    For any $\alpha \in [0, 1]$ and dataset $\{X_1, \dots, X_{n+1}\}$, we have the identity
$$
\mathbbm{1}\bigg(\frac{1}{n+1} \bigg(\sum_{i=1}^n 1(X_{n+1} \leq X_i) + 1\bigg) \leq \alpha\bigg) = \mathbbm{1}(X_{n+1} > q_{1-\alpha}),
$$
where $q_{1-\alpha}$ is the $1-\alpha$ quantile of $\{X_1, \dots, X_{n+1}\}$ given by
$$
q_{1-\alpha} = \inf\bigg\{t \in \mathbb{R} : \frac{1}{n+1} \sum_{i=1}^{n+1} \mathbbm{1}(X_i \leq t) \geq 1-\alpha\bigg\}.
$$
Moreover, by letting $X_{(\lceil(1-\alpha)(n+1)\rceil)}$ denote the $\lceil(1-\alpha)(n+1)\rceil$th order statistic of $\{X_1, \dots, X_{n+1}\}$ and $X_{(0)} = -\infty$, we have $q_{1-\alpha} = X_{(\lceil(1-\alpha)(n+1)\rceil)}$.
\end{lemma}
\begin{proof}
   The second claim follows by noting that
\begin{equation}\label{eq62}
\begin{aligned}
q_{1-\alpha} &= \inf\bigg\{t \in \mathbb{R} : \frac{1}{n+1} \sum_{i=1}^{n+1} \mathbbm{1}(X_i \leq t) \geq 1-\alpha\bigg\}\\
&= \inf\bigg\{t \in \mathbb{R} : \frac{1}{n+1} \sum_{i=1}^{n+1} \mathbbm{1}(X_i < t) \geq 1-\alpha\bigg\}\\
&= \inf\bigg\{t \in \mathbb{R} : \sum_{i=1}^{n+1} \mathbbm{1}(X_i < t) \geq (1-\alpha)(n+1)\bigg\}\\
&= X_{(\lceil(1-\alpha)(n+1)\rceil)}.
\end{aligned}
\end{equation}
For the first claim, denote $G(x) = \sum_{i=1}^{n+1} \mathbbm{1}(X_i < x)$, which is a left-continuous step function. We then have
\begin{align*}
    \mathbbm{1}\bigg(\frac{1}{n+1} \bigg(\sum_{i=1}^n \mathbbm{1}(X_{n+1} \leq X_i) + 1\bigg) \leq \alpha\bigg)& = \mathbbm{1}(G(X_{n+1})\geq (1-\alpha)(n+1))\\
    & \leq \mathbbm{1}(X_{n+1} > q_{1-\alpha}),
\end{align*}
where the last step follows since $G$ is a left-continuous step function. 
Therefore, $X_{n+1}$ cannot be the $1-\alpha$ quantile and should be greater than $q_{1-\alpha}$.

Moreover, the event $X_{n+1} > q_{1-\alpha}$ implies that $G(X_{n+1}) \geq (1-\alpha)(n+1)$ by the definition of $q_{1-\alpha}$ as in expression \eqref{eq62}. Hence we conclude that
$$
\mathbbm{1}(G(X_{n+1}) \geq (1-\alpha)(n+1)) = \mathbbm{1}(X_{n+1} > q_{1-\alpha}),
$$
and the first claim follows.
\end{proof}

The following lemma is an alternative quantile expression. The proof can be found in \citet[Lemma~S.17]{Kim2026DPP}, and thus we omit it here. 
\begin{lemma}[Alternative Expression]
\label{lemma_repre2}
Given $\alpha \in (0, 1)$ and $n \geq 1$, set
$$
\alpha_\star = \max\bigg\{\bigg(\frac{n+1}{n} \alpha - \frac{1}{n}\bigg), 0\bigg\}.
$$
Then, for any dataset $\{X_1, \dots, X_{n+1}\}$, we have the identity
$$
\mathbbm{1}(X_{n+1} > q_{1-\alpha}) = \mathbbm{1}(X_{n+1} > r_{1-\alpha_\star}) \mathbbm{1}\bigg(\alpha \geq \frac{1}{n+1}\bigg),
$$
where $q_{1-\alpha}$ and $r_{1-\alpha_\star}$ are the $1-\alpha$ quantile of $\{X_i\}_{i=1}^{n+1}$ and the $1-\alpha_\star$ quantile of $\{X_i\}_{i=1}^n$, respectively.
\end{lemma}

The next result is concerned with the sensitivity of the quantiles and we borrow the proof of Lemma~S.18 in \cite{Kim2026DPP}. 

\begin{lemma}[Sensitivity of Quantiles]\label{lemma_sen_quan} Suppose that the test statistic $T$ has the sensitivity at most $\Delta_T$. Let us denote by $r_{1-\alpha}(\mcX_n; \{\bphi_i, \zeta_i\}_{i=1}^B)$ the $1-\alpha$ quantile of $\{M_i\}_{i=1}^B$ where $M_i = T(\mcX_n^{\bphi_i}) + 2\Delta_T \xi_{\epsilon, \delta}^{-1} \zeta_i$ with $\xi_{\epsilon, \delta} = \epsilon + \log(1/(1-\delta))$ and $\zeta_i \overset{\text{i.i.d.}}{\sim} \Lap(0, 1)$ for $i \in [B]$.
Then for any $\alpha \in [0, 1)$, the sensitivity of the $1-\alpha$ quantile satisfies
$$
\sup_{d_{\ham}(\mcX_n, \tilde{\mcX}_n) \leq 1} 
|r_{1-\alpha}(\mcX_n; \{\bphi_i, \zeta_i\}_{i=1}^B) - r_{1-\alpha}(\tilde{\mcX}_n; \{\bphi_i, \zeta_i\}_{i=1}^B)| \leq \Delta_T,
$$
for any resamples $\bphi_1, \dots, \bphi_B$ and any $\zeta_1, \dots, \zeta_B \overset{\text{i.i.d.}}{\sim} \Lap(0, 1)$.
\end{lemma}
\begin{proof}
Let $\tilde{\mcX}_n$ be a neighboring dataset of $\mcX_n$ differing only in their $k$th component for some $k \in [n]$. Denote the resampled test statistics computed on $\tilde\mcX_n^{\bphi_1}, \dots, \tilde\mcX_n^{\bphi_B}$ by $\tilde T_1, \dots, \tilde T_B$. We write $T_i=T(\mcX_n^{\bphi_i})$, $r_{1-\alpha} = r_{1-\alpha}(\mcX_n; \{\bphi_i, \zeta_i\}_{i=1}^B)$ and $\tilde{r}_{1-\alpha} = r_{1-\alpha}(\tilde{\mcX}_n; \{\bphi_i, \zeta_i\}_{i=1}^B)$. Having this notation, first note that $T_i \geq \tilde{T}_i - \Delta_T$ for all $i \in [B]$ and thus
$$
1-\alpha \leq \frac{1}{B} \sum_{i=1}^B \mathbbm{1}(T_i + 2\Delta_T \xi_{\epsilon, \delta}^{-1} \zeta_i \leq r_{1-\alpha}) 
\leq \frac{1}{B} \sum_{i=1}^B \mathbbm{1}(\tilde{T}_i + 2\Delta_T \xi_{\epsilon, \delta}^{-1} \zeta_i \leq r_{1-\alpha} + \Delta_T),
$$
which implies that $\tilde{r}_{1-\alpha} \leq r_{1-\alpha} + \Delta_T$.

Next we argue that $\tilde{r}_{1-\alpha} \geq r_{1-\alpha} - \Delta_T$. For this direction, let $\epsilon > 0$ be an arbitrary constant. Then by the definition of $r_{1-\alpha}$ and $T_i \leq \tilde{T}_i + \Delta_T$ for all $i \in [B]$,
$$
1-\alpha \geq \frac{1}{B} \sum_{i=1}^B \mathbbm{1}(T_i + 2\Delta_T \xi_{\epsilon, \delta}^{-1} \zeta_i \leq r_{1-\alpha} - \epsilon) 
\geq \frac{1}{B} \sum_{i=1}^B \mathbbm{1}(\tilde{T}_i + 2\Delta_T \xi_{\epsilon, \delta}^{-1} \zeta_i \leq r_{1-\alpha} - \epsilon - \Delta_T).
$$
Hence $\tilde{r}_{1-\alpha} > r_{1-\alpha} - \epsilon - \Delta_T$. Since $\epsilon$ is arbitrary, we conclude $\tilde{r}_{1-\alpha} \geq r_{1-\alpha} - \Delta_T$. In summary, we have established that $|\tilde{r}_{1-\alpha} - r_{1-\alpha}| \leq \Delta_T$, which holds for any $k \in [n]$, any resamples $\bphi_1, \dots, \bphi_B$, and any $\zeta_1, \dots, \zeta_B \overset{\text{i.i.d.}}{\sim} \Lap(0, 1)$. Therefore, the desired claim follows.
\end{proof}

Pointwise consistency is often regarded as a weak power guarantee, while the uniform power is more desirable. We next state the uniform power of  the DP resampling test. It is a generalization of uniform power of the DP permutation tests \citep[Theorem~4]{Kim2026DPP}. For completeness, we formulate it as Lemma~\ref{lemma_uniformpower}.

\begin{lemma}
\label{lemma_uniformpower}
For $\alpha \in (0, 1)$, $\beta \in (0, 1 - \alpha)$ and $\xi_{\epsilon, \delta} > 0$, assume that $B \geq 6\alpha^{-1}\log(2\beta^{-1})$ and for any $P \in \mathcal{P}_1$,
\begin{align*}
    \mathbb{E}_P[T(X_n)] - \mathbb{E}_{P, \bphi}[T(X_n^{\bphi})] &\geq C_1 \sqrt{\frac{\Var_P[T(X_n)] + \Var_{P, \bphi}[T(X_n^{\bphi})]}{\alpha \beta}}\\
    &+ C_2 \frac{\Delta_T}{\xi_{\epsilon, \delta}} \max\bigg\{\log\bigg(\frac{1}{\alpha}\bigg), \log\bigg(\frac{1}{\beta}\bigg)\bigg\},
\end{align*}
where $C_1$ and $C_2$ are universal constants. Then the uniform power of the private resampling test is bounded below by $1 - \beta$ as
$$
\inf_{P \in \mathcal{P}_1} \mathbb{P}_P(\hat{p}_{\DP} \leq \alpha) \geq 1 - \beta.
$$
\end{lemma}
\begin{proof}
    The proof structure is also similar to our proof of Theorem~\ref{th_mini_sepa_dpdhsic}. We outline the structure here, and details can be referred to the proof of \citet[Theorem~4]{Kim2026DPP}.

     The key idea is to replace the random resampling threshold with a deterministic one using concentration inequalities. By the quantile representation of resampling test (Lemma~\ref{lemma_repre1}), the type II error can be expressed as
    $$
        \mbE[1-\phi_{\dpdHSIC}]=\mbP(\hat{p}_{\DP}>\alpha)=\mbP(M_0\leq q_{1-\alpha,B}),
    $$
    where $q_{1-\alpha,B}$ is the $1-\alpha$ quantile of $M_0,M_1,\dots,M_B$ and $M_i=T(\mcX^{\bpi_i}_n)+2\Delta_T\xi^{-1}_{\epsilon,\delta}\zeta_i$. The proof is divided into two steps, where the quantile is bounded in the first step and the type II error is controlled in the second step.
\end{proof}

\subsection{Lemmas on dHSIC}\label{app_le_dhsic}
Next we demonstrate the reformulations of the squared population dHSIC and the squared V-statistic of dHSIC. 
\begin{lemma}[Reformulations of dHSIC]
\label{le_sq_dhsic}
    Suppose that $\bk$ is measurable and $\mbE\sqrt{\bk(x,x')}<\infty$, then the squared dHSIC in \eqref{df_dhsic} can be reformulated as 
    $$
    \dHSIC_{\bk}(P_{X^1,\dots, X^d})^2 = \mbE\bigg\{ \prod_{j=1}^d k^j(X_1^j, X_2^j) \bigg\} + \mbE \bigg\{ \prod_{j=1}^d k^j(X_{2j-1}^j, X_{2j}^j) \bigg\} - 2  \mbE \bigg\{ \prod_{j=1}^d k^j(X_1^j, X_{j+1}^j) \bigg\},
    $$
    and the squared V-statistic can be rewritten as
    $$
    \widehat{\dHSIC}^2(\mcX_n)= \frac{1}{n^2} \sum_{\bj_2^n} \prod_{j=1}^d
k^j(X_{i_1}^j, X_{i_2}^j)  
+ \frac{1}{n^{2d}} \sum_{\bj_{2d}^n} \prod_{j=1}^d
k^j(X_{i_{2j-1}}^j, X_{i_{2j}}^j) 
- \frac{2}{n^{d+1}} \sum_{\bj_{d+1}^n} \prod_{j=1}^d
k^j(X_{1}^j, X_{j+1}^j).
    $$
\end{lemma}
\begin{proof}
    According to the reproducing property and \citet[Lemma~3]{Gretton2012AKT}, there exist mean embeddings $\mu_p,\mu_q\in\mcH_{\bk}$ such that $\sup_{f\in\mcF_{\bk}}
    \mbE_{P_{X^1,\dots, X^d}}[f(X^1,\dots, X^d)]=\langle f, \mu_p\rangle_{\mcH_{\bk}}$ and $\mbE_{\otimes_{j=1}^d P_{X^j}}[f(X^1,\dots, X^d)]=\langle f, \mu_q\rangle_{\mcH_{\bk}}$ for all $f\in\mcF_{\bk}$. Then we have 
    \begin{align*}
    \dHSIC_{\bk}(P_{X^1,\dots, X^d})^2 &=\bigg[\sup_{f\in\mcF_{\bk}}\bigg\{
    \mbE_{P_{X^1,\dots, X^d}}[f(X^1,\dots, X^d)]-\mbE_{\otimes_{j=1}^d P_{X^j}}[f(X^1,\dots, X^d)]\bigg\}\bigg]^2\\
    &=\bigg\{\sup_{f\in\mcF_{\bk}}\langle f, \mu_p-\mu_q\rangle_{\mcH_{\bk}}
    \bigg\}^2=\|\mu_p-\mu_q\|_{\mcH_{\bk}}
    ^2\\
    &=\langle \mu_p,\mu_p\rangle_{\mcH_{\bk}}+ \langle \mu_q,\mu_q\rangle_{\mcH_{\bk}}- 2\langle \mu_p,\mu_q\rangle_{\mcH_{\bk}}\\
    &=\mbE_{P_{X^1,\dots, X^d}}\bk(X,X')+\mbE_{\otimes_{j=1}^d P_{X^j}}\bk(X,X')-2\mbE_{P_{X^1,\dots, X^d},\otimes_{j=1}^d P_{X^j}}\bk(X,X'),
    \end{align*}
where the subscripts of the expectation in last line indicate the distributions of $X$ and $X'$. The proof is thus completed. The sample version is similar to the population version, which is omitted here. 
\end{proof}

Next we derive the concentration inequality for permuted dHSIC, which is built in two steps. The first establishes the concentration inequality for a permuted U-statistic of dHSIC, which is given by 
\begin{align}\label{eq_02}
    U_{\bpi, \dHSIC} =& \frac{(n-2)!}{n!} \sum_{(i,j) \in \bi_2^n} \prod_{l=1}^d
k^l(X_{\pi_i^l}^l, X_{\pi_j^l}^l)  
+ \frac{(n-2d)!}{n!} \sum_{(i_1, j_1,\dots, i_d, j_d) \in \bi_{2d}^n} \prod_{l=1}^d
k^l(X_{\pi_{i_l}^l}^l, X_{\pi_{j_l}^l}^l) \nonumber\\
&- \frac{2(n-d-1)!}{n!} \sum_{(i, j_1, \dots, j_d) \in \bi_{d+1}^n} \prod_{l=1}^d
k^l(X_{\pi_{i}^l}^l, X_{\pi_{j_l}^l}^l).
\end{align}
The second step bounds the discrepancy between the V- and U-statistics. The concentration inequality for a permuted U-statistic of dHSIC is not a trivial extension of the $d=2$ case, which has been proved in \cite{Kim2022MinimaxOO}. 
We now formulate the first step as Lemma~\ref{le_u_dhsic_con_per} for its own theoretical significance.  
\begin{lemma}[Concentration Inequality for $U_{\bpi, \dHSIC}$] \label{le_u_dhsic_con_per}
Consider the permuted U-statistic \eqref{eq_02} and assume that the kernels $k^j$ are bounded as $0\leq k^j(x,x')\leq K^j$ for all $x,x'\in \mbX^j$ and $j\in[d]$. Let $K_0:=\prod_{j=1}^dK^j$,
then for all $t>0$ and some constant $C > 0$, we have
$$\mathbb{P}_{\bpi} ( U_{\bpi, \dHSIC} \geq t \mid \mathcal{X}_n ) 
\leq \exp \bigg\{ -C \min \bigg( \frac{n^2t^2}{K_0^2}, \frac{nt}{K_0} \bigg) \bigg\}. $$
\end{lemma}
\begin{proof}

    The proof is completed in three steps after a reformulation. We introduce $d$ groups of Rademacher random variables in Step 1, which are used to bound individual terms in Step 2. In Step 3, we combine the inequalities derived in Step 2 to obtain the desired result. 
    To start with, we reformulate the U-statistic as 
    \begingroup\small
    \begin{align}\label{eq_25}
        &U_{\bpi, \dHSIC}\nonumber\\
        & = \frac{(n-2d-2)!}{n!}\sum_{\bi_{2d+2}^n}\bigg\langle \prod_{j=1}^dk^j(X^j_{\pi^j_{i_{1}}},\cdot)-\prod_{j=1}^dk^j(X^j_{\pi^j_{i_{2j+1}}},\cdot), \prod_{j=1}^dk^j(X^j_{\pi^j_{i_{2}}},\cdot)-\prod_{j=1}^dk^j(X^j_{\pi^j_{i_{2j+2}}},\cdot) \bigg\rangle_{\mcH}\nonumber\\
        &= \frac{(n-2d-2)!}{n!}\sum_{\bi_{2d+2}^n}\bigg\langle\sum_{S_1\subseteq[d]} \prod_{j\in S^c_1}\bigg\{k^j(X^j_{\pi^j_{i_{1}}},\cdot)-k^j(X^j_{\pi^j_{i_{2j+1}}},\cdot)\bigg\}\prod_{j\in S_1}k^j(X^j_{\pi^j_{i_{2j+1}}},\cdot)-\prod_{j=1}^dk^j(X^j_{\pi^j_{i_{2j+1}}},\cdot), \nonumber\\
        &\quad\sum_{S_2\subseteq[d]} \prod_{j\in S^c_2}\bigg\{k^j(X^j_{\pi^j_{i_{2}}},\cdot)-k^j(X^j_{\pi^j_{i_{2j+2}}},\cdot)\bigg\}\prod_{j\in S_2}^dk^j(X^j_{\pi^j_{i_{2j+2}}},\cdot)-\prod_{j=1}^dk^j(X^j_{\pi^j_{i_{2j+2}}},\cdot) \bigg\rangle_{\mcH}\nonumber\\
        &= \frac{(n-2d-2)!}{n!}\sum_{\bi_{2d+2}^n}\bigg\langle\sum_{S_1\subsetneqq[d]} \prod_{j\in S_1^c}\bigg\{k^j(X^j_{\pi^j_{i_{1}}},\cdot)-k^j(X^j_{\pi^j_{i_{2j+1}}},\cdot)\bigg\}
        \prod_{j\in S_1}k^j(X^j_{\pi^j_{i_{2j+1}}},\cdot)
        , \nonumber\\
        &\quad\sum_{S_2\subsetneqq[d]} \prod_{j\in S_2^c}\bigg\{k^j(X^j_{\pi^j_{i_{2}}},\cdot)-k^j(X^j_{\pi^j_{i_{2j+2}}},\cdot)\bigg\}
        \prod_{j\in S_2}k^j(X^j_{\pi^j_{i_{2j+2}}},\cdot) \bigg\rangle_{\mcH}\nonumber\\
        &=\frac{(n-2d-2)!}{n!}\sum_{S_1\subsetneqq[d]}\sum_{S_2\subsetneqq[d]}\sum_{\bi_{2d+2}^n}\prod_{j\in S_1\cap S_2}k^j(X^j_{\pi^j_{i_{2j+1}}},X^j_{\pi^j_{i_{2j+2}}})\cdot\nonumber\\
        &\prod_{j\in S_1^c\cap S_2}\bigg\{k^j(X^j_{\pi^j_{i_{1}}},X^j_{\pi^j_{i_{2j+2}}})-k^j(X^j_{\pi^j_{i_{2j+1}}},X^j_{\pi^j_{i_{2j+2}}})
        \bigg\}\prod_{j\in S_1\cap S_2^c}\bigg\{k^j(X^j_{\pi^j_{i_{2j+1}}},X^j_{\pi^j_{i_{2}}})-k^j(X^j_{\pi^j_{i_{2j+1}}},X^j_{\pi^j_{i_{2j+2}}})\bigg\}\cdot\nonumber\\
        &\prod_{j\in S_1^c\cap S_2^c}\bigg\{k^j(X^j_{\pi^j_{i_{1}}},X^j_{\pi^j_{i_{2}}})+k^j(X^j_{\pi^j_{i_{2j+1}}},X^j_{\pi^j_{i_{2j+2}}})-k^j(X^j_{\pi^j_{i_{1}}},X^j_{\pi^j_{i_{2j+2}}})-k^j(X^j_{\pi^j_{i_{2j+1}}},X^j_{\pi^j_{i_{2}}})
        \bigg\}\nonumber\\
        &=:\frac{(n-2d-2)!}{n!}\sum_{S_1\subsetneqq[d]}\sum_{S_2\subsetneqq[d]}\sum_{\bi_{2d+2}^n}g_{\bpi}(S_1,S_2;\bX_{\bpi_{i_1}},\dots, \bX_{\bpi_{i_{2d+2}}}),
    \end{align}
    \endgroup
    where $\bX_{\bpi_{i_{j}}}=(X^1_{\pi^1_{i_j}},\dots, X^d_{\pi^d_{i_j}})$ for $j\in[2d+2]$. For brevity and when there is no ambiguity, we denote the notation $g_{\bpi}(S_1,S_2;\bX_{\bpi_{i_1}},\dots, \bX_{\bpi_{i_{2d+2}}})$ by $ g_{\bpi}(S_1,S_2)$ without further explanation. 
    
    \textbf{Step 1 (Introducing Rademacher r.v.s)} We introduce $d$ groups of Rademacher random variables, and each group has $\lfloor n/(d+1)\rfloor$ independent Rademacher random variables. 
    To elaborate clearly, we focus on the first group, which is introduced by discussing the element ``$1$" belongs to which set.
    Specifically, we categorize the proper subsets $(S_1,S_2)$ of $[d]$ into three distinct classes: the first class denoted by superscript ``$a$" satisfies $1\in S_1\cap S_2$ , and the second class denoted with superscript ``$b$"  satisfies $1 \in S_1^c\cap S_2^c$. The superscripts ``$c_1$" and ``$c_2$" denote subsets such that $1 \in S_1^c\cap S_2$ and $1 \in S_1\cap S_2^c$ , respectively. 
    
    According to the classification on $(S_1,S_2)$, we divide $U_{\bpi, \dHSIC}$ into corresponding parts as
    $$
    U_{\bpi, \dHSIC}:={U}_{\bpi, \dHSIC}^{a}+{U}_{\bpi, \dHSIC}^{b}+{U}_{\bpi, \dHSIC}^{c_1}+{U}_{\bpi, \dHSIC}^{c_2}.
    $$
    Formally speaking, ${U}_{\bpi, \dHSIC}^{a}$ is defined as 
    \begingroup\small
    \begin{align*}
        &{U}_{\bpi, \dHSIC}^{a}:=\frac{(n-2d-2)!}{n!}\sum_{\bi_{2d+2}^n}\sum_{\substack{S_1,S_2\subsetneqq[d]\\1\in S_1\cap S_2}}\prod_{j\in S_1\cap S_2}k^j(X^j_{\pi^j_{i_{2j+1}}},X^j_{\pi^j_{i_{2j+2}}})\cdot\\
        &\prod_{j\in S_1^c\cap S_2}\bigg\{k^j(X^j_{\pi^j_{i_{1}}},X^j_{\pi^j_{i_{2j+2}}})-k^j(X^j_{\pi^j_{i_{2j+1}}},X^j_{\pi^j_{i_{2j+2}}})
        \bigg\}\prod_{j\in S_1\cap S_2^c}\bigg\{k^j(X^j_{\pi^j_{i_{2j+1}}},X^j_{\pi^j_{i_{2}}})-k^j(X^j_{\pi^j_{i_{2j+1}}},X^j_{\pi^j_{i_{2j+2}}})\bigg\}\cdot
        \\
        &\prod_{j\in S_1^c\cap S_2^c}\bigg\{k^j(X^j_{\pi^j_{i_{1}}},X^j_{\pi^j_{i_{2}}})+k^j(X^j_{\pi^j_{i_{2j+1}}},X^j_{\pi^j_{i_{2j+2}}})-k^j(X^j_{\pi^j_{i_{1}}},X^j_{\pi^j_{i_{2j+2}}})-k^j(X^j_{\pi^j_{i_{2j+1}}},X^j_{\pi^j_{i_{2}}})
        \bigg\}\\
        &:= \frac{(n-2d-2)!}{n!}\sum_{\bi_{2d+2}^n}h_a(\bX_{\bpi_{i_1}},\dots, \bX_{\bpi_{i_{2d+2}}}),
    \end{align*}
    \endgroup
    where $\bX_{\bpi_{i_{j}}}=(X^1_{\pi^1_{i_j}},\dots, X^d_{\pi^d_{i_j}})$ for $j\in[2d+2]$. 
    ${U}_{\bpi, \dHSIC}^{b}$,  ${U}_{\bpi, \dHSIC}^{c_i}$, $h_b(\cdot)$ and $h_{c_i}(\cdot)$ for $i=1,2$ are defined similarly. 
    Let $M^j:=\{m_1^j,\dots,m^j_{\lfloor n/(d+1)\rfloor}\}$ be a $\lfloor n/(d+1)\rfloor$-tuple uniformly sampled without replacement from $[n]$, $j=1,\dots,d+1$. $M^j$ are split into disjoint subsets of $[n]$. We define another statistic $\tilde{U}_{\bpi, \dHSIC}^{a}$ as 
    \begin{align*}
        \tilde{U}_{\bpi, \dHSIC}^{a} 
        := \frac{(\lfloor n/(d+1)\rfloor-2)!}{\lfloor n/(d+1)\rfloor!} \sum_{ \bi_2^{\lfloor n/(d+1)\rfloor}} h_a(\bX_{\bpi_{m^{1}_{i_1}}},\bX_{\bpi_{m^{1}_{i_2}}},\bX_{\bpi_{m^{2}_{i_1}}},\bX_{\bpi_{m^{2}_{i_2}}},\dots,\bX_{\bpi_{m^{d+1}_{i_{1}}}}, \bX_{\bpi_{m^{d+1}_{i_{2}}}}).
    \end{align*}
    $\tilde{U}_{\bpi, \dHSIC}^{b}$ and $\tilde{U}_{\bpi, \dHSIC}^{c_i}$ are defined similarly. Taking expectation of $\tilde{U}_{\bpi, \dHSIC}^{a}$, $\tilde{U}_{\bpi, \dHSIC}^{b}$ and $\tilde{U}_{\bpi, \dHSIC}^{c_i}$ over $M^j$ yields that
    \begin{equation}\label{eq_03}
    \begin{aligned}
        U^a_{\bpi, \dHSIC}&=\mbE_{M}[\tilde{U}_{\bpi, \dHSIC}^{a}\mid\mcX_n,\bpi],\\
        U^b_{\bpi, \dHSIC}&=\mbE_{M}[\tilde{U}_{\bpi, \dHSIC}^{b}\mid\mcX_n,\bpi],\\
        U^{c_i}_{\bpi, \dHSIC}&=\mbE_{M}[\tilde{U}_{\bpi, \dHSIC}^{c_i}\mid\mcX_n,\bpi],
    \end{aligned}
    \end{equation}
    where $M$ denotes the summary of $M^1,\dots,M^{d+1}$. 

    Note that the distributions of $\tilde{U}_{\bpi, \dHSIC}^{a}$, $\tilde{U}_{\bpi, \dHSIC}^{b}$ and $\tilde{U}_{\bpi, \dHSIC}^{c_i}$ do not change even if we randomly switch the order between $X^1_{\pi^1_{m^{1}_{k}}}$ and $X^1_{\pi^1_{m^{2}_{k}}}$ for $k=1,\dots,\lfloor n/(d+1)\rfloor$. On the other hand, $h_b(\cdot)$ changes the sign if we switch the order between $X^1_{\pi^1_{m^{1}_{k}}}$ and $X^1_{\pi^1_{m^{2}_{k}}}$ once, while $h_a(\cdot)$ does not change. Therefore, we introduce i.i.d.\ Rademacher random variables $\varepsilon_1^1,\dots,\varepsilon_{\lfloor n/(d+1)\rfloor}^1$ to represent the sign changes. 
    Therefore, it follows that 
    \begin{align*}
        &\tilde{U}_{\bpi, \dHSIC}^{a} + \tilde{U}_{\bpi, \dHSIC}^{b}+ \tilde{U}_{\bpi, \dHSIC}^{c_1}+ \tilde{U}_{\bpi, \dHSIC}^{c_2}\\
        \overset{d}{=}\ & \tilde{U}_{\bpi, \dHSIC}^{a,\varepsilon} + \tilde{U}_{\bpi, \dHSIC}^{b,\varepsilon} + \tilde{U}_{\bpi, \dHSIC}^{c_1,\varepsilon} + \tilde{U}_{\bpi, \dHSIC}^{c_2,\varepsilon}\\
        :=\ & \tilde{U}_{\bpi, \dHSIC}^{a}+
        \frac{(\lfloor n/(d+1)\rfloor-2)!}{\lfloor n/(d+1)\rfloor!} \sum_{ \bi_2^{\lfloor n/(d+1)\rfloor}} \varepsilon_{i_1}^1\varepsilon_{i_2}^1h_b(\bX_{\bpi_{m^{1}_{i_1}}},\bX_{\bpi_{m^{1}_{i_2}}},\dots,\bX_{\bpi_{m^{d+1}_{i_{1}}}}, \bX_{\bpi_{m^{d+1}_{i_{2}}}})\\
        &+\frac{(\lfloor n/(d+1)\rfloor-2)!}{\lfloor n/(d+1)\rfloor!} \sum_{ \bi_2^{\lfloor n/(d+1)\rfloor}} \varepsilon_{i_1}^1h_{c_1}(\bX_{\bpi_{m^{1}_{i_1}}},\bX_{\bpi_{m^{1}_{i_2}}},\dots,\bX_{\bpi_{m^{d+1}_{i_{1}}}}, \bX_{\bpi_{m^{d+1}_{i_{2}}}})\\
        &+\frac{(\lfloor n/(d+1)\rfloor-2)!}{\lfloor n/(d+1)\rfloor!}\sum_{ \bi_2^{\lfloor n/(d+1)\rfloor}} \varepsilon_{i_2}^1h_{c_2}(\bX_{\bpi_{m^{1}_{i_1}}},\bX_{\bpi_{m^{1}_{i_2}}},\dots,\bX_{\bpi_{m^{d+1}_{i_{1}}}}, \bX_{\bpi_{m^{d+1}_{i_{2}}}}),
    \end{align*}
    where the notation $\overset{d}{=}$ represents equality in distribution. Now we find a new statistic represented by quadratic and linear forms of Rademacher random variables. This procedure can proceed by 
    switching the order between $X^j_{\pi^1_{m^{1}_{k}}}$ and $X^j_{\pi^1_{m^{2}_{k}}}$ for $k=1,\dots,\lfloor n/(d+1)\rfloor$ and $j\in[d]$, with corresponding Rademacher random variables $\varepsilon_1^j,\dots,\varepsilon_{\lfloor n/(d+1)\rfloor}^j$, denoted as $\bve^j$. Similar to the discussion above, the original statistic is equal in distribution to the sum of the U-statistics, where the kernel functions are the product of the original kernels and some Rademacher random variables. 

    \textbf{Step 2 (Bounding each term)} We next bound the $\sum_{\bi_{2d+2}^n}g_{\bpi}(S_1,S_2)$ in \eqref{eq_25} for a given $(S_1,S_2)$. We discuss two cases of the relationship between the $(S_1,S_2)$. 
    
    \textbf{Case 1.} When $S_1^c\cap S_2^c\neq\varnothing$, we focus on the minimum element in $S_1^c\cap S_2^c$. Without loss of generality, we suppose that $1\in S_1^c\cap S_2^c$. Recalling the expression in \eqref{eq_25}, we consider the kernel of the U-statistic using splitting data as
    $$
    g_{\bpi,M}(S_1,S_2;i_1,i_2):=g_{\bpi}(S_1,S_2;\bX_{\bpi_{m^{1}_{i_1}}},\bX_{\bpi_{m^{1}_{i_2}}},\bX_{\bpi_{m^{2}_{i_1}}},\bX_{\bpi_{m^{2}_{i_2}}},\dots,\bX_{\bpi_{m^{d+1}_{i_{1}}}}, \bX_{\bpi_{m^{d+1}_{i_{2}}}}),
    $$
    and define the kernel with Rademacher r.v.s as 
    \begin{align*}
         &\prod_{j\in S_1^c\cap S_2}\varepsilon_{i_1}^j\prod_{j\in S_1\cap S_2^c}\varepsilon_{i_2}^j\prod_{j\in S_1^c\cap S_2^c}\varepsilon_{i_1}^j\varepsilon_{i_2}^jg_{\bpi,M}(S_1,S_2;i_1,i_2)\\
         &= \varepsilon_{i_1}^1\varepsilon_{i_2}^1\prod_{j\in S_1^c\cap S_2}\varepsilon_{i_1}^j\prod_{j\in S_1\cap S_2^c}\varepsilon_{i_2}^j\prod_{j\in S_1^c\cap S_2^c\backslash\{1\}}\varepsilon_{i_1}^j\varepsilon_{i_2}^jg_{\bpi,M}(S_1,S_2;i_1,i_2) \\
         &=:\varepsilon_{i_1}^1\varepsilon_{i_2}^1g_{\bpi,M}^{a,\bve}(S_1,S_2;i_1,i_2;\bve^2,\dots,\bve^d).
    \end{align*}
    $g_{\bpi,M}(S_1,S_2;i_1,i_2)$ is the split version of $g_{\bpi}(S_1,S_2)$ with the splitting $M^1,\dots,M^{d+1}$, and it is similar to the form of $h_b(\bX_{\bpi_{m^{1}_{i_1}}},\bX_{\bpi_{m^{1}_{i_2}}},\dots,\bX_{\bpi_{m^{d+1}_{i_{1}}}})$. $g_{\bpi,M}^{a,\bve}(S_1,S_2;i_1,i_2;\bve^2,\dots,\bve^d)$ introduces corresponding Rademacher r.v.s, and may potentially depend on $\bve^2,\dots,\bve^d$ by product. 
    By the discussions in Step 1, it holds that
    \begin{equation}\label{eq_26}
        \begin{aligned}
            &\frac{(n-2d-2)!}{n!}\sum_{\bi_{2d+2}^n}g_{\bpi}(S_1,S_2)\\
            =\ &\mbE_M\bigg(\frac{(\lfloor n/(d+1)\rfloor-2)!}{\lfloor n/(d+1)\rfloor!} \sum_{ \bi_2^{\lfloor n/(d+1)\rfloor}} g_{\bpi,M}(S_1,S_2;i_1,i_2)\mid \mcX_n,\bpi\bigg)\\
            \overset{d}{=}\ &\mbE_M\bigg(\frac{(\lfloor n/(d+1)\rfloor-2)!}{\lfloor n/(d+1)\rfloor!} \sum_{ \bi_2^{\lfloor n/(d+1)\rfloor}} \varepsilon_{i_1}^1\varepsilon_{i_2}^1g_{\bpi,M}^{a,\bve}(S_1,S_2;i_1,i_2;\bve^2,\dots,\bve^d)\mid \mcX_n,\bpi\bigg).
        \end{aligned} 
        \end{equation} 
    We denote $a_{i_1,i_2}(S_1,S_2;\bpi,M)=g_{\bpi,M}^{a,\bve}(S_1,S_2;i_1,i_2;\bve^2,\dots,\bve^d)$ for $1\leq i_1\neq i_2\leq \lfloor n/(d+1)\rfloor$, and $a_{i_1,i_2}(S_1,S_2;\bpi,M)=0$ for $1\leq i_1= i_2\leq \lfloor n/(d+1)\rfloor$. Use $\bA_{\bpi,M}(S_1,S_2;\bve^2,\dots,\bve^d)$ to denote a $\lfloor n/(d+1)\rfloor\times\lfloor n/(d+1)\rfloor$ matrix whose elements are $a_{i_1,i_2}(S_1,S_2;\bpi,M)$.
    As the Rademacher r.v.s are bounded between $-1$ and $1$, we have 
    $$
    \|\bA_{\bpi,M}(S_1,S_2;\bve^2,\dots,\bve^d)\|\leq\|\bA_{\bpi,M}(S_1,S_2;\bve^2,\dots,\bve^d)\|_F\lesssim nK_0,
    $$
    where $\|\cdot\|_F$ and $\|\cdot\|$ denote the Frobenius norm and the operator norm, respectively. 
    We denote the quantity inside the conditional expectation in the last line of \eqref{eq_26} as $\tilde{U}_{\bpi,S_1,S_2}^{A,\varepsilon}$. 
    For a given $(S_1,S_2)$, \eqref{eq_26} is a quadratic form of $\bve^1$ conditional on $\bpi$, $M$ and $\bve^2,\dots,\bve^d$. The proof of Hanson-Wright inequality \citep{Rudelson2013HansonWrightIA} along with sub-Gaussian Rademacher r.v.s yields that 
    \begin{equation}\label{eq_04}
    \begin{aligned}
        &\mbE_{M,\bpi,\bve^1,\dots,\bve^d}\{\exp(\lambda\tilde{U}_{\bpi,S_1,S_2}^{A,\varepsilon})\mid \mcX_n\}\\
        =\ &\mbE_{M,\bpi,\bve^2,\dots,\bve^d}\{\mbE_{\bve^1}[\exp(\lambda\tilde{U}_{\bpi,S_1,S_2}^{A,\varepsilon})\mid \mcX_n,\bve^2,\dots,\bve^d]\}\\
        \leq\ &\mbE_{M,\bpi,\bve^2,\dots,\bve^d}\{\exp(C_1\lambda^2n^{-2}K_0^2)\mid \mcX_n,\bve^2,\dots,\bve^d\}\\
        =\ & \exp (C_1\lambda^2n^{-2}K_0^2),
    \end{aligned}
    \end{equation}
    for $0\leq \lambda\leq c_1/K_0$. 
    
    \textbf{Case 2.} When $S_1^c\cap S_2^c=\varnothing$, both $S_1^c\cap S_2$ and $S_1\cap S_2^c$ are non-empty. In fact, $S_1^c\cap S_2\neq\varnothing$ follows  $(S_1^c\cap S_2^c)\cup(S_1^c\cap S_2)= S_1^c\neq\varnothing$ as $S_1\subsetneqq[d]$ and  $S_1\cap S_2^c\neq\varnothing$ holds similarly. We next focus on the minimum elements in $S_1^c\cap S_2$ and $S_1\cap S_2^c$. 
    Without loss of generality, we suppose that $1\in S_1^c\cap S_2$ and $2\in S_1\cap S_2^c$. Thus the kernel with Rademacher r.v.s is
    \begin{align*}
         &\quad\prod_{j\in S_1^c\cap S_2}\varepsilon_{i_1}^j\prod_{j\in S_1\cap S_2^c}\varepsilon_{i_2}^j\prod_{j\in S_1^c\cap S_2^c}\varepsilon_{i_1}^j\varepsilon_{i_2}^jg_{\bpi,M}(S_1,S_2;i_1,i_2)\\
         &= \varepsilon_{i_1}^1\varepsilon_{i_2}^2\prod_{j\in S_1^c\cap S_2\backslash\{1\}}\varepsilon_{i_1}^j\prod_{j\in S_1\cap S_2^c\backslash\{2\}}\varepsilon_{i_2}^j\prod_{j\in S_1^c\cap S_2^c}\varepsilon_{i_1}^j\varepsilon_{i_2}^jg_{\bpi,M}(S_1,S_2;i_1,i_2) \\
         &=: \varepsilon_{i_1}^1\varepsilon_{i_2}^2g_{\bpi,M}^{b,\bve}(S_1,S_2;i_1,i_2;\bve^3,\dots,\bve^d),
    \end{align*}
    and then the following equation holds similar to \eqref{eq_26}. 
        \begin{equation}\label{eq_27}
        \begin{aligned}
            &\frac{(n-2d-2)!}{n!}\sum_{\bi_{2d+2}^n}g_{\bpi}(S_1,S_2)\\
            =\ &\mbE_M\bigg(\frac{(\lfloor n/(d+1)\rfloor-2)!}{\lfloor n/(d+1)\rfloor!} \sum_{ \bi_2^{\lfloor n/(d+1)\rfloor}} g_{\bpi,M}(S_1,S_2;i_1,i_2)\mid \mcX_n,\bpi\bigg)\\
            \overset{d}{=}\ &\mbE_M\bigg(\frac{(\lfloor n/(d+1)\rfloor-2)!}{\lfloor n/(d+1)\rfloor!} \sum_{ \bi_2^{\lfloor n/(d+1)\rfloor}} \varepsilon_{i_1}^1\varepsilon_{i_2}^2g_{\bpi,M}^{b,\bve}(S_1,S_2;i_1,i_2;\bve^3,\dots,\bve^d)\mid \mcX_n,\bpi\bigg).
        \end{aligned} 
        \end{equation}
    Define $\bB_{\bpi,M}(S_1,S_2;\bve^3,\dots,\bve^d)$ to be a $\lfloor n/(d+1)\rfloor\times\lfloor n/(d+1)\rfloor$ matrix like $\bA_{\bpi,M}(\cdot)$. Its elements are associated with the split $M^1,\dots,M^{d+1}$, and may potentially depend on other groups of Rademacher r.v.s $\bve^3,\dots,\bve^d$ by product. It follows that  
    $$
    \|\bB_{\bpi,M}(S_1,S_2;\bve^3,\dots,\bve^d)\|\leq\|\bB_{\bpi,M}(S_1,S_2;\bve^3,\dots,\bve^d)\|_F\lesssim nK_0.
    $$
    We denote the quantity inside the conditional expectation in the last line of \eqref{eq_27} by $\tilde{U}_{\bpi,S_1,S_2}^{B,\varepsilon}$. 
    For a given $(S_1,S_2)$, \eqref{eq_27} is a quadratic form of $\bve^1$ and $\bve^2$ conditional on $\bpi$, $M$ and $\bve^3,\dots,\bve^d$. Similar proof leads to 
    \begin{equation}\label{eq_28}
    \begin{aligned}
        &\mbE_{M,\bpi,\bve^1,\dots,\bve^d}\{\exp(\lambda\tilde{U}_{\bpi,S_1,S_2}^{B,\varepsilon})\mid \mcX_n\}\\
        =\ &\mbE_{M,\bpi,\bve^3,\dots,\bve^d}\{\mbE_{\bve^1,\bve^2}[\exp(\lambda\tilde{U}_{\bpi,S_1,S_2}^{B,\varepsilon})\mid \mcX_n,\bve^3,\dots,\bve^d]\}\\
        \leq\ &\mbE_{M,\bpi,\bve^3,\dots,\bve^d}\{\exp(C_2\lambda^2n^{-2}K_0^2)\mid \mcX_n,\bve^3,\dots,\bve^d\}\\
        =\ &\exp(C_2\lambda^2n^{-2}K_0^2)
    \end{aligned}
    \end{equation}
    for $0\leq \lambda\leq c_2/K_0$. As the bounds are the same for $\tilde{U}_{\bpi,S_1,S_2}^{A,\varepsilon}$ and $\tilde{U}_{\bpi,S_1,S_2}^{B,\varepsilon}$, we unify them as $\tilde{U}_{\bpi,S_1,S_2}^{\varepsilon}$, and which type it is depends on $S_1,S_2$ in the subscripts. 
    
    \textbf{Step 3 (Combining bounds)} By noting that $(S_1,S_2)$ pairs are finite as $d$ is fixed, we combine the inequalities together. 
    It holds that
    \begin{align*}
        &\mathbb{P}_{\bpi} ( U_{\bpi, \dHSIC} \geq t \mid \mathcal{X}_n) \\
        =\ &\mathbb{P}_{\bpi} \bigg( \frac{(n-2d-2)!}{n!}\sum_{S_1\subsetneqq[d]}\sum_{S_2\subsetneqq[d]}\sum_{\bi_{2d+2}^n}g_{\bpi}(S_1,S_2;i_1,i_2) \geq t \mid \mathcal{X}_n \bigg) \\
        \overset{(i)}{\leq}\ & e^{-\lambda t}\mbE_{\bpi}\bigg[\exp\bigg(\lambda\frac{(n-2d-2)!}{n!}\sum_{S_1\subsetneqq[d]}\sum_{S_2\subsetneqq[d]}\sum_{\bi_{2d+2}^n}g_{\bpi}(S_1,S_2;i_1,i_2)\bigg)\mid \mcX_n\bigg] \\
        \overset{(ii)}{\leq}\ & e^{-\lambda t}\mbE_{M,\bpi,\bve^1,\dots,\bve^d}\bigg[\exp\bigg(\lambda\sum_{S_1\subsetneqq[d]}\sum_{S_2\subsetneqq[d]}\tilde{U}_{\bpi,S_1,S_2}^{\varepsilon}\bigg)\mid \mcX_n\bigg]\\
        \overset{(iii)}{\leq}\ & e^{-\lambda t}\prod_{S_1\subsetneqq[d]}\prod_{S_2\subsetneqq[d]}\mbE_{M,\bpi,\bve^1,\dots,\bve^d}[\exp(\lambda C_d\tilde{U}_{\bpi,S_1,S_2}^{\varepsilon})\mid \mcX_n]^{1/C_d}\\
        \overset{(iv)}{\leq}\ & e^{-\lambda t}\prod_{S_1\subsetneqq[d]}\prod_{S_2\subsetneqq[d]}\exp(C_3\lambda^2n^{-2}K_0^2)^{1/C_d}\\
        =\ & \exp(-\lambda t+C_4\lambda^2n^{-2}K_0^2), 
    \end{align*}
    provided that $0\leq \lambda\leq c_0/K_0$, where Step~$(i)$ uses the Chernoff bound for any $\lambda\geq0$. Step~$(ii)$ follows from Jensen's inequality along with \eqref{eq_26},  \eqref{eq_27}, and noting the summation still equals in distribution after taking conditional expectation. 
    Step~$(iii)$ follows by H{\"o}lder's inequality with constant $C_d$ associated with $d$ such that $\sum_{S_1\subsetneqq[d]}\sum_{S_2\subsetneqq[d]}C_d=1$. 
    Step~$(iv)$ holds by \eqref{eq_04} and \eqref{eq_28}. Optimizing over $\lambda$, we have the desired 
    $$
    \mathbb{P}_{\bpi} ( U_{\bpi, \dHSIC} \geq t \mid \mathcal{X}_n) \leq \exp \bigg\{ -C \min \bigg( \frac{n^2t^2}{K_0^2}, \frac{nt}{K_0} \bigg) \bigg\}. 
    $$
\end{proof}

Based on Lemma~\ref{le_u_dhsic_con_per}, we formally state the concentration inequality for permuted dHSIC and prove it by bounding the discrepancy between the V- and U-statistics.

\begin{lemma}[Concentration Inequality for Permuted dHSIC]\label{le_dhsic_con_per}
Assume that the kernels $k^j$  are bounded as $0\leq k^j(x,x')\leq K^j$ for all $x,x'\in \mbX^j$ and $j\in[d]$. For any $\alpha \in (0,1)$,
    $$
\mathbb{P}_{\bpi} \bigg( \hat{\dHSIC}(\mathcal{X}_n^{\bpi}) \geq C \sqrt{\frac{\prod K^j}{n}} 
\max \bigg\{ \log^{1/4}\bigg(\frac{1}{\alpha}\bigg), \log^{1/2}\bigg(\frac{1}{\alpha}\bigg), 1 \bigg\} \mid \mcX_n \bigg) 
\leq \alpha,
$$
where $C$ is some positive constant.
\end{lemma}
\begin{proof}
    Under the assumptions of Lemma~\ref{le_u_dhsic_con_per}, for all $t>0$
\begin{equation}\label{eq_05}
   \mathbb{P}_{\bpi} ( U_{\bpi, \dHSIC} \geq t \mid \mathcal{X}_n ) 
\leq \exp \bigg\{ -C \min \bigg( \frac{n^2t^2}{K_0^2}, \frac{nt}{K_0} \bigg) \bigg\}. 
\end{equation}
On the other hand, Lemma~\ref{le_sq_dhsic} yields that the squared dHSIC can be written as
\begin{align*}
    \widehat{\dHSIC}^2(\mcX_n^{\bpi})=& \frac{1}{n^2} \sum_{(i,j) \in \bj_2^n} \prod_{l=1}^d
k^l(X_{\pi_i^l}^l, X_{\pi_j^l}^l)  
+ \frac{1}{n^{2d}} \sum_{(i_1, j_1,\dots, i_d, j_d) \in \bj_{2d}^n} \prod_{l=1}^d
k^l(X_{\pi_{i_l}^l}^l, X_{\pi_{j_l}^l}^l) \nonumber\\
&- \frac{2}{n^{d+1}} \sum_{(i, j_1, \dots, j_d) \in \bj_{d+1}^n} \prod_{l=1}^d
k^l(X_{\pi_{i}^l}^l, X_{\pi_{j_l}^l}^l).
\end{align*}
Then the difference between the U-statistic and the squared V-statistic can  be bounded as
\begin{align*}
    & |U_{\bpi, \dHSIC}-\widehat{\dHSIC}^2(\mcX_n^{\bpi})|\\
    \leq & \bigg|\frac{(n-2)!}{n!} \sum_{(i,j) \in \bi_2^n} \prod_{l=1}^d k^l(X_{\pi_i^l}^l, X_{\pi_j^l}^l)- \frac{1}{n^2} \sum_{(i,j) \in \bj_2^n} \prod_{l=1}^d k^l(X_{\pi_i^l}^l, X_{\pi_j^l}^l) 
    \bigg|\\
    + & \bigg|\frac{(n-2d)!}{n!} \sum_{(i_1, j_1,\dots, i_d, j_d) \in \bi_{2d}^n} \prod_{l=1}^d k^l(X_{\pi_{i_l}^l}^l, X_{\pi_{j_l}^l}^l)- \frac{1}{n^{2d}} \sum_{(i_1, j_1,\dots, i_d, j_d) \in \bj_{2d}^n} \prod_{l=1}^d k^l(X_{\pi_{i_l}^l}^l, X_{\pi_{j_l}^l}^l)
    \bigg|\\
    + & \bigg|\frac{2(n-d-1)!}{n!} \sum_{(i, j_1, \dots, j_d) \in \bi_{d+1}^n} \prod_{l=1}^d k^l(X_{\pi_{i}^l}^l, X_{\pi_{j_l}^l}^l)- \frac{2}{n^{d+1}} \sum_{(i, j_1, \dots, j_d) \in \bj_{d+1}^n} \prod_{l=1}^d k^l(X_{\pi_{i}^l}^l, X_{\pi_{j_l}^l}^l)
    \bigg|\\
    \leq & C_1 \frac{K_0}{n}.
\end{align*}
Using the upper bound, we link the concentration inequality for the V- and U-statistics as
\begin{align*} 
    \mathbb{P}_{\bpi} ( \hat{\dHSIC}(\mathcal{X}_n^{\bpi}) \geq t\mid \mcX_n ) &= \mathbb{P}_{\bpi} ( \hat{\dHSIC}^2(\mathcal{X}_n^{\bpi}) \geq t^2\mid \mcX_n )\\
    &\leq  \mathbb{P}_{\bpi} ( |\hat{\dHSIC}^2(\mathcal{X}_n^{\bpi})-U_{\bpi, \dHSIC}| + U_{\bpi, \dHSIC} \geq t^2\mid \mcX_n )\\
    &\leq  \mathbb{P}_{\bpi} \bigg( U_{\bpi, \dHSIC} \geq t^2-C_1 \frac{K_0}{n}\mid \mcX_n \bigg)\\
    &\leq  \exp \bigg\{ -C_2 \min \bigg( \frac{n^2(t^2-C_1{K_0}/{n})^2}{K_0^2}, \frac{n(t^2-C_1{K_0}/{n})}{K_0} \bigg) \bigg\}.
\end{align*}
The last inequality is based on \eqref{eq_05}. The proof is completed by setting the last exponential bound equal to $\alpha$ and solving for $t$. 
\end{proof}

\begin{lemma}[Exponential Inequality for the Empirical dHSIC]\label{le_ineq_e_dhsic}
Assume that the kernels $k^j$  are bounded as $0\leq k^j(x,x')\leq K^j$ for all $x,x'\in \mbX^j$ and $j\in[d]$. For any $t\geq0$,
    $$
    \mbP \bigg\{ 
    | \widehat{\dHSIC}(\mathcal{X}_n) - \dHSIC_{\bk}(P_{X_1,\dots,X_d}) | 
    \geq C_1 \sqrt{\frac{K_0}{n}} + t 
    \bigg\} 
    \leq 2 \exp \bigg( -\frac{C_2 t^2 n}{K_0} \bigg),
    $$
where $C_1, C_2$ are positive constants and $K_0:=\prod K^j$.
\end{lemma}
\begin{proof}
    With a little abuse of notations, $C_1, C_2,\dots$ refer to universal constant only associated with the upper bounds for the kernels. Proposition~\ref{sen_dhsic} indicates that the empirical dHSIC has global sensitivity at most $C_3\sqrt{K_0}/n$. By McDiarmid's inequality \citep[Theorem~29]{Gretton2012AKT}, we have 
    \begin{equation}\label{eq_07}
        \mbP \{ | \widehat{\dHSIC}(\mathcal{X}_n) - \mbE[\widehat{\dHSIC}(\mathcal{X}_n)] | \geq t \} \leq 2 \exp \bigg( -\frac{C_2 t^2 n}{K_0} \bigg) \quad \text{for all } t \geq 0.
    \end{equation}
    We complete the proof by bounding the difference between the expectation of dHSIC and the population dHSIC. 
    \begin{align*}
    &| \mbE[\widehat{\dHSIC}(\mathcal{X}_n)] - \dHSIC_{\bk}(P_{X^1,\dots, X^d}) | \\
    =\ & \bigg| \mathbb{E} \bigg[ \sup_{f \in \mathcal{F}_{\bk}} \bigg\{\frac{1}{n} \sum_{i=1}^n f(X_i^1,\dots,X_i^d) - \frac{1}{n^d} \sum_{\bj^n_d} f(X_{i_1}^1,\dots,X_{i_d}^d) \bigg\} \\
    &- \sup_{f\in\mcF_{\bk}}\bigg\{\mbE_{P_{X^1,\dots, X^d}}[f(X^1,\dots, X^d)]-\mbE_{\otimes_{j=1}^d P_{X^j}}[f(X^1,\dots, X^d)]
    \bigg\} \bigg] \bigg| \\
    \overset{(i)}{\leq}\ & \mathbb{E}\sup_{f \in \mathcal{F}_{\bk}}\bigg|\bigg\{\frac{1}{n} \sum_{i=1}^n f(X_i^1,\dots,X_i^d) - \frac{1}{n^d} \sum_{\bj^n_d} f(X_{i_1}^1,\dots,X_{i_d}^d) \bigg\} \\
    & - \bigg\{\mbE_{P_{X^1,\dots, X^d}}[f(X^1,\dots, X^d)]-\mbE_{\otimes_{j=1}^d P_{X^j}}[f(X^1,\dots, X^d)]
    \bigg\}  \bigg| \\
    \overset{(ii)}{\leq}\ & \mathbb{E}\sup_{f \in \mathcal{F}_{\bk}}\bigg|\frac{1}{n} \sum_{i=1}^n f(X_i^1,\dots,X_i^d) - \mbE_{P_{X^1,\dots, X^d}}[f(X^1,\dots, X^d)]\bigg|\\
    & + \mathbb{E}\sup_{f \in \mathcal{F}_{\bk}}\bigg|
    \frac{1}{n^d} \sum_{\bj^n_d} f(X_{i_1}^1,\dots,X_{i_d}^d)-\mbE_{\otimes_{j=1}^d P_{X^j}}[f(X^1,\dots, X^d)]  \bigg| \\
    :=\ & \text{(I)} + \text{(II)},
\end{align*}
where Step~$(i)$ holds by Jensen's inequality and the triangle inequality, and Step~$(ii)$ uses the triangle inequality. Let $\{(\tilde X^1_i,\dots,\tilde X^d_i)\}$ be i.i.d.\ copies of $(X^1,\dots, X^d)$ and $\{\varepsilon_i\}_{i=1}^n$ be i.i.d.\ Rademacher random variables. By Jensen's inequality, we have
\begin{align}\label{eq_06}
    \text{(I)}&\leq \mathbb{E}\sup_{f \in \mathcal{F}_{\bk}}\bigg|\frac{1}{n} \sum_{i=1}^n f(X_i^1,\dots,X_i^d) - \frac{1}{n} \sum_{i=1}^n f(\tilde X_i^1,\dots,\tilde X_i^d)\bigg|\nonumber\\
    &\leq \mathbb{E}\sup_{f \in \mathcal{F}_{\bk}}\bigg|\frac{1}{n} \sum_{i=1}^n \varepsilon_i\{f(X_i^1,\dots,X_i^d) - f(\tilde X_i^1,\dots,\tilde X_i^d)\} \bigg|\nonumber\\
    &\leq 2\mathbb{E}\sup_{f \in \mathcal{F}_{\bk}}\bigg|\frac{1}{n} \sum_{i=1}^n \varepsilon_if(X_i^1,\dots,X_i^d)\bigg|\leq 2\sqrt{K_0/n},
\end{align}
where the last inequality is a direct application of \citet[Lemma~22]{Bartlett2003RademacherAG}. 
It is sufficient to verify that $\text{(II)}\lesssim\sqrt{K_0/n}$ to complete the proof.
\begin{align*}
    \text{(II)}\leq\ &\frac{1}{n}\sum_{j=1}^n\mathbb{E}\sup_{f \in \mathcal{F}_{\bk}}\bigg|
    \frac{1}{n^{d-1}} \sum_{\bj^n_{d-1}} f(X_{i_1}^1,\dots,X_{i_{d-1}}^{d-1},X_{j}^{d})-\mbE_{\otimes_{j=1}^d P_{X^j}}[f(X^1,\dots, X^d)]  \bigg|\\
    =\ &\mathbb{E}\sup_{f \in \mathcal{F}_{\bk}}\bigg|
    \frac{1}{n^{d-1}}\bigg( \sum_{\bj^{n-1}_{d-1}} f(X_{i_1}^1,\dots,X_{i_{d-1}}^{d-1},X_{n}^{d})+\sum_{\bj^{n}_{d-1}\backslash\bj^{n-1}_{d-1}} f(X_{i_1}^1,\dots,X_{i_{d-1}}^{d-1},X_{n}^{d})\bigg)\\
    &-\mbE_{\otimes_{j=1}^d P_{X^j}}[f(X^1,\dots, X^d)]  \bigg|\\
    \leq\ &\mathbb{E}\sup_{f \in \mathcal{F}_{\bk}}\bigg|
    \frac{1}{n^{d-1}}\sum_{\bj^{n-1}_{d-1}} f(X_{i_1}^1,\dots,X_{i_{d-1}}^{d-1},X_{n}^{d})-\mbE_{\otimes_{j=1}^d P_{X^j}}[f(X^1,\dots, X^d)]  \bigg|\\
    +&\mathbb{E}\sup_{f \in \mathcal{F}_{\bk}}\bigg|\frac{1}{n^{d-1}}\sum_{\bj^{n}_{d-1}\backslash\bj^{n-1}_{d-1}} f(X_{i_1}^1,\dots,X_{i_{d-1}}^{d-1},X_{n}^{d})\bigg|\\
    \leq\ &\mathbb{E}\sup_{f \in \mathcal{F}_{\bk}}\bigg|
    \frac{1}{(n-1)^{d-1}}\sum_{\bj^{n-1}_{d-1}} f(X_{i_1}^1,\dots,X_{i_{d-1}}^{d-1},X_{n}^{d})-\mbE_{\otimes_{j=1}^d P_{X^j}}[f(X^1,\dots, X^d)]  \bigg|\\
    +& O(n^{-d})\sum_{\bj^{n-1}_{d-1}}\mathbb{E}\sup_{f \in \mathcal{F}_{\bk}}|f(X_{i_1}^1,\dots,X_{i_{d-1}}^{d-1},X_{n}^{d})|
    +\frac{1}{n^{d-1}}\sum_{\bj^{n}_{d-1}\backslash\bj^{n-1}_{d-1}}\mathbb{E}\sup_{f \in \mathcal{F}_{\bk}}| f(X_{i_1}^1,\dots,X_{i_{d-1}}^{d-1},X_{n}^{d})|\\
    \leq\ &\mathbb{E}\sup_{f \in \mathcal{F}_{\bk}}\bigg|
    \frac{1}{(n-1)^{d-1}}\sum_{\bj^{n-1}_{d-1}} f(X_{i_1}^1,\dots,X_{i_{d-1}}^{d-1},X_{n}^{d})-\mbE_{\otimes_{j=1}^d P_{X^j}}[f(X^1,\dots, X^d)]  \bigg|+C_4\sqrt{K_0}/n.
\end{align*}
We take the an expectation conditional on $X^d$ and $X^d_n$. Then the problem reduces to $n-1$ observations among $d-1$ groups. By mathematical induction principle, the calculations to the expectation come down to the expectation for a single variable $X^1$, which is no more than $2\sqrt{K_0/n}$ by \eqref{eq_06}. Therefore, we have
$$
    | \mathbb{E}[\widehat{\dHSIC}(\mathcal{X}_n)] - \dHSIC_{\bk}(P_{X^1,\dots, X^d}) |\leq C_1\sqrt{K_0/n}. 
$$
The proof is completed by combining the above inequality and \eqref{eq_07}. 
\end{proof}

\end{appendices}

\end{document}